\documentclass[10pt]{article}
\usepackage[lmargin=1in,rmargin=1in,tmargin=1in,bmargin=1in]{geometry}
\usepackage{amsmath,amsthm,amsfonts,amssymb,comment,slashed,mathtools}
\usepackage[utf8]{inputenc}
\usepackage[T1]{fontenc}  
\usepackage[english]{babel}
\usepackage{xcolor}
\usepackage{accents}

\usepackage{csquotes}
\usepackage{bm}
\usepackage[shortlabels]{enumitem}
\usepackage[affil-it]{authblk}

\usepackage[section]{placeins}
\usepackage{soul} 

\usepackage{bm}
\usepackage{mathrsfs, graphicx} 
\usepackage{stmaryrd}
\usepackage{epigraph}

\usepackage[backend=biber,style=alphabetic,giveninits=true,doi=false,isbn=false,url=false,sorting=nyt,maxbibnames=99,date=year]{biblatex}
\renewbibmacro{in:}{}
\usepackage{chngcntr}

\usepackage[%
bookmarks=true,
colorlinks=true,
linkcolor=black,
urlcolor=black,
citecolor=black,
plainpages=false,
pdfpagelabels,
final]{hyperref}
\usepackage[capitalise,nameinlink]{cleveref}

\numberwithin{equation}{section}

\theoremstyle{plain}
\newtheorem{thm}{Theorem}[section]
\newtheorem{prop}[thm]{Proposition}
\newtheorem{lem}[thm]{Lemma}
\newtheorem*{thm*}{Theorem}

\newtheorem*{conj*}{Conjecture}

\theoremstyle{definition}
\newtheorem{defn}[thm]{Definition}

\theoremstyle{remark}
\newtheorem{rk}[thm]{Remark}
\crefname{thm}{Theorem}{Theorems} 
\crefname{thmrough}{Theorem}{Theorems} 
\crefname{lem}{Lemma}{Lemmas} 
\crefname{conj}{Conjecture}{Conjectures}
\crefname{prop}{Proposition}{Propositions}

\renewcommand{\Bbb}{\mathbb}
\newcommand{\ve}{\varepsilon}
\newcommand{\les}{\lesssim}

\renewcommand{\k}{\mathbf{k}}

\usepackage{accents}

\makeatletter
\renewcommand{\paragraph}{%
  \@startsection{paragraph}{4}%
  {\z@}{1.25ex \@plus 1ex \@minus .2ex}{-1em}%
  {\normalfont\normalsize\bfseries}%
}
\makeatother

\begin{document}

\title{Semilinear wave equations on extremal Reissner--Nordstr\"om \\ black holes revisited}

\author[1]{Yannis Angelopoulos\thanks{yannis@caltech.edu}}
\author[2,3]{Ryan Unger\thanks{runger@berkeley.edu}}
\affil[1]{\small California Institute of Technology, Division of Physics, Mathematics, and Astronomy,

1200 East California Boulevard, Pasadena, CA 91125, United States of America \vskip.1pc \ 
}

\affil[2]{\small  Stanford University, Department of Mathematics, 
	
450 Jane Stanford Way, Building 380, Stanford, CA 94305, United States of America
 \vskip.1pc \  
	}
 \affil[3]{\small  University of California, Berkeley,  Department of Mathematics,
	
970 Evans Hall, Berkeley, CA 94720, United States of America \vskip.1pc \  
	}

\date{August 6, 2025}

\maketitle

\begin{abstract}
We revisit global existence and decay for small-data solutions of semilinear wave equations on extremal Reissner--Nordstr\"om black hole backgrounds satisfying the classical null condition, a problem which was previously addressed by the first author in joint work with Aretakis and Gajic \cite{AAG20}. In this paper, we develop a new approach based on propagating a significantly weaker set of estimates, which allows for a simpler and more streamlined proof. Our proof does not require tracking sharp estimates for the solution in the near-horizon region, which means that it is compatible with, but does not imply, the non-decay and growth hierarchy of derivatives of the solution along the event horizon expected from the Aretakis instability. In particular, this approach is in principle compatible with other settings where stronger horizon instabilities are expected, such as nonlinear charged scalar fields on extremal Reissner--Nordstr\"om, or nonlinear waves on extremal Kerr. We also sketch how our proof applies to semilinear problems on spacetimes settling down to extremal Reissner--Nordstr\"om, such as those constructed in our joint work with Kehle \cite{AKU24}. 
\end{abstract}

\thispagestyle{empty}
\tableofcontents


\section{Introduction} 

\subsection{Background and the main result}

In this paper, we consider systems of semilinear wave equations of the form
\begin{equation}\label{eq:wave}
    \Box_{g_M} \phi = \mathcal N(x,\phi,d\phi)
\end{equation}
on $(3+1)$-dimensional extremal Reissner--Nordstr\"om black hole backgrounds $(\mathcal M^{3+1},g_M)$, where $M>0$ is the mass of the black hole. Here $\phi:\mathcal M\to \Bbb R^N$ is a collection of scalar fields on which we impose no symmetry assumptions, $\Box_{g_M}$ is the covariant wave operator with respect to $g_M$, and the nonlinearity $\mathcal N(x,\phi,d\phi)$ satisfies a strong version of the null condition at the event horizon $\mathcal H^+$ and at null infinity $\mathcal I^+$. An explicit example of such a system is given by the celebrated \emph{wave map} system with domain $(\mathcal M,g_M)$ and target an arbitrary $N$-dimensional Riemannian manifold. The main motivation for studying \eqref{eq:wave} is to model nonlinear issues in the problem of nonlinear asymptotic stability of extremal Reissner--Nordstr\"om black holes as solutions to the Einstein--Maxwell equations without symmetry assumptions. We refer the reader to \cite{DHRT,Daf24,KU24,AKU24} for a formulation of the stability problem for extremal black holes.

The main difficulty in studying linear and nonlinear waves on extremal Reissner--Nordstr\"om is the absence of the stabilizing mechanism of the horizon redshift effect which is fundamental to our understanding of linear waves on subextremal black holes, such as the celebrated Schwarzschild solution \cite{dafermos2009red,dafermos2013lectures}. In a remarkable series of papers \cite{Aretakis-instability-1,Aretakis-instability-2,Aretakis-instability-3}, Aretakis showed that
ingoing null derivatives of solutions to the linear wave equation on extremal Reissner–Nordström generically do not decay on the event horizon, and higher derivatives may even grow polynomially in time. This horizon instability, which has come to be known as the \emph{Aretakis instability}, has significant ramifications for nonlinear problems. In a short and elegant work \cite{aretakis2013nonlinear}, Aretakis considered the semilinear equations
\begin{equation}
    \Box_{g_M}\phi = \chi\big(\phi^{2n} + (Y\phi)^{2n}\big),\label{eq:Aretakis}
\end{equation}
where $n\in\Bbb N$, $Y$ denotes the translation-invariant null derivative transverse to $\mathcal H^+$, which is equal to $\partial_r$ in ingoing Eddington--Finkelstein $(v,r)$ coordinates, and $\chi$ is a cutoff function supported near the horizon at $r=M$. For any $n$, these equations enjoy a standard small-data global existence and stability theory on subextremal backgrounds. However, Aretakis showed that for any $n$, generic solutions to \eqref{eq:Aretakis} blow up \emph{in finite affine time} along the event horizon in extremal Reissner--Nordstr\"om. 

The finite-time blowup mechanism for \eqref{eq:Aretakis} is of ODE type and is mainly caused by the interaction of the $(Y\phi)^{2n}$ term in the nonlinearity with the conservation law structure that causes the Aretakis instability in the linear case. This effect is reminiscent of the well-known Riccati-type blowup for the equation $\Box \phi = (\partial_t\phi)^2$ in Minkowski space \cite{john1979blow}. This example motivated the classical null condition \cite{klainerman1986null,Christodoulou-null-condition}, a structural assumption on quadratic nonlinearities which excludes the $ (\partial_t\phi)^2$ nonlinearity and allows one to prove global existence. In the subextremal (or Minkowski) setting, it is in fact only necessary to assume the null condition near null infinity in order to prove global existence. Aretakis' example \eqref{eq:Aretakis} indicates that some null structure condition for quadratic derivatives is also necessary \emph{at the event horizon} in the extremal case. Note moreover that finite-time blowup for \eqref{eq:Aretakis} occurs even when $n\ge 2$, which indicates that the most general form of the null condition in the extremal setting must also constrain the structure of higher order nonlinearities at the horizon. This is in sharp contrast to the situation at null infinity, where cubic (or higher) terms in $\phi$ or $d\phi$ are harmless if the quadratic terms satisfy the null condition. 

Under the assumption of spherically symmetric initial data, the first named author of the present paper proved global stability of the zero solution for scalar equations with nonlinearity $ f(x,\phi)g_{M}^{-1}(d\phi,d\phi)$ (a generalization of the standard $Q_0$ metric null form) in \cite{A16}. The Aretakis instability survives in this nonlinear setting: it is shown that $Y\phi$ generically does not decay along $\mathcal H^+$ (it remains approximately constant) and $Y^2\phi$ generically grows along $\mathcal H^+$. Nevertheless, appropriate energies of $\phi$ decay (up to a slight nonlinear loss) just as for the linear wave equation, which is why \cite{A16} is fundamentally a \emph{stability} result. In later joint work with Aretakis and Gajic \cite{AAG20}, the assumption on spherical symmetry was removed, at the expense of a significant increase in technical complexity: 

\begin{thm}[\cite{AAG20}]\label{thm:AAG20}
        Consider a scalar semilinear wave equation of the form \eqref{eq:wave} on an extremal Reissner--Nordstr\"om black hole spacetime $(\mathcal M^{3+1},g_M)$, with nonlinearity $\mathcal N(x,\phi,d\phi)= f(x,\phi)g_{M}^{-1}(d\phi,d\phi)$.  Let $\mathring\phi$ be smooth characteristic initial data for \eqref{eq:wave} posed along a bifurcate null hypersurface that spans $\mathcal H^+$ and $\mathcal I^+$. If $\mathring\phi$ is sufficiently small in a suitable weighted norm $\|\cdot\|_0$, then:
        \begin{enumerate}
            \item The solution $\phi$ to \eqref{eq:wave} with initial data $\mathring\phi$ exists globally in the domain of outer communication, up to and including the event horizon $\mathcal H^+$.  Suitable energies of $\phi$ (weaker than $\|\cdot\|_0$), as well as $\phi$ itself measured pointwise, decay polynomially in time. 
           
            \item The ingoing null derivative of the solution, $Y\phi$, is approximately constant along $\mathcal H^+$. For generic data $\mathring\phi$, the second ingoing null derivative of the solution, $Y^2\phi$, grows linearly in advanced time $v$ along $\mathcal H^+$.
        \end{enumerate}
\end{thm}

The point of view taken in \cite{A16,AAG20} is that parts 1.~and 2.~are essentially coupled and thus ought to be proved \emph{concurrently}, i.e., that one should design bootstrap assumptions that are strong enough to prove asymptotics of the solution along the event horizon (and hence point 2.), while being flexible enough to actually carry out the continuity argument proving point 1. Outside of symmetry, this demand greatly complicates the proof of stability: following \cite{angelopoulos2020late}, the proof of pointwise boundedness of $Y\phi$ at $\mathcal H^+$ requires commuting \eqref{eq:wave} with the vector field $Y$ near the horizon and with the vector field $r^2\partial_v$ in the far region, in order to extend the $(r-M)^{-p}$- and $r^p$-hierarchies (see already \cref{sec:linear-review}) past $p=2$. This commutation process generates many error terms in the energy estimates that require careful treatment and the initial data norm $\|\mathring\phi\|_0$ necessarily measures initial smallness of these commuted energies. 

\begin{rk}
    While \cref{thm:AAG20} was originally stated for scalar semilinear equations, the result immediately extends to the class of semilinear systems considered in the present paper. 
\end{rk}

Much more fundamentally, however, boundedness of $Y\phi$ at $\mathcal H^+$ is not actually compatible with the expected behavior of the linear wave equation in the case of a charged scalar field on extremal Reissner--Nordstr\"om \cite{Gaj25} or for non-axisymmetric solutions on extremal Kerr \cite{Gajic23} (see also \cite{cgz-exkerr}). With these more difficult problems in mind for the future, it is therefore desirable to obtain a proof of global existence for small-data solutions of \eqref{eq:wave} on extremal Reissner--Nordstr\"om that is compatible with the expected growth rate $|Y\phi|\gtrsim v^{1/2}$ along $\mathcal H^+$ in the charged scalar field and Kerr cases. 

We accomplish this in the present paper by propagating a much weaker hierarchy of estimates than in \cite{AAG20}. Since our scheme avoids commuting with $Y$ and $r^2\partial_v$, as a bonus this greatly simplifies the argument for global existence. 

\begin{thm}\label{thm:rough}
Global existence for semilinear systems on extremal Reissner--Nordstr\"om, i.e., part 1.~of \cref{thm:AAG20}, holds under a smallness assumption on $\mathring\phi$ with respect to a norm $\|\cdot\|_\star$ that only involves commutation with time-translation and rotation Killing fields and only controls $(r-M)^{-p}$- and $r^p$-energies up to $p=2-\delta$, where $\delta>0$ is a small auxiliary parameter, and hence is strictly weaker than $\|\cdot\|_0$.  
\end{thm}

For the precise assumptions we make on the nonlinearity $\mathcal N$, see already \cref{sec:str-intro,sec:nonlinearity}, and for the precise statement of the theorem, see already \cref{sec:statement}. The norm $\|\cdot\|_\star$ is defined in \cref{sec:loc}. For comments on semilinear equations \emph{not} covered by this theorem, see already \cref{sec:other-eqns}.

We stress that the importance of this statement is not that global existence holds for a larger class of data, but that the \emph{proof} of global existence can be carried out using much weaker energy and pointwise estimates than those needed to prove the Aretakis instability. In particular, given only a smallness assumption on $\|\cdot\|_\star$, the best estimate for $Y\phi$ along $\mathcal H^+$ one can hope to prove is $|Y\phi|\les  v^{1/2+\delta/2} \|\mathring\phi\|_\star$. Therefore, the proof of \cref{thm:rough} fulfills the requirement of being compatible (in principle) with the expected stronger instabilities in the charged scalar field and Kerr cases. Of course, once one has used \cref{thm:rough} to prove global existence, one can then revisit the solution and give a simpler proof of part 2.~of \cref{thm:AAG20} under a suitable smallness assumption on $\|\cdot\|_0$ by following the broad strokes of \cite{AAG20}.

This theorem may also be thought of as an instantiation of the principle (see for instance \cite{DL17}) that in certain problems exhibiting both stability and instability phenomena, the stability phenomenon can often be understood independently, i.e., stability can (and should!) be proved using estimates that are \emph{consistent with}, but \emph{strictly weaker than}, those needed to prove instability. 

\subsection{Overview of the proof} 

The proof of the main theorem is a standard bootstrap argument involving certain pointwise, energy, and spacetime integral assumptions. In this section, we outline the main ideas, starting with the relevant theory for the linear wave equation on extremal Reissner--Nordstr\"om. 

\subsubsection{Review of linear waves on extremal Reissner--Nordstr\"om}\label{sec:linear-review}

Here we briefly review some basic aspects of the theory of solutions to the linear wave equation
\begin{equation}
    \Box_{g_M}\phi=0\label{eq:linear-wave}
\end{equation}
on extremal Reissner--Nordstr\"om. Aretakis initiated the study of this problem in \cite{Aretakis-instability-1,Aretakis-instability-2} but we will make use of technical advances made by the first-named author, Aretakis, and Gajic in \cite{angelopoulos2020late}. For a brief review of the geometry of extremal Reissner--Nordstr\"om, we refer the reader to \cref{sec:ERN} of the present paper, \cite[Section 2.2]{AKU24}, \cite[Section 2]{Aretakis-instability-1}, and the appendix of \cite{stefanos-ern_full}. 

The general strategy to prove energy decay statements for waves on extremal Reissner--Nordstr\"om consists of, as in the subextremal case, deriving a hierarchy of weighted energy boundedness inequalities and time-integrated energy decay estimates. For a scalar solution $\phi$ to \eqref{eq:linear-wave}, the hierarchy takes the form
\begin{multline}\label{eq:boundedness-intro-1} 
    \int_{C(\tau_2)}r^p(\partial_v\psi)^2\, d\omega dv+ \int_{\underline C(\tau_2)} (r-M)^{-p} (\partial_u\psi)^2\,d\omega du+\cdots\\ \les \int_{C(\tau_1)}r^p(\partial_v\psi)^2\, d\omega dv+ \int_{\underline C(\tau_1)} (r-M)^{-p} (\partial_u\psi)^2\,d\omega du+\cdots,
\end{multline} \vspace{-5mm}
\begin{align}
\label{eq:rp-intro}    \int_{\tau_1}^{\tau_2}\int_{C(\tau)} r^{p-1}(\partial_v\psi)^2\, d\omega dvd\tau +\cdots&\les \int_{C(\tau_1)}r^p(\partial_v\psi)^2\, d\omega dv + \cdots,\\
\label{eq:r-Mp-intro}    \int_{\tau_1}^{\tau_2}\int_{\underline C(\tau)} \chi_\mathrm{trap}(r-M)^{-p+1}(\partial_u\psi)^2\,d\omega dud\tau +\cdots&\les\int_{\underline C(\tau_1)}  (r-M)^{-p} (\partial_u\psi)^2\,d\omega du + \cdots,
\end{align}
where $(u,v)$ denote Eddington--Finkelstein double null coordinates on the domain of outer communication, $\tau$ is proper time along a timelike hypersurface $\Gamma$ with constant area-radius, $\tau_1\le\tau_2$, $p\in[0,2]$, $d\omega$ denotes the standard area element of the unit sphere $S^2$, $\cdots$ denotes terms which either contain angular derivatives or which are lower in the $p$-hierarchy, the foliations $C(\tau)$ and $\underline C(\tau)$ are defined via Penrose diagram in \cref{fig:ERN-intro} above, $\psi\doteq r\phi$, and $\chi_\mathrm{trap}$ is a cutoff function that vanishes in a neighborhood of the \emph{photon sphere} $r=2M$. 

 \begin{figure}
\centering{
\def\svgwidth{11pc}
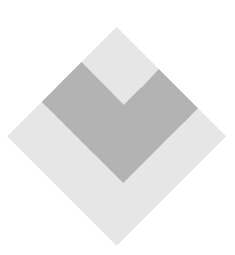}
\caption{A Penrose diagram of extremal Reissner--Nordstr\"om depicting the foliations $C(\tau)$ and $\underline C(\tau)$ used in the estimates \eqref{eq:boundedness-intro-1}--\eqref{eq:r-Mp-intro}. The region of integration in \eqref{eq:rp-intro} and \eqref{eq:r-Mp-intro} is shaded darker. We have also indicated here the photon sphere, located at $r=2M$, where certain derivatives in the Morawetz estimate \eqref{eq:r-Mp-intro} degenerate.}
\label{fig:ERN-intro}
\end{figure}

The inequality \eqref{eq:rp-intro} is the celebrated \emph{$r^p$-weighted estimate} of Dafermos and Rodnianski \cite{dafermos2010new} and relies only on the asymptotic flatness of the metric. The estimate \eqref{eq:r-Mp-intro} is specific to the extremal case and can be thought of as a ``horizon analogue'' of the $r^p$ estimates at $\mathcal I^+$: The event horizon $\mathcal H^+$ of extremal Reissner--Nordstr\"om has $r=M$ and hence $r-M$ is a degenerate weight towards $\mathcal H^+$, dual to $r^{-1}$ at $\mathcal I^+$. The estimate \eqref{eq:r-Mp-intro} states that the time integral of the $(p-1)$-weighted horizon flux is bounded by the initial value of the $p$-weighted horizon flux. This horizon hierarchy---special cases of which were introduced in \cite{Aretakis-instability-1} and in full generality was derived in \cite{angelopoulos2020late}---replaces the fundamental redshift estimate from \cite{dafermos2009red,dafermos2013lectures} in the subextremal case. Note that the $p=0$ horizon flux is equivalent to the $T$-energy, where $T$ is the time-translation Killing field. The degeneration of \eqref{eq:r-Mp-intro} at the photon sphere due to trapping is familiar from the subextremal case and we refer to \cite{dafermos2009red,Aretakis-instability-1,dafermos2013lectures} for discussion.

Using the pigeonhole principle as in \cite{dafermos2010new}, \eqref{eq:boundedness-intro-1}--\eqref{eq:r-Mp-intro} can be used to prove the energy decay estimate
\begin{equation}\label{eq:ED}
    \int_{C(\tau)} r^p(\partial_v\psi)^2\, d\omega dv+\int_{\underline C(\tau)} (r-M)^{-p} (\partial_u\psi)^2\,d\omega du \le C_\star\tau^{-2+p},
\end{equation} for every $p\in[0,2]$ and $\tau\ge\tau_0$, where $C_\star$ is a constant depending on the data at $\tau=\tau_0$. (In fact, $C_\star$ necessarily depends on a higher order norm of $\phi$ \cite{Sbierski-Gaussian}.) This estimate can then be used to prove pointwise decay of $\psi$ itself. By commuting the wave equation with $T$ and the rotation vector fields $\Gamma_1,\Gamma_2$, and $\Gamma_3$, higher order versions of \eqref{eq:boundedness-intro-1}--\eqref{eq:ED} can be proved.

\subsubsection{The main energy hierarchy and the structure of the error terms}\label{sec:str-intro}

Following the usual strategy to prove global existence for nonlinear wave equations on black hole spacetimes, we seek to establish the estimates \eqref{eq:boundedness-intro-1}--\eqref{eq:r-Mp-intro} for solutions $\phi:\mathcal M\to\Bbb R^N$ of \eqref{eq:wave}. We assume the nonlinearity $\mathcal N(x,\phi,d\phi)$ is a sum of terms of the form
\begin{equation}\label{eq:nlin-intro-1}
    f(x,\phi)Y\phi_\alpha \partial_v\phi_\beta \quad\text{or}\quad f(x,\phi)r^{-2}\Gamma_i\phi_\alpha \Gamma_j\phi_\beta
\end{equation}
for $x$ close to $\mathcal H^+$ or $\mathcal I^+$, where $\phi_\alpha$ denotes a scalar component of $\phi$. We refer to nonlinearities of the form  \eqref{eq:nlin-intro-1} as \emph{strong null form} nonlinearities.

Given an index $p$, a number of commutations $k$ (where we allow for any combination of $k$ commutations with $T,\Gamma_1,\Gamma_2$, and $\Gamma_3$), we define the energy $\mathcal E_{p,k}(\tau)$ to be the $r^p$ energy on the left-hand side of \eqref{eq:boundedness-intro-1}, $\underline{\mathcal E}{}_{p,k}(\tau)$ to be the $(r-M)^{-p}$ energy on the left-hand side of \eqref{eq:boundedness-intro-1}, and then define the master energy\footnote{We also include energy fluxes along outgoing cones in the near region and along ingoing cones in the far region, but suppress these at this level of discussion.} 
\[\mathcal X_{p,k}(\tau_1,\tau_2)\doteq \sup_{\tau\in[\tau_1,\tau_2]}\big(\mathcal E_{p,k}(\tau)+\underline{\mathcal E}{}_{p,k}(\tau)\big)+\text{RHS of \eqref{eq:r-Mp-intro} with $k$ commutations}.\]
Since the estimates \eqref{eq:boundedness-intro-1}--\eqref{eq:r-Mp-intro} for the linear wave equation \eqref{eq:linear-wave} are proved by direct integration by parts arguments, we may directly repeat the proofs for the nonlinear equation \eqref{eq:wave} (treated as a linear equation with inhomogeneity) to obtain
\begin{equation}\label{eq:intro-h-1}
    \mathcal X_{p,k}(\tau_1,\tau_2) \les \mathcal E_{p,k}(\tau_1)+\underline{\mathcal E}{}_{p,k}(\tau_2)+\mathbb T_k(\tau_1,\tau_2) + \mathbb E_{p,k}(\tau_1,\tau_2)+\underline{\mathbb E}{}_{p,k}(\tau_1,\tau_2),
\end{equation}
where $\mathbb T_k(\tau_1,\tau_2)$, $\mathbb E_{p,k}(\tau_1,\tau_2)$, and $\underline{\mathbb E}{}_{p,k}(\tau_1,\tau_2)$ are spacetime integral nonlinear errors that are at least cubic in $\phi$, and we have omitted some ``anomalous'' error terms arising from the $T$-energy and Morawetz estimates. Here $\mathbb T_k(\tau_1,\tau_2)$ denotes an error term which is supported in the ``medium-$r$'' region away from $\mathcal H^+$ and $\mathcal I^+$ and requires special care because of the degeneration of \eqref{eq:r-Mp-intro} at the photon sphere $r=2M$. The terms $\mathbb E{}_{p,k}(\tau_1,\tau_2)$ and $\underline{\mathbb E}{}_{p,k}(\tau_1,\tau_2)$ come from the $r^p$ and $(r-M)^{-p}$ estimates, respectively, and take the form 
\begin{align}
\label{eq:intro-error-1}   \mathbb E_{p,k}(\tau_1,\tau_2)&\doteq \int_{\tau_1}^{\tau_2}\int_{C(\tau)} r^p\big(|\partial_v\phi^{n-6}||\partial_u\psi^k|+|\partial_u\phi^{n-6}||\partial_v\psi^k|+\cdots\big)|\partial_v\psi^k|\,d\omega dvd\tau ,\\
     \underline{\mathbb E}{}_{p,k}(\tau_1,\tau_2) &\doteq \int_{\tau_1}^{\tau_2}\int_{\underline C(\tau)} (r-M)^{-p}\big(|\partial_v\phi^{n-6}||\partial_u\psi^k|+|\partial_u\phi^{n-6}||\partial_v\psi^k|+\cdots\big)|\partial_u\psi^k| \,d\omega dud\tau,  \label{eq:intro-error-2} 
\end{align} where we recall the notation $\psi=r\phi$, the notation $f^k$ for a function $f:\mathcal M\to\Bbb R^N$ means the collection of all instances of $f$ with at most $k$ applications of the vector fields $T,\Gamma_1,\Gamma_2$, and $\Gamma_3$, and $n\ge 12$ is the maximum number of commutations we allow in our bootstrap scheme below. 

The proof of \eqref{eq:intro-h-1} is carried out in \cref{sec:a-priori} and the precise structure of the error terms is obtained later in \cref{sec:structure}. These energy hierarchies (in fact, extended versions thereof) were already obtained in \cite{AAG20}. The novelty of the present paper lies in how we handle the error terms in \eqref{eq:intro-h-1}, in particular, being able to close the bootstrap argument without extending the hierarchies past $p=2$.  

\subsubsection{The bootstrap assumptions}\label{sec:boot-intro}

As usual, we control the nonlinear errors in \eqref{eq:intro-h-1} by a continuity argument. Let $\ve_0\ge \|\mathring\phi\|_\star$ be the global smallness parameter controlling the initial data, let $A$ be a large constant to be determined, fix $0<\delta<\frac {1}{100}$ an auxiliary parameter, fix $n\ge 12$ the maximum number of commutations, and set $\ve\doteq A\ve_0$. For $\tau_f$ a bootstrap time, we assume the pointwise estimates
\begin{align}
 \label{eq:boot-intro-1}     r|\partial_u\phi^{n-6}|+ r^2|\partial_v\phi^{n-6}|& \le \ve^{1/2} \quad\text{for $\{r\ge\Lambda\}\cap\{\tau\le\tau_f\}$},\\
 \label{eq:boot-intro-2}        (r-M)^{-3/2+\delta}|\partial_u\phi^{n-6}|+|\partial_v\phi^{n-6}| & \le \ve^{1/2}\quad\text{for $\{r\le\Lambda\}\cap\{\tau\le\tau_f\}$},
\end{align} the energy estimates
\begin{align}
  \label{eq:boot-intro-3}     \mathcal X_{p,n-2}(\tau_1,\tau_2)&\le \ve^2\tau^{-2+\delta+p}_1 \hspace{-35mm} &\text{for $p\in[0,2-\delta]$},\\
 \label{eq:boot-intro-4}       \mathcal X_{p,n-1}(\tau_1,\tau_2) &\le \ve^2 \tau^{\max\{-1-\delta+p,-1\}}_1 \hspace{-35mm} &\text{for $p\in[0,1+\delta]$}, \\
  \label{eq:boot-intro-5}       \mathcal X_{p,n}(\tau_1,\tau_2)&\le \ve^2 \tau^{\max\{0,-1+3\delta+p\}}_2 \hspace{-35mm} &\text{for $p\in[0,1+\delta]$}\textcolor{white}{,}
\end{align}
when $\tau_1\le\tau_2\le\tau_f$, and the low-order nondegenerate integrated energy decay estimate at trapping
\begin{equation}\label{eq:boot-intro-6}
    \int_{\tau_1}^{\tau_2}\int_{\underline C(\tau)\cap\{d\chi_\mathrm{trap}\ne 0\}}|\partial\phi^{n-3}|^2\,d\omega dud\tau\le \ve^{3/2}\tau_1^{-2+\delta},
\end{equation}
where $\partial\phi^{n-3}$ denotes the collection of all first derivatives of $\phi^{n-3}$. Note that we do not directly assume any pointwise decay statements. Moreover, \eqref{eq:boot-intro-3} and \eqref{eq:boot-intro-4} give boundedness at the top end of the indicated $p$-ranges, but $\mathcal X_{1+\delta,n}(\tau_1,\tau_2)$ is allowed to grow mildly, like $\tau_2^{4\delta}$. 

As is standard, we use the bootstrap assumptions \eqref{eq:boot-intro-1}--\eqref{eq:boot-intro-6} to control the nonlinear errors in \eqref{eq:intro-h-1} so that after applying the standard pigeonhole argument of Dafermos--Rodnianski \cite{dafermos2010new}, we improve the constant on the right-hand side of \eqref{eq:boot-intro-3}--\eqref{eq:boot-intro-5} to $\tfrac 12$ for $A$ taken sufficiently large. The power of $\ve$ in \eqref{eq:boot-intro-1} and \eqref{eq:boot-intro-2} is improved using the method of characteristics and the power of $\ve$ in \eqref{eq:boot-intro-6} is improved by commuting with angular derivatives and relating the integral to $\mathcal X_{p,n-2}(\tau_1,\tau_2)$. We will outline some of these estimates in the following section.  
The precise definition of the bootstrap set is given in \cref{sec:bs-set}.

\subsubsection{Outline of the main error estimates}\label{sec:intro-outline}

To estimate the errors $\mathbb T_k$, $\mathbb E_{p,k}$, and $\underline{\mathbb E}{}_{p,k}$ in \eqref{eq:intro-h-1}, we follow the typical strategy of proving pointwise estimates for terms with $\le n-6$ derivatives (``lower order terms'') and applying the energy bootstraps \eqref{eq:boot-intro-3}--\eqref{eq:boot-intro-5}. However, as is well known, without proving some \emph{decay} for the lower order terms, such a procedure cannot close. Inspired by the work of Dafermos, Holzegel, Rodnianski, and Taylor in \cite{DHRT22}, we close the bootstrap argument using only pointwise estimates which are also integrated in $u$ or $v$ and can be directly obtained from the bootstraps in \cref{sec:boot-intro} without resorting to proving ``more decay.'' 

For example, the first error term in \eqref{eq:intro-error-2} can be estimated by
\begin{equation*}
    \int_{\tau_1}^{\tau_2}(r-M)^{-p}|\partial_v\phi^{n-6}||\partial_u\psi^k|^2 \,d\omega du dv \les \sup_{\tau\in[\tau_1,\tau_2]}\underline{\mathcal E}{}_{p,k}(\tau)\cdot\int_{\tau_1}^{\tau_2}\sup_{\underline C(\tau)}|\partial_v\phi^{n-6}| \,d\tau\les \ve^{3/4}\mathcal X_{p,k}(\tau_1,\tau_2)
\end{equation*}
if we can prove the $L^1_v L^\infty_{u,\omega}$ estimate
\begin{equation}\label{eq:L1-Linfty-intro}
    \int_{\tau_1}^{\tau_2}\sup_{\underline C(\tau)}|\partial_v\phi^{n-6}| \,d\tau\les\ve^{3/4}. 
\end{equation}  By breaking up the resulting spacetime integral along dyadic times and performing Cauchy--Schwarz in every dyadic interval, \eqref{eq:L1-Linfty-intro} can be inferred from the $L^2_v L^\infty_{u,\omega}$ estimate
\begin{equation}\label{eq:L2-Linfty-intro}
    \int_{\tau_1}^{\tau_2}\sup_{\underline C(\tau)}|\partial_v\phi^{n-6}|^2 \,d\tau\les \ve^{3/2}\tau_1^{-2+3\delta/2}.
\end{equation}
This estimate, which already appeared in \cite{AAG20}, is proved by first estimating the supremum over $\underline C(\tau)$ using the fundamental theorem of calculus in $u$, the wave equation, and the Sobolev inequality on $S^2$. This yields a spacetime integral which is directly estimated by the bootstrap assumptions. (The fractional powers of $\ve$ in \eqref{eq:L1-Linfty-intro} and \eqref{eq:L2-Linfty-intro} are artifacts of the bootstrap argument and can be removed after global existence has been proved.) We also have analogous $L^2_uL^\infty_{v,\omega}$ and $L^1_u L^\infty_{u,\omega}$ estimates for $r\partial_u\phi^{n-6}$ in the far region. 

Bulk error terms in the medium-$r$ region, including $\mathbb T_k(\mathcal R)$, are handled by commuting with rotation vector fields to remove trapping, as in the original \cite{dafermos2009red}, and then using $L^1L^\infty$ estimates for lower order terms (which are obtained from \eqref{eq:boot-intro-6}) as in \cite{DHRT22}. In particular, the nondegenerate $(n-3)$-commuted nondegenerate Morawetz estimate is controlled by the $(n-2)$-commuted degenerate Morawetz estimate, so that \eqref{eq:boot-intro-6} is immediately improved by \eqref{eq:boot-intro-3} for $\ve_0$ sufficiently small. 

We also require $r$-weighted pointwise estimates for $\partial_v\phi^{n-6}$ and $(r-M)$-weighted pointwise estimates for $\partial_u\phi^{n-6}$. As indicated by \eqref{eq:boot-intro-2}, we propagate boundedness of $(r-M)^{-3/2+\delta}|\partial_u\phi^{n-6}|$ near $\mathcal H^+$. Note crucially that $(r-M)^{-3/2+\delta}|\partial_u\phi^{n-6}|$ in fact \emph{vanishes} on $\mathcal H^+$ since the vector field $(r-M)^{-3/2+\delta}\partial_u$ does.  Using the method of characteristics and our bootstrap assumptions, we improve the $\ve$-power in \eqref{eq:boot-intro-2} and obtain the \emph{growing} estimate
\begin{equation}\label{eq:intro-growth}
    (r-M)^{-2}|\partial_u\phi^{n-6}|\les \ve \tau^{1/2+\delta/2}
\end{equation}
near $\mathcal H^+$, where we recall that $Y=(1-\frac{M}{r})^{-2}\partial_u$ is the nondegenerate translation-invariant null derivative transverse to $\mathcal H^+$. The reason for the power $1/2+\delta/2$ is that our bootstrap assumptions only imply that angular derivatives of $\phi^{n-6}$ decay like $\tau^{-1/2+\delta/2}$ along $\mathcal H^+$ since we do not extend the $p$-range for $\ell\ge 1$ modes of the solution (unlike \cite{AAG20}). By interpolation, we obtain the estimate \[(r-M)^{-q}|\partial_u\phi^{n-6}|\les \ve \tau^{(\frac{1+\delta}{1+2\delta})(q-3/2+\delta)}\] near $\mathcal H^+$, with $q\in [3/2 -\delta ,2]$. For different values of $p$ and $k$, we utilize this estimate with different values of $q$ depending on which $(r-M)$-weight is required to stay within the $(r-M)^{-p}$-hierarchy. For instance, when $p=2-\delta$ and $k=n-2$, we get enough decay from the Morawetz estimate for $\partial_v\phi^{n-2}$ that we can plug the growing estimate \eqref{eq:intro-growth} into \eqref{eq:intro-error-2} and easily recover boundedness. However, for the case $p=1+\delta$ and $k=n$, we are forced to use a mildly growing pointwise estimate in \eqref{eq:intro-error-2} in order to obtain the correct $(r-M)$-weights, which forces mild growth on $\mathcal X_{1+\delta,n}(\tau_1,\tau_2)$. 

\begin{rk}
    It seems necessary to take $p$ up to $1+\delta$ for $k=n$ because of one specific term in the $T$-energy error estimate (which is anomalous and hence was not discussed here). See already the estimate for $\mathbb E(\partial_v,\mathrm{i},k,\mathcal R)$ in \cref{sec:T-error} below. This is different from the situation at null infinity, where boundedness for the $T$-energy is also anomalous but can be handled with $p= \delta$ (see for instance \cite[Appendix C]{DHRT22}). It is precisely going above $p=1$ that forces topmost order energy growth in our bootstrap scheme. 
\end{rk}

\begin{rk}
    The growing upper bound in \eqref{eq:intro-growth} is caused by $\slashed\Delta\phi$ and $|\slashed\nabla\phi|^2$ terms in the equation for $\partial_vY\psi$, where $\slashed\nabla$ and $\slashed\Delta$ are the induced covariant derivative and Laplacian, respectively, on the symmetry spheres. Since these terms are automatically localized to $\ell\ge 1$ modes, the scheme of \cite{AAG20} shows that (under the stronger smallness assumption of \cref{thm:AAG20}) these terms are integrable in $\tau$. Note that restricted to spherically symmetric data $\mathring\phi$, these terms simply do not appear, and hence boundedness of $Y\phi$ is automatic. Therefore, restricted to spherical symmetry, \cref{thm:rough} generalizes the result of \cite{A16}, including the Aretakis instability, to systems satisfying the strong null condition.
\end{rk}

The $L^1L^\infty$ and $L^2 L^\infty$ estimates are obtained in \cref{sec:L1-L2}, the pointwise estimates are obtained in \cref{sec:pw}, the bulk terms in the medium-$r$ region are estimated in \cref{sec:trapping}, the $T$-energy and Morawetz  errors are estimated in \cref{sec:T-error}, the $(r-M)^{-p}$  errors are estimated in \cref{sec:horizon-error}, and the $r^p$ errors are estimated in \cref{sec:rp-error}.

Finally, the proof of the main theorem is completed in \cref{sec:completing}.

\subsection{Comments on other equations and settings}\label{sec:other-settings}

\subsubsection{Remarks on other semilinearies} \label{sec:other-eqns}

The strong null form nonlinearities we treat in this paper (recall \eqref{eq:nlin-intro-1}) are natural generalizations of the most basic null form $Q_0(\phi_\alpha,\phi_\beta)=g^{-1}(d\phi_\alpha,d\phi_\beta)$, but do not include all of the larger class of standard quadratic nonlinearities of the form $\partial\phi_\alpha\bar\partial\phi_\beta$, where $\bar\partial\in\{\partial_v,r^{-1}\Gamma_1,r^{-1}\Gamma_2,r^{-1}\Gamma_3\}$ denotes a ``good derivative'' and $\partial$ is allowed to be $\bar\partial$ or $Y$. See for instance \cite{luk2013null,DHRT22} for work on these types of equations in the subextremal case (where the structure at $\mathcal H^+$ is unimportant). Moreover, as already shown by Aretakis in \cite{aretakis2013nonlinear}, the extremal case is also sensitive to higher order derivative terms at $\mathcal H^+$, i.e., cubic (and higher) terms at the horizon cannot just be ignored, unlike at $\mathcal I^+$. One can also entertain equations with degenerate $(r-M)$-weights on the nonlinearity to weaken the contribution of bad $Y$ derivatives. 

Understanding the general structure of equations on extremal Reissner--Nordstr\"om for which global stability can hold is an interesting open problem. Unfortunately, the methods in the present paper already seem to fail for some ``good $\cdot$ bad'' quadratic nonlinearities such as $r^{-1}\Gamma_i\phi_\alpha Y\phi_\beta$. This term produces an $(r-M)^{-p}$ error of the form
\begin{equation}\label{eq:intro-aux-1}
    \iint(r-M)^{-p} |\partial_u\phi^{n-6}||\slashed\nabla\phi^k||\partial_u\psi^k|\,d\omega du dv.
\end{equation} Since the $(r-M)^{-p}$ bulk has weights $(r-M)^{-p+3}$ for $|\slashed\nabla\phi^k|^2$ and $(r-M)^{-p+1}$ for $|\partial_u\psi^k|^2$ , estimating \eqref{eq:intro-aux-1} seems to require \emph{boundedness} for $Y\phi^{n-6}\sim (r-M)^{-2}\partial_u\phi^{n-6}$, which as we explained in \cref{sec:intro-outline} does not follow from our scheme. Using the stronger (and more complicated) estimates from \cite{AAG20}, it does appear however that one can prove global stability for $r^{-1}\Gamma_i\phi_\alpha Y\phi_\beta$ under a suitable smallness assumption for $\|\cdot\|_0$. Therefore, it would seem like a full classification of equations for which stability holds \emph{does} require refined estimates at the horizon, which would then mean that the picture is also  different for the charged scalar field and Kerr cases.

\begin{rk}
    The methods in this paper \emph{do} apply to all semilinear terms of the form $f(x,\phi)\bar\partial\phi_\alpha\bar\partial\phi_\beta$, but as the ``angular $\cdot$ angular'' terms considered in \eqref{eq:nlin-intro-1} seem to already behave the worst, we omit the detailed estimates for the other types. 
\end{rk}

\subsubsection{Semilinear systems on asymptotically extremal backgrounds} 

Recently, the present authors, in collaboration with Kehle, established the codimension-one nonlinear asymptotic stability of the extremal Reissner--Nordstr\"om exterior in the nonlinear Einstein--Maxwell-neutral scalar field system in spherical symmetry \cite{AKU24}. Note that codimension one is optimal as extremal Reissner--Nordstr\"om is itself codimension one in the Reissner--Nordstr\"om family. Moreover, we showed that the Aretakis instability for the scalar field on the dynamical background persists in this setting. 

With a view towards the problem of nonlinear stability of extremal Reissner--Nordstr\"om for the Einstein--Maxwell system outside of symmetry, it is natural to study the analogue of \cref{thm:rough} on the spherically symmetric asymptotically extremal black hole spacetimes constructed in \cite{AKU24}. 

\begin{thm}\label{thm:dynamical}
    \cref{thm:rough} holds for systems of semilinear wave equations of the form \eqref{eq:wave} on the asymptotically extremal spacetimes constructed in \cite{AKU24}, provided the nonlinearity $\mathcal N$ satisfies the null condition from \cref{def:null-cond} below \ul{everywhere} in the spacetime.
\end{thm}

\begin{rk}
    The extra assumption on the nonlinearity $\mathcal N$ is a purely technical assumption coming from the fact that higher-order estimates for the dynamical black holes in \cite{AKU24} are not yet available. Note that this theorem does apply as stated to the wave map system, for instance.    
\end{rk}

We will not prove \cref{thm:dynamical} in detail, but will outline the proof, emphasizing the additional difficulties compared to \cref{thm:rough}, in \cref{app:A} below. It remains to be seen what the ``correct'' assumptions for non-spherically symmetric perturbations of extremal Reissner--Nordstr\"om are that accurately reflect the expected behavior of the Einstein--Maxwell system. 

\subsection*{Acknowledgments} The authors would like to thank John Anderson, Mihalis Dafermos, Christoph Kehle, and Jonathan Luk for helpful conversations. R.U. acknowledges support from the NSF grant DMS-2401966.

\section{Preliminaries}

\subsection{The extremal Reissner--Nordstr\"om background}\label{sec:ERN}

In this section, we briefly review some important geometric aspects of Reissner--Nordstr\"om black hole spacetimes and set up our notation. For more information, we refer the reader to \cite[Section 2.2]{AKU24}, \cite[Section 2]{Aretakis-instability-1}, and the appendix of \cite{stefanos-ern_full}. 
Fix $M>0$ and let $\mathcal M\doteq (-\infty,\infty)_v\times (0,\infty)_r\times S^2$, where $S^2$ carries standard spherical polar coordinates $(\vartheta,\varphi)$. Then the \emph{extremal Reissner--Nordstr\"om metric with mass $M$} on $\mathcal M$ is given by
\begin{equation*}
    g = -4Ddv^2+4dvdr +r^2\mathring{\slashed g},
\end{equation*}
where 
 \begin{equation*}
     D(r)\doteq \left(1-\frac{M}{r}\right)^2
 \end{equation*}
 and $\mathring{\slashed g}\doteq d\vartheta^2+\sin^2\vartheta\,d\varphi^2$ is the standard round metric on the unit sphere. The coordinates $(v,r)$ are called \emph{ingoing Eddington--Finkelstein} coordinates. The metric $g$ describes a static black hole spacetime with event horizon $\mathcal H^+$ located at $r=M$.  The vector field $T\doteq\frac 12\partial_v$ is the time-translation Killing vector field (equal to $\partial_t$ in Schwarzschild coordinates $(t,r)$, in the domain of outer communication).
The vector field $Y\doteq \partial_r$ with respect to $(v,r)$-coordinates is past-directed null and is transverse to $\mathcal H^+$. 
 
 We define the \emph{tortoise coordinate} $r_*=r_*(r)$ by $dr_* /dr=D^{-1}$ with normalization $r_*(2M)=0$. Explicitly,
\begin{equation*}
 r_*(r)=   r-M-\frac{M^2}{r-M}+2M\log\left(\frac{r-M}{M}\right).
\end{equation*} The Schwarzschild $t$ coordinate is given by $t\doteq 2v-r_*$.
Let $\Lambda \ge 100M$ be a constant to be chosen later. Using $r_*$, we can now bring $g$ into \emph{double null} form. Let $u\doteq v-r_*+c$, where $c$ is a normalization constant chosen so that $r=\Lambda$ at $\{u=0\}\cap\{v=0\}$. With respect to the coordinates $(u,v,\vartheta,\varphi)$, $g$ takes the form 
\begin{equation*}
    g = -4D dudv+r^2\mathring{\slashed g},
\end{equation*} where $r$ is now an implicitly defined function of $u$ and $v$.
These coordinates only cover the domain of outer communication $\{r>M\}$.  The event horizon $\mathcal H^+$ formally corresponds to $u=+\infty$ and null infinity $\mathcal I^+$ formally corresponds to $v=+\infty$. From the identity $r_*(r)=v-u+c$, we infer that in these coordinates, 
\begin{equation*}
  T=  \tfrac 12(\partial_u+\partial_v),\quad Y = -D^{-1}\partial_u,\quad   \partial_ur =- D,\quad \partial_v r = D.
\end{equation*}

The sets of constant $u$ and $v$ are null hypersurfaces with respect to $g$. The intersection of the hypersurfaces $\{u=u'\}$ and $\{v=v'\}$ is diffeomorphic to $S^2$ and is denoted by $S_{u',v'}$, with induced Riemannian metric
\begin{equation*}
    \slashed g\doteq r^2\mathring{\slashed g}.
\end{equation*}
The vector fields $\partial_u$ and $\partial_v$ are orthogonal to $S_{u',v'}$.

Define $\hat u = \hat u(u)$ by $d\hat u/du = D|_{v=0}$ with normalization $\hat u(0)=0$. Then $(\hat u,v)$ is again a double null coordinate system, but is now regular across $\mathcal H^+$. We shall pose initial data on the bifurcate null hypersurface 
\begin{equation}\label{eq:Sigma-0}
    \Sigma_0\doteq \big(\{0\le \hat u \le \Lambda + \tfrac 12 M\}\times\{v=0\} \times S^2\big)\cup \big(\{\hat u=0\}\times \{v\ge 0\}\times S^2\big)
\end{equation} and we will primarily be concerned with the region
   $ \mathcal D\doteq \{u\ge 0\}\cap \{v\ge 0\}\subset D^+(\Sigma_0)$.
   
We can think of the timelike hypersurface $\Gamma\doteq\{r=\Lambda\}\subset\mathcal D$ as a timelike curve in the $(u,v)$ plane. Let $\tau$ be a proper time parametrization of $\Gamma$, normalized to have $\tau=1$ at $(u,v)=(0,0)$. Extend $\tau$ to a continuous function on $\mathcal D$ by requiring it to be constant in $u$ for $r<\Lambda$ and constant in $v$ for $r>\Lambda$. It follows that $\partial_v\tau\sim 1$ for $r\le\Lambda$ and $\partial_u\tau\sim 1$ for $r\ge \Lambda$. 

\subsection{Angular derivatives and notation for commuted quantities}

Let $\slashed\nabla$ denote the induced covariant derivative on the symmetry spheres $S_{u,v}$. Note that $\slashed\nabla$ computed relative to $\slashed g$ and $\mathring{\slashed g}$ is in fact the same connection, since the Christoffel symbols are scaling invariant. We use capital Latin letters ($A$, $B$, etc.) to denote abstract indices for tensors intrinsic to the symmetry spheres. Given a tensor field $\slashed T$ tangent to the symmetry spheres, we define the norm $|\slashed T|$ relative to $\slashed g$ and $|\slashed T|_\circ$ relative to $\mathring{\slashed g}$. For example, if $f$ is a function on $S_{u,v}$, then
\begin{equation*}
    |\slashed\nabla f| \doteq  \big((\slashed g^{-1})^{AB} \slashed\nabla_Af\slashed\nabla_B f\big)^{1/2} = \big(r^{-2}(\mathring{\slashed g}^{-1})^{AB} \slashed\nabla_Af\slashed\nabla_B f\big)^{1/2} =r^{-1}|\slashed\nabla f|_\circ.
\end{equation*}

In this paper, we will commute with rotation vector fields, which are Killing fields for $g$. Relative to the standard spherical coordinates $(\vartheta,\varphi)$, these are defined to be
\begin{equation*}
    \Gamma_1\doteq \partial_{\varphi}, \quad
    \Gamma_2\doteq \sin\varphi\partial_\vartheta+\cot\vartheta\cos\varphi\partial_\varphi, \quad
    \Gamma_3\doteq \cos\varphi\partial_\vartheta-\cot\vartheta\sin\varphi\partial_\varphi
\end{equation*}
and satisfy the $\mathfrak{so}(3)$ commutation relation $[\Gamma_i,\Gamma_j]=\epsilon_{ij}{}^k\Gamma_k$. We also set
\begin{equation*}
    \Gamma_0\doteq T,
\end{equation*}
the time-translation Killing field, which commutes with the $\Gamma_i$'s. 

For $k\in\Bbb N_0$, let $\mathbf k=(k_1,\dotsc, k_k)$ denote a $k$-tuple for which each entry is either 0, 1, 2, or 3. We write $|\k|=k$. Given a function $f$ on (an open set of) $\mathcal M$, we define
\begin{equation*}
    \Gamma^\k f \doteq f^\k \doteq \Gamma_{k_k}\cdots \Gamma_{k_1} f.
\end{equation*}
In this context, we also define
\begin{equation*}
    |f^k|\doteq \sum_{|\mathbf k|\le k} |f^\k|.
\end{equation*}
Therefore, instead of writing for instance ``$|f^\k|\les 1$ for all $|\k|\le n-6$,'' it suffices to simply write ``$|f^{n-6}|\les 1$.'' It will always be clear from context that $k$ is not meant as a standard exponential when this notation is being used.

Since the vector fields $\Gamma_i$ span the symmetry spheres, it holds that
\begin{equation}\label{eq:ang-1}
    \sum_{i=1}^3|\Gamma_i f|\sim |\slashed\nabla f|_\circ \sim r|\slashed\nabla f|.
\end{equation}
The induced Laplacian on the symmetry spheres is defined by
\begin{equation*}
    \slashed\Delta\doteq (\slashed g^{-1}){}^{AB}\slashed\nabla_A\slashed\nabla_B = r^{-2}\mathring{\slashed\Delta},
\end{equation*}
where $\mathring{\slashed\Delta}$ is the Laplacian on the unit sphere. Since $\mathring{\slashed\Delta}=(\Gamma_1)^2+(\Gamma_2)^2+(\Gamma_3)^2$ on functions, it holds that 
\begin{equation}\label{eq:ang-2}
    |\slashed\Delta f|\les \sum_{i=1}^3r^{-1}|\slashed\nabla \Gamma_i f|\les r^{-1}|\slashed\nabla f^1|.
\end{equation}

\begin{rk}
    The second inequality is very suboptimal, as our notation allows for an application of $T$ on the right-hand side. We do not make an attempt to carefully distinguish between angular and $T$ commutations in this paper, so this estimate suffices for our purposes.
\end{rk}

We use the shorthand $\omega=(\vartheta,\varphi)$ to refer to the spherical variable. The integration measure with respect to $\mathring{\slashed g}$ is then written as
\begin{equation*}
    d\omega \doteq \sin\vartheta\,d\vartheta d\varphi.
\end{equation*}
We require the following standard Sobolev inequality:
 \begin{equation}\label{eq:angular-Sobolev}
     \sup_{S_{u,v}}|f|^2 \les \sum_{|\k|\le 2}\int_{S_{u,v}}| f^\k|^2\,d\omega\doteq \int_{S_{u,v}}|f^2|^2\,d\omega,
    \end{equation}
    which is again very suboptimal since our notation allows for applications of $T$ and $T^2$ on the right-hand side. 

Let $N\in\Bbb N$ and let $\phi:\mathcal M\to \Bbb R^N$. We define the notations $\slashed\nabla \phi$, $\phi^\k$, etc., by applying the above definitions to each of the scalar components $\phi_\alpha$, $\alpha\in\{1,\dotsc,N\}$, of $\phi$. Norms are then defined by using the standard inner product on $\Bbb R^N$, for instance,
\begin{equation*}
    |\slashed\nabla\phi|^2 \doteq \sum_{\alpha=1}^N |\slashed \nabla\phi_\alpha|^2.
\end{equation*}
We will also use the dot product notation $\cdot$ to indicate when we are taking an inner product with respect to this $\Bbb R^N$ structure, for instance,
\begin{equation*}
    \partial_u\phi\cdot \partial_v\phi \doteq \sum_{\alpha=1}^N \partial_u\phi_\alpha\partial_v\phi_\alpha.
\end{equation*}

\subsection{Semilinear systems satisfying the strong null condition}\label{sec:nonlinearity}

Let $\phi$ be a collection of $N$ real-valued scalar fields on $\mathcal M$. We consider systems of semilinear wave equations of the form 
\begin{equation}
    \Box_g\phi =\mathcal N(x,\phi,d\phi),\label{eq:wave-1}
\end{equation}
where $g$ is a spherically symmetric metric and $\mathcal N(x,\phi,d\phi)$ is an $\Bbb R^N$-valued smooth function of its variables. Our definition of null condition generalizes the standard metric null form example where $\phi$ is a scalar and \[\mathcal N=g^{-1}(d\phi,d\phi)=-4Y\phi\partial_v\phi + |\slashed\nabla\phi|^2, \]
and scalar equations with $\mathcal N= f(x,\phi) g^{-1}(d\phi,d\phi) $, which were considered in \cite{AAG20}.

\begin{defn}\label{def:null-cond}
    We say that $\mathcal N$ satisfies the \emph{strong null condition} if there exist numbers $r_1>r_0>M$ such that each component of $\mathcal N$ can be expressed as a sum of terms of the form
    \begin{equation}\label{eq:nlin-0}
     f(x,\phi)\phi_\alpha\phi_\beta,\quad f(x,\phi)\phi_\alpha X_1\phi_\beta ,\quad \text{or}\quad   f(x,\phi) X_1\phi_\alpha X_2\phi_\beta
    \end{equation}
    when $r_0\le r\le r_1$, where $X_1,X_2\in \{\partial_u,\partial_v,\Gamma_1,\Gamma_2,\Gamma_3\}$ and a sum of terms of the form
    \begin{equation}\label{eq:nlin-1}
     f(x,\phi)D^{-1}\partial_u\phi_\alpha\partial_v\phi_\beta
    \end{equation}
    or 
    \begin{equation}\label{eq:nlin-2}
       f(x,\phi)r^{-2}\Gamma_i\phi_\alpha\Gamma_j\phi_\beta,
    \end{equation} for $r<r_0$ and $r>r_1$.
 Here $f$ stands for any smooth function $f:\mathcal M\times \Bbb R^N\to \Bbb R$ such that 
    \begin{equation}\label{eq:f-assn}
        \sup_{\mathcal M\times\{|\phi|\les 1\}} |T^{k_1}\slashed\nabla_\omega^{k_2} D_\phi^{k_3} f| \les_{k_1,k_2,k_3} 1
    \end{equation}
    for any $k_1$, $k_2$, and $k_3$.
\end{defn}

\begin{rk}
    We emphasize that the null structure of $\mathcal N$ is only required near the horizon (i.e., for $r-M$ small) and near null infinity (i.e., for $r$ large). Without loss of generality, we may assume that $r_0<2M$.
\end{rk}

Since the vector fields $\Gamma_0,\dotsc,\Gamma_4$ are Killing, we obtain from \eqref{eq:wave-1} the equations
\begin{align}
\label{eq:wave-phi-uv}  \partial_u\partial_v\phi^\k  & =  r^{-1}D(\partial_v-\partial_u)\phi^\k  +D\slashed\Delta\phi^\k - D \mathcal N^\k,\\
 \label{eq:wave-lambda}  \partial_u(r\partial_v\phi^\k) &=  -D \partial_u\phi^\k +rD \slashed\Delta\phi^\k - rD \mathcal N^\k,\\
  \label{eq:wave-nu} \partial_v(r\partial_u\phi^\k) &=  D \partial_v\phi^\k + rD\slashed\Delta\phi^\k - rD \mathcal N^\k.
\end{align}
Defining $\psi\doteq r\phi$, we also have
\begin{equation}
  \label{eq:wave-psi}  \partial_u\partial_v\psi^{\mathbf k}= -DD'\phi^{\mathbf k}+D\slashed\Delta\psi^{\mathbf k}-rD\mathcal N^\k,
\end{equation} where $D'\doteq dD/dr$, since $\partial_u\partial_vr = -DD'$.

Given a vector field $X=X^u\partial_u+X^v\partial_v$, we derive from \eqref{eq:wave-phi-uv} the general multiplier identity
\begin{multline}\label{eq:main-energy-identity}
    \partial_v\big(r^2 X^u|\partial_u\phi^\k|^2\big) + \partial_u\big(r^2X^v|\partial_v\phi^\k|^2\big) = r^2\partial_v X^u|\partial_u\phi^\k|^2+r^2\partial_uX^v|\partial_v\phi^\k|^2 +2rD(X^u- X^v)\partial_u\phi^\k\cdot \partial_v\phi^\k\\
    + 2D X^u \mathring{\slashed\Delta}\phi^\k\cdot \partial_u\phi^\k +2D X^v \mathring{\slashed\Delta}\phi^\k \cdot \partial_v\phi^\k - r^2D (X^u\partial_u\phi^\k + X^v\partial_v\phi^\k)\cdot  \mathcal N^\k.
\end{multline}

\section{Precise statement of the main theorem}\label{sec:thm}

\subsection{Definitions and notation} 

\subsubsection{Foliations and regions} \label{sec:regions}

Recall the region $\mathcal D\doteq\{u\ge 0,v\ge 0\}$ in the domain of outer communication of extremal Reissner--Nordstr\"om and the timelike hypersurface $\Gamma\doteq\{r=\Lambda\}$ in $\mathcal D$, where we now choose $\Lambda\ge \max\{r_1,100M\}$ (recall $r_1$ from \cref{def:null-cond}). Given $\tau_f\ge 1$, we set
\begin{equation*}
    \mathcal D^{\tau_f} \doteq J^-(\Gamma(\tau_f)).
\end{equation*}
Let $\tau\mapsto(\Gamma^u(\tau),\Gamma^v(\tau))$ be the arc length parametrization of the projection of $\Gamma$ to the $(u,v)$-plane. We define four classes of null hypersurfaces in $\mathcal D^{\tau_f}$:
\begin{align*}
   C_u^{\tau_f}&\doteq \big(\{u\}\times[0,\Gamma^v(\tau)]\times S^2\big)\cap \{r\ge \Lambda \} ,&\underline C{}_v^{\tau_f}&\doteq\big([0,\Gamma^{u_{\tau_f}}(\tau_f)]\times\{v\}\times S^2\big)\cap\{r\le \Lambda \},\\
   H_u^{\tau_f}&\doteq \big(\{u\}\times[0,\Gamma^v(\tau)]\times S^2\big)\cap \{r\le \Lambda \},&\underline H{}_v^{\tau_f}&\doteq\big([0,\Gamma^{u_{\tau_f}}(\tau_f)]\times\{v\}\times S^2\big)\cap\{r\ge \Lambda \}.
\end{align*}
We can send $\tau_f\to \infty$, in which case we simply write $C_u^\infty=C_u$, etc. It is also convenient to write
\begin{equation*}
    C^{\tau_f}(\tau)\doteq C^{\tau_f}_{\Gamma^u(\tau)} ,\quad \underline C^{\tau_f}(\tau)\doteq \underline C{}_{\Gamma^v(\tau)}^{\tau_f}.
\end{equation*}

 Finally, we set
\begin{align*}
    \mathcal D_{\le\Lambda}&\doteq \{r\le\Lambda\}\cap\mathcal D,\\
     \mathcal D_{\ge\Lambda}&\doteq \{r\ge\Lambda\}\cap\mathcal D,\\
     \mathcal A&\doteq \{r\le r_{-1}\}\cap\mathcal D,\\
     \mathcal A'&\doteq \{r\le r_{-2}\}\cap\mathcal D,\\
     \mathcal B_{\rho_1,\rho_2}&\doteq \{\rho_1\le r\le\rho_2\}\cap\mathcal D,
\end{align*}
where $r_{-1}\doteq M+ (r_0-M)/2$, $r_{-2}\doteq M+ (r_{-1}-M)/2$, and $\rho_2>\rho_1>M$, as well as the obvious $\tau_f$-modified versions where we intersect with $\mathcal D^{\tau_f}$. 

 \begin{figure}
\centering{
\def\svgwidth{15pc}
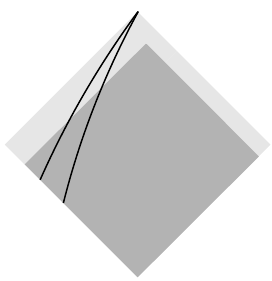}
\caption{A Penrose diagram showing the main regions, null hypersurfaces, and energies involved in our bootstrap argument. The dark shaded region is $\mathcal D^{\tau_f}$. The function $\tau$ measures advanced time to the left of $\Gamma$ and retarded time to the right of $\Gamma$.}
\label{fig:bootstrap-setup}
\end{figure}

\subsubsection{Definitions of energies}\label{def:energies}

Fix a parameter $\delta\in (0,\frac{1}{100})$, define $p_\star = p$ if $p<1$ and $p_\star =0$ if $p\ge 1$, and set $c_{p_\star}\doteq (p_\star-1)^2$. Given $k\in\Bbb N_0$, $\tau_f\in[1,\infty]$, $\tau,\tau_1,\tau_2\in[1,\tau_f]$, $p\in\{0\}\cup[\delta,2-\delta]$, $(u,v,\cdot)\in\mathcal D^{\tau_f}$, $\phi:\mathcal D^{\tau_f}\to\Bbb R^N$ with $\psi\doteq r\phi$, we define:
\begin{enumerate}
    \item The outgoing $r^p$-weighted flux to null infinity:
    \begin{equation*}
        \mathcal E_{p,k}^{\tau_f}(\tau)\doteq \int_{C^{\tau_f}(\tau)}\big(r^p|\partial_v\psi^k|^2+c_{p_\star}r^{p_\star+2}|\partial_v\phi^k|^2+r^2|\slashed\nabla\phi^k|^2+c_{p_\star}r^{p_\star}|\phi^k|^2\big)\,d\omega dv.
    \end{equation*}
    \item The ingoing $(r-M)^{-p}$-weighted flux to the horizon:
    \begin{equation*}
       \underline{ \mathcal E}{}_{p,k}^{\tau_f}(\tau)\doteq \int_{\underline C^{\tau_f}(\tau)}\big((r-M)^{-p}(|\partial_u\psi^k|^2+|\partial_u\phi^k|^2)+(r-M)^{2}|\slashed\nabla\phi^k|^2+c_{p_\star}(r-M)^{-p_\star+2}|\phi^k|^2\big)\,d\omega du.
    \end{equation*}
    \item The ingoing $r^p$-weighted flux near null infinity:
\begin{equation*}
    \underline{\mathcal F}{}_{p,k}^{\tau_f}(v,\tau_1,\tau_2)\doteq \int_{\underline H{}_v^{\tau_f}\cap\{\tau_1\le\tau\le\tau_2\}} \big(r^2|\partial_u\phi^k|^2+r^{p+2}|\slashed\nabla\phi^k|^2+r^p|\phi^k|^2\big)\, d\omega du.
\end{equation*}
    \item The outgoing $(r-M)^{-p}$-weighted flux near the horizon:
    \begin{equation*}
           \mathcal F_{p,k}^{\tau_f}(u,\tau_1,\tau_2)\doteq \int_{H^{\tau_f}_u\cap\{\tau_1\le\tau\le\tau_2\}} \big(|\partial_u\phi^k|^2+(r-M)^{-p+2}|\slashed\nabla\phi^k|^2+(r-M)^{-p+2}|\phi^k|^2\big) \,d\omega dv.
    \end{equation*}
    \item The integrated energy in the far region:
    \begin{equation*}
    \mathcal M^{\tau_f}_{p-1,k}(\tau)\doteq \mathcal M^{\tau_f}_{-1,k}(\tau)+\mathbf 1_{p\in[\delta,2-\delta]}\dot{\mathcal E}_{p-1,k}^{\tau_f}(\tau),
\end{equation*}
where
    \begin{equation*}
        \mathcal M^{\tau_f}_{-1,k}(\tau)\doteq \int_{C^{\tau_f}(\tau)} \big(r^{1-\delta}|\partial_u\phi^k|^2+r^{1-\delta}|\partial_v\phi^k|^2+r|\slashed\nabla\phi^k|^2+r^{-1-\delta}|\phi^k|^2\big)\,d\omega dv,
    \end{equation*}
    \begin{equation*}
        \dot{\mathcal E}^{\tau_f}_{p-1,k}(\tau)\doteq \int_{C^{\tau_f}(\tau)} \big(r^{p-1}|\partial_v\psi^k|^2+r^{p+1}|\partial_v\phi^k|^2+r^{p+1}|\slashed\nabla\phi^k|^2+r^{p-1}|\phi^k|^2\big)\,d\omega dv.
    \end{equation*}

    \item The integrated energy in the near region with degeneration at trapping: 
    \begin{equation*}
    \underline{\mathcal M}_{p-1,k}^{\tau_f}(\tau)\doteq   \underline{\mathcal M}_{-1,k}^{\tau_f}(\tau)+\mathbf 1_{p\in[\delta,2-\delta]}  \dot{\underline{\mathcal E}}{}_{p-1,k}^{\tau_f}(\tau),
\end{equation*}
where
\begin{multline*}
    \underline{\mathcal M}_{-1,k}^{\tau_f}(\tau)\doteq \int_{\underline C^{\tau_f}(\tau)\cap\mathcal A}\big((r-M)^{1+\delta}(|\partial_u\phi^k|^2+|\partial_v\phi^k|^2)+(r-M)^3|\slashed\nabla\phi^k|^2+(r-M)^4|\phi^k|^2\big)\,d\omega du \\ + \int_{\underline C^{\tau_f}(\tau)\cap\mathcal B_{r_{-1},3M}}\big(|R\phi^k|^2+|\phi^k|^2\big)\,d\omega du + \int_{\underline C^{\tau_f}(\tau)\cap\mathcal B_{3M,\Lambda}}\big(|\partial_u\phi^k|^2+|\partial_v\phi^k|^2+|\slashed\nabla\phi^k|^2+|\phi^k|^2\big)\,d\omega du,
\end{multline*}
\begin{equation*}
    \dot{\underline{\mathcal E}}{}_{p-1,k}^{\tau_f}(\tau)\doteq\int_{\underline C^{\tau_f}(\tau)\cap\mathcal A}\big((r-M)^{-p+1}(|\partial_u\psi^k|^2+|\partial_u\phi^k|^2)+(r-M)^{-p+3}|\slashed\nabla\phi^k|^2+(r-M)^{-p+3}|\phi^k|^2\big) d\omega du,
\end{equation*}
where $R\doteq \frac 12(\partial_v-\partial_u)$  is the $\partial_{r_*}$ vector field in the coordinates $(t,r_*,\vartheta,\varphi)$.

\item The integrated energy in the near region with trapping removed:
\begin{equation*}
     \widetilde{\underline{\mathcal M}}_{-1,k}^{\tau_f}(\tau)\doteq  \underline{\mathcal M}_{-1,k}^{\tau_f}(\tau)+\int_{\underline C^{\tau_f}(\tau)\cap\mathcal B_{r_{-1},3M}}\big(|\partial_u\phi^k|^2+|\partial_v\phi^k|^2+|\slashed\nabla\phi^k|^2\big)\,d\omega du.
\end{equation*}
\item The spacetime integral controlling the trapping region:
\begin{equation*}
    \mathcal Y_k^{\tau_f}(\tau_1,\tau_2)\doteq \int_{\tau_1}^{\tau_2}\widetilde{\underline{\mathcal M}}{}_{-1,k}^{\tau_f}(\tau)\,d\tau.
\end{equation*}
    \item The master energy with degeneration at trapping: 
    \begin{multline*}
     \mathcal X^{\tau_f}_{p,k}(\tau_1,\tau_2)\doteq \sup_{\tau\in[\tau_1,\tau_2]}\big(\mathcal E^{\tau_f}_{p,k}(\tau)+\underline{\mathcal E}{}^{\tau_f}_{p,k}(\tau)\big) + \sup_{v\ge 0}\underline{\mathcal F}{}_{p,k}^{\tau_f}(v,\tau_1,\tau_2)+\sup_{u\ge 0}\mathcal F^{\tau_f}_{p,k}(u,\tau_1,\tau_2)\\+\int_{\tau_1}^{\tau_2}\big(\mathcal M^{\tau_f}_{p-1,k}(\tau)+\underline{\mathcal M}{}^{\tau_f}_{p-1,k}(\tau)\big)\,d\tau.
\end{multline*}
\end{enumerate}
When $\tau_f=\infty$, we simply omit the superscript $\tau_f$ everywhere. 

\begin{rk}
We capture the degeneration of $T\phi^k$ and $\slashed\nabla\phi^k$ in the Morawetz estimate at the photon sphere $r=2M$ by simply omitting these terms from $\underline{\mathcal M}_{-1,k}^{\tau_f}$ in the integral over $\mathcal B_{r_{-1},3M}$. One could include the terms $(r-2M)^2|T\phi^k|^2$ and $(r-2M)^2|\slashed\nabla\phi^k|^2$ in the integral over  $\mathcal B_{r_{-1},3M}$ (see already \eqref{eq:Mor-intermediate-1}), but this is unnecessary for our purposes.
\end{rk}

\subsubsection{Local existence and the initial data norm}\label{sec:loc}

As mentioned in \cref{sec:ERN}, we pose characteristic initial data for the wave equation \eqref{eq:wave-1} on the bifurcate null hypersurface $\Sigma_0$, which was defined explicitly in \eqref{eq:Sigma-0} and corresponds to $\tau=1$.  

Let $\mathring\phi$ be a smooth function on $\Sigma_0$. Using the wave equation \eqref{eq:wave-1}, we can compute the jet of any smooth solution $\phi$ of \eqref{eq:wave-1} such that $\mathring\phi = \phi|_{\Sigma}$. Indeed, $T\Sigma_0$ is $(3+1)$-dimensional along the bifurcation sphere, and we may view the equation  \eqref{eq:wave-1} (see also \eqref{eq:wave-phi-uv}) as a transport equation for transverse derivatives to $\Sigma_0$, which we integrate away from the bifurcation sphere. Note that these transport equations are possibly nonlinear since we do not assume the null condition holds everywhere. For $n$ fixed, we say that $\mathring\phi$ is \emph{n-admissible} if this procedure indeed generates the $(n+1)$-jet of a solution $\phi$ to \eqref{eq:wave-1} along $\Sigma_0$. For $n$-admissible characteristic data, we define the initial data norm
\begin{multline}
    \|\mathring\phi\|_\star^2 \doteq \mathcal E_{2-\delta,n-2}[\phi](1)+\underline{\mathcal E}{}_{2-\delta,n-2}[\phi](1)+ \mathcal E_{1+\delta,n}[\phi](1)+\underline{\mathcal E}{}_{1+\delta,n}[\phi](1) \\+ \sup_{\underline C(1)}|(r-M)^{-2}\partial_u\psi^{n-6}|^2+\sup_{C(1)}|r^{3/2-\delta}\partial_v\psi^{n-6}|^2.
\end{multline}

\begin{rk}\label{rk:seed}
    It is possible to construct a norm $\|\mathring\phi\|_\flat$ that can be computed directly in terms of $\mathring\phi$ and such that if $\|\mathring\phi\|_\flat$ is small, then $\mathring\phi$ is $n$-admissible and $ \|\mathring\phi\|_\star\les\|\mathring\phi\|_\flat$. However, $\|\cdot\|_\flat$ will necessarily depend on strictly more than $n$ derivatives of $\mathring\phi$ because of the loss of angular derivatives when integrating transport equations to obtain the jet of $\phi$ from $\mathring\phi$. For an example of this procedure in a more complicated setting, see for instance \cite[Chapter 5]{DHRT}.
\end{rk}

By methods as in \cite[Section 7]{Luk-local-existence}, one can show that if $\mathring\phi$ is smooth $n$-admissible and $\|\mathring\phi\|_\star$ is sufficiently small, then there exists a smooth solution $\phi$ to \eqref{eq:wave-1} defined in a future neighborhood of $\Sigma_0$ such that $\phi|_{\Sigma_0}=\mathring\phi$. Note that smallness is essential since we do not assume that the null condition holds everywhere. See \cite[Section 7.2]{Luk-local-existence}.

\subsection{Statement of the main theorem}\label{sec:statement}

\begin{thm}\label{thm:main} Let $M>0$, $0<\delta<\frac{1}{100}$, and $n\ge 12$. For any nonlinearity $\mathcal N$ satisfying the hypotheses of \cref{def:null-cond} on the extemal Reissner--Nordstr\"om black hole of mass $M$, there exist $C>0$ and $\ve_\mathrm{stab}>0$, depending only on $M,\delta,n$, and $\mathcal N$, with the following property. 

Let $\mathring\phi$ be smooth $n$-admissible characteristic data for the wave equation \eqref{eq:wave-1} on the bifurcate null hypersurface $\Sigma_0$, satisfying the smallness assumption 
\begin{equation}\label{eq:smallness}
    \|\mathring\phi\|_\star\le \ve_0\le\ve_\mathrm{stab}.
\end{equation}
Then the solution to \eqref{eq:wave-1} with initial data $\mathring\phi$ exists on the domain of outer communication of the black hole, $\mathcal D$, and extends smoothly to the event horizon $\mathcal H^+$. Moreover, the solution satisfies the estimates
\begin{gather*} 
\mathcal Y_{n-3}(\tau,\infty)\le C\ve_0^2\tau^{-2+\delta},\\
     \mathcal X_{0,n-2}(\tau,\infty)\le C\ve_0^2\tau^{-2+\delta},\quad     \mathcal X_{2-\delta,n-2}(\tau,\infty)\le C\ve_0^2,\\
       \mathcal X_{0,n-1}(\tau,\infty)\le C\ve_0^2\tau^{-1},\quad     \mathcal X_{1+\delta,n-1}(\tau,\infty)\le C\ve_0^2, \\
          \mathcal X_{0,n}(\tau,\infty)\le C\ve_0^2,\quad     \mathcal E_{1+\delta,n}(\tau)+\underline{\mathcal E}{}_{1+\delta,n}(\tau)\le C\ve_0^2\tau^{4\delta}
\end{gather*}
for every $\tau\ge 1$.
\end{thm}

\section{A priori energy estimates}\label{sec:a-priori}

In this section, we derive the fundamental hierarchies of energy estimates for solutions of the nonlinear wave equation \eqref{eq:wave-1}, where we treat the nonlinearity $\mathcal N$ as an inhomogeneity. In \cref{sec:conventions}, we define the spacetime region and null hypersurfaces involved in our energy estimates. In \cref{sec:Hardy}, we summarize some relevant Hardy inequalities. In \cref{sec:T-energy}, we use the multiplier $T$ to derive the basic degenerate energy estimate. In \cref{sec:Morawetz}, we prove Morawetz estimates for $\phi$, i.e., weighted spacetime $L^2$ bounds for $\phi$ and its derivatives. In \cref{sec:horizon-hierarchy}, we prove $(r-M)^{-p}$-weighted energy estimates for $\partial_u\psi$. Finally, in  \cref{sec:rp-hierarchy}, we prove $r^p$-weighted energy estimates for $\partial_v\psi$ in the far region.

\subsection{Conventions and notation}\label{sec:conventions}

The energy estimates in this section will take place on regions $\mathcal R$ in the domain of outer communication of extremal Reissner--Nordstr\"om which is defined as follows: Recall the timelike hypersurface $\Gamma\doteq\{r=\Lambda\}$. Let $1\le\tau_1\le\tau_2$ and consider $(u_1,v_1)$ and $(u_2,v_2)$ defined by $\tau(u_1,v_1)=\tau_1$ and $\tau(u_2,v_2)=\tau_2$. Given $u_2'>u_2$ and $v_2'>v_2$, we now set 
\begin{equation*}
    \mathcal R\doteq \big([u_1,u_2]\times[v_1,v_2']\cup [u_1,u_2']\times[v_1,v_2]\big)\times S^2
\end{equation*}
and 
\begin{equation*}
    \mathcal R_{\le\Lambda}\doteq\mathcal R\cap\{r\le\Lambda\},\quad  \mathcal R_{\ge\Lambda}\doteq\mathcal R\cap\{r\ge\Lambda\}.
\end{equation*}
The boundary $\partial \mathcal R$ consists of null hypersurfaces which are numbered I--VI as depicted in \cref{fig:butterfly} below. 

To avoid cluttering equations, we will omit the measures in most of the integrals in the remainder of the paper. Integrals along outgoing cones ($C_u$ or $H_u$) are always taken with respect to $d\omega dv$ and integrals along ingoing cones ($\underline C{}_v$ or $\underline H{}_v$) are always taken with respect to $d\omega du$. Spacetime volume integrals, which are written as $\iint$, are always taken with respect to $d\omega du dv$.

 \begin{figure}[h]
\centering{
\def\svgwidth{12pc}
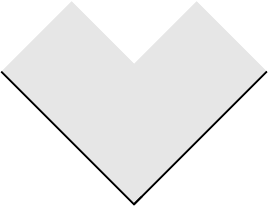}
\caption{A Penrose diagram depicting the region $\mathcal R$ and the hypersurfaces I--VI used in the energy estimates in this section.}
\label{fig:butterfly}
\end{figure}

\subsection{Hardy inequalities}\label{sec:Hardy}

To bound lower order terms in our estimates, we require the following Hardy inequalities. They can be proved by using Lemmas 6.5 and 6.20 in \cite{AKU24} with $\alpha = p-2$ and $f=\chi\psi^k$ for appropriate choices of cutoff functions $\chi$, integrating over $\omega\in S^2$, and also using the trivial estimates $|\partial_u\phi^k|\les r^{-1}|\partial_u\psi^k|+r^{-1}D|\phi^k|$ and $|\partial_v\phi^k|\les r^{-1}|\partial_v\psi^k|+r^{-1}D|\phi^k|$ .

\begin{lem}[$(r-M)^{-p}$-weighted Hardy inequalities]\label{lem:hardy-near}
Let $\mathcal R$ be a region as depicted in \cref{fig:butterfly} and let $k\in\mathbb N_0$. For any $p<1$ and $\tau\in[\tau_1,\tau_2]$, it holds that
\begin{equation*}
   (p-1)^2 \int_{\underline C(\tau)\cap\mathcal R\cap\mathcal A} (r-M)^{-p+2}|\phi^k|^2 \les  \int_{\underline C(\tau)\cap\mathcal R\cap\mathcal A} (r-M)^{-p}|\partial_v\psi^k|^2\ + \int_{\underline C(\tau)\cap\mathcal R\cap(\mathcal A\setminus\mathcal A')}|\phi^k|^2,
\end{equation*}
\begin{equation*}
    \int_{\underline C(\tau)\cap\mathcal R\cap\mathcal A}(r-M)^{-p}|\partial_u\phi^k|^2 \les  \int_{\underline C(\tau)\cap\mathcal R\cap\mathcal A} \big((r-M)^{-p}|\partial_u\psi^k|^2+(r-M)^{-p+2}|\phi^k|^2\big),
\end{equation*} where the regions $\mathcal A$ and $\mathcal A'$ were defined in \cref{sec:regions},
and for any $p<2$, it holds that
\begin{equation*}
   (p-2)^2 \iint_{\mathcal R\cap\mathcal A}(r-M)^{-p+3}|\phi^k|^2\les\iint_{\mathcal R\cap\mathcal A} (r-M)^{-p+1}|\partial_v\psi^k|^2 + \int_{\mathcal R\cap(\mathcal A\setminus\mathcal A')}|\phi^k|^2
\end{equation*}
\begin{equation*}
    \iint_{\mathcal R\cap\mathcal A} (r-M)^{-p+1}|\partial_u\phi^k|^2\les  \int_{\mathcal R\cap\mathcal A} \big((r-M)^{-p+1}|\partial_u\psi^k|^2+(r-M)^{-p+3}|\phi^k|^2\big).
\end{equation*}
\end{lem}

\begin{lem}[$r^p$-weighted Hardy inequalities]\label{lem:hardy-far} Let $\mathcal R$ be a region as depicted in \cref{fig:butterfly} and let $k\in\mathbb N_0$. For any $p<1$ and $\tau\in[\tau_1,\tau_2]$, it holds that
\begin{equation*}
   (p-1)^2 \int_{C(\tau)\cap\mathcal R} r^{p}|\phi^k|^2\,d\omega dv \les  \int_{C(\tau)\cap\mathcal R} r^{p}|\partial_v\psi^k|^2+ \int_{C(\tau)\cap\mathcal R\cap\{r\le2\Lambda\}}|\phi^k|^2,
\end{equation*}
\begin{equation*}
    \int_{C(\tau)\cap\mathcal R} r^{p+2}|\partial_v\phi^k|^2\les  \int_{C(\tau)\cap\mathcal R} r^{p}\big(|\partial_v\psi^k|^2+|\phi^k|^2\big),
\end{equation*}
and for any $p<2$, it holds that
\begin{equation*}
   (p-2)^2 \iint_{\mathcal R_{\ge\Lambda}} r^{p-1}|\phi^k|^2 \les \iint_{\mathcal R_{\ge\Lambda}} r^{p-1}|\partial_v\psi^k|^2 + \int_{\mathcal R\cap\{\Lambda\le r\le2\Lambda\}}|\phi^k|^2,
\end{equation*}
\begin{equation*}
    \iint_{\mathcal R_{\ge\Lambda}} r^{p+1}|\partial_v\phi^k|^2 \les  \int_{\mathcal R_{\ge\Lambda}} r^{p-1}\big(|\partial_v\psi^k|^2+|\phi^k|^2\big).
\end{equation*}
\end{lem}

\subsection{The \texorpdfstring{$T$}{T}-energy estimate}\label{sec:T-energy}

We now prove the basic degenerate energy estimate for \eqref{eq:wave-1}. We use the multiplier vector field $T$, the time-translation Killing field, in the general multiplier identity \eqref{eq:main-energy-identity}. As in \cite{AKU24}, we also use a null Lagrangian current (analogous to $\varpi$ in \cite{DHRT22,Morawetz-note}) to add a zeroth order term to the estimate.

\begin{prop}[The $T$-energy estimate]\label{prop:T-energy-est}   Let $\phi$ be a solution of the wave equation \eqref{eq:wave-1} on the region $\mathcal R$ depicted in \cref{fig:butterfly}. For any $k\in\Bbb N_0$, it holds that
  \begin{multline}\label{eq:T-energy-est}
\int_\mathrm{III}\big(r^2|\partial_u\phi^k|^2+r^2|\slashed\nabla\phi^k|^2+|\phi^k|^2\big) +\int_\mathrm{IV}\big(r^2|\partial_v\phi^k|^2+r^2|\slashed\nabla\phi^k|^2+|\phi^k|^2\big) \\ +\int_\mathrm{V}\big(|\partial_u\phi^k|^2+D|\slashed\nabla\phi^k|^2+D|\phi^k|^2\big)+\int_\mathrm{VI}\big(|\partial_v\phi^k|^2+D|\slashed\nabla\phi^k|^2+D|\phi^k|^2\big) \\
     \les \int_\mathrm{I}\big(|\partial_u\phi^k|^2+D|\slashed\nabla\phi^k|^2+D|\phi^k|^2\big) + \int_\mathrm{II}\big(r^2|\partial_v\phi^k|^2+r^2|\slashed\nabla\phi^k|^2+|\phi^k|^2\big) + \mathbb E_{T,k}(\mathcal R),
 \end{multline}
 where
\begin{equation*}
    \mathbb E_{T,k}(\mathcal R)\doteq \iint_{\mathcal R}Dr^2\big(|\partial_u\phi^k|+|\partial_v\phi^k|\big)|\Gamma^k\mathcal N|.
\end{equation*}
\end{prop}
\begin{proof}
Let $|\k|\le k$. We apply the multiplier identity \eqref{eq:main-energy-identity} to the Killing field $T=\frac 12(\partial_u+\partial_v)$ and integrate over $\mathcal R$ (with respect to $d\omega dudv$). For this vector field, the first three terms on the right-hand side of \eqref{eq:main-energy-identity} vanish and the final term gives rise to $\mathbb E_{T,k}(\mathcal R)$. Next we consider the term containing $\mathring{\slashed\Delta}$ and $\partial_u$. We integrate by parts on the sphere and then in $u$ to obtain the identity
\begin{multline*}
    \iint_{\mathcal R}DT^u\mathring{\slashed\Delta}\phi^\k\cdot\partial_u\phi\,d\omega du dv = -\iint_{\mathcal R}D \partial_u |\slashed\nabla\phi^\k|^2_\circ \,d\omega du dv \\ 
     = \iint_{\mathcal R}r^2\partial_uD|\slashed\nabla\phi^\k|^2 \,d\omega du dv - \int_\mathrm{VI}r^2D |\slashed\nabla \phi^\k|^2 \,d\omega dv+\int_\mathrm{II} r^2D |\slashed\nabla \phi^\k|^2 \,d\omega dv.
\end{multline*}
Similarly, for the term containing $\mathring{\slashed\Delta}$ and $\partial_v$, we obtain
\begin{multline*}
    \iint_{\mathcal R}D T^v \mathring{\slashed\Delta}\phi^\k\cdot\partial_v\phi\,d\omega du dv = -\iint_{\mathcal R}D \partial_v |\slashed\nabla\phi^\k|^2_\circ \,d\omega du dv \\ 
     = \iint_{\mathcal R} r^2\partial_vD|\slashed\nabla\phi^\k|^2 \,d\omega du dv - \int_\mathrm{VI}r^2D|\slashed\nabla \phi^\k|^2 \,d\omega du+\int_\mathrm{II} r^2D |\slashed\nabla \phi^\k|^2 \,d\omega du.
\end{multline*}
Since $\partial_vD=-\partial_uD$, the bulk terms cancel out and we arrive at 
 \begin{multline}\label{eq:T-energy-est-2}
\int_\mathrm{III}\big(r^2|\partial_u\phi^k|^2+r^2|\slashed\nabla\phi^k|^2\big) +\int_\mathrm{IV}\big(r^2|\partial_v\phi^k|^2+r^2|\slashed\nabla\phi^k|^2\big) +\int_\mathrm{V}\big(|\partial_u\phi^k|^2+D|\slashed\nabla\phi^k|^2\big) +\int_\mathrm{VI}\big(|\partial_v\phi^k|^2+D|\slashed\nabla\phi^k|^2\big)
  \\   \les \int_\mathrm{I}\big(|\partial_u\phi^k|^2+D|\slashed\nabla\phi^k|^2\big) + \int_\mathrm{II}\big(r^2|\partial_v\phi^k|^2+r^2|\slashed\nabla\phi^k|^2\big) + \mathbb E_{T,k}(\mathcal R).
 \end{multline}

To estimate the zeroth order term, we integrate the trivial identity
    \begin{equation*}
        \partial_u\big(\partial_v(( r-M)|\phi^\k|^2)\big)+\partial_v\big({-}\partial_u(( r-M)|\phi^\k|^2)\big)=0
    \end{equation*}
    over $\mathcal R$ and use the following estimates (which are easily derived using Young's inequality)
\begin{equation*}
 - r^2|\partial_u\phi^\k|^2 +D |\phi^\k|^2   \les -\partial_u\big(( r-M)|\phi^\k|^2\big)\les r^2|\partial_u\phi^\k|^2 +D |\phi^\k|^2,
\end{equation*}
\begin{equation*}
   - r^2|\partial_v\phi^\k|^2 +D |\phi^\k|^2   \les \partial_v\big(( r-M)|\phi^\k|^2\big)\les r^2|\partial_v\phi^\k|^2 +D |\phi^\k|^2,
\end{equation*}
to obtain
 \begin{multline}\label{eq:null-lagrangian}
\int_\mathrm{III}\big({-}r^2|\partial_u\phi^\k|^2+|\phi^\k|^2\big)\,d\omega du +\int_\mathrm{IV}\big({-}r^2|\partial_v\phi^\k|^2+|\phi^\k|^2\big)d\omega dv  
     +\int_\mathrm{V}\big({-}|\partial_u\phi^\k|^2+D|\phi^\k|^2\big)\, d\omega du\\ +\int_\mathrm{VI}\big({-}|\partial_v\phi^\k|^2+D|\phi^\k|^2\big)\, d\omega dv 
     \les \int_\mathrm{I}\big(|\partial_u\phi^\k|^2+D|\phi^\k|^2\big)\,d\omega du + \int_\mathrm{II}\big(r^2|\partial_v\phi^\k|^2+|\phi^\k|^2\big)d\omega dv.
 \end{multline}
  The estimate \eqref{eq:T-energy-est} now follows from adding a small multiple of \eqref{eq:null-lagrangian} to \eqref{eq:T-energy-est-2} and summing over $\k$.
\end{proof}

\subsection{Integrated local energy decay}\label{sec:Morawetz}
\subsubsection{The Morawetz estimate}

In this section, we prove the following \emph{Morawetz estimate} with degeneration at the photon sphere $r=2M$ and at the event horizon $r=M$.

\begin{prop}[The Morawetz estimate]\label{prop:Morawetz} Let $\phi$ be a solution of the wave equation \eqref{eq:wave-1} on the region $\mathcal R$ depicted in \cref{fig:butterfly}. For any $k\in\Bbb N_0$, it holds that
\begin{multline}\label{eq:Morawetz}
  \iint_{\mathcal R\cap \mathcal A}\big((r-M)^{1+\eta}(|\partial_u\phi^k|^2+|\partial_v\phi^k|^2)+(r-M)^3|\slashed\nabla\phi^k|^2+(r-M)^4|\phi^k|^2\big)  \\ + \iint_{\mathcal R\cap\mathcal B_{r_{-1},3M}}\big(|R\phi^k|^2+|\phi^k|^2\big)+\iint_{\mathcal R\cap \mathcal B_{3M,\infty}} \big(r^{1-\eta}(|\partial_u\phi^k|^2+|\partial_v\phi^k|^2)+r|\slashed\nabla\phi^k|^2+r^{-1-\eta}|\phi^k|^2\big)\\ \les \int_\mathrm{I}\big(|\partial_u\phi^k|^2+D|\slashed\nabla\phi^k|^2+D|\phi^k|^2\big) + \int_\mathrm{II}\big(r^2|\partial_v\phi^k|^2+r^2|\slashed\nabla\phi^k|^2+|\phi^k|^2\big) + \mathbb E_{T,k}(\mathcal R)+ \mathbb E_{Z,k}(\mathcal R),
\end{multline}
where $R\doteq \frac 12 (\partial_v-\partial_u)$ is the $\partial_{r_*}$ vector field in $(t,r_*)$-coordinates and
\begin{equation*}
    \mathbb E_{Z,k}(\mathcal R)\doteq \iint_{\mathcal R} D^2 r  |\phi^k||\mathcal N^k|.
\end{equation*}
\end{prop}

\begin{proof}[Sketch of proof] The Morawetz estimate for the linear wave equation on extremal Reissner--Nordstr\"om was first proved by Aretakis in \cite{Aretakis-instability-1}, but we outline here a proof of \eqref{eq:Morawetz} making use of observations by \cite{AAG20}, Apetroaie \cite{apetroaie}, and Holzegel--Mavrogiannis--Velozo Ruiz \cite{Morawetz-note}. 

First, we multiply \eqref{eq:wave-psi} by 
\begin{equation}\label{eq:X1}
    X_1\psi^\k \doteq - ( \partial_v f_1 - \partial_u f_1 ) \psi^\k - 2 f_1 (  \partial_v \psi^\k - \partial_u \psi^\k ) ,
\end{equation}
where 
\[f_1(r) \doteq (r-2M) \frac{\sqrt{r^2 + 4Mr - 4M^2}}{r^2}.\]
After integrating over $\mathcal R$, integrating by parts, and using the $T$-energy estimate \eqref{eq:T-energy-est} to control boundary terms, we obtain (see \cite[proof of Theorem 3]{Morawetz-note} for details)
\begin{equation}\label{mor:aux2}
     \iint_{\mathcal R} \left(D^{3/2}r^{-1}|R\phi^\k|^2+\left(1-\frac{2M}{r}\right)^2D^{3/2}r|\slashed\nabla\phi^\k|^2+D^2 r^{-2}|\phi^\k|^2\right)  
   \les \text{RHS of \eqref{eq:T-energy-est}} + \iint_{\mathcal R} Dr|X_1 \psi^\k||\mathcal{N}^\k|.
\end{equation}

Following \cite{apetroaie}, we can estimate the $T$ derivative by multiplying the wave equation \eqref{eq:wave-psi} by $h_1\psi^\k$, where 
\[ h_1 (r) \doteq - r^{-1} D^{3/2} ( 2 D^{1/2} -1 )^2 (1-D^{1/2} )^2.\]
 We also optimize the $r$-weight in \eqref{mor:aux2} by multiplying  \eqref{eq:wave-psi} by $X_2\psi^\k$ (and integrating by parts, etc.), which is defined in the same manner as \eqref{eq:X1} but with $f_1$ replaced by $f_2(r)\doteq (1-r^{-\delta})\chi_\mathrm{far}$, where $\chi_\mathrm{far}$ equals 1 on $\mathcal B_{3M,\infty}$ and vanishes on $\mathcal A$. Similarly, we optimize the weight on the $T$ derivative by multiplying \eqref{eq:wave-psi} by $h_2\psi^\k$ with $h_2(r)\doteq r^{-1-\delta}\chi_\mathrm{far}$ and integrating by parts. These modifications lead to the improved estimate
\begin{multline}\label{eq:Mor-intermediate-1}
 \iint_{\mathcal R} \left(D^{3/2}r^{1-\delta}|R\phi^\k|^2+\left(1-\frac{2M}{r}\right)^2D^{3/2}r^{1-\delta}|T\phi^\k|^2+\left(1-\frac{2M}{r}\right)^2D^{3/2}r|\slashed\nabla\phi^\k|^2+D^2 r^{-1-\delta}|\phi^\k|^2\right)   \\
   \les \text{RHS of \eqref{eq:T-energy-est}}+ \iint_{\mathcal R} Dr\big(|X_1 \psi^\k|+(|h_1|+|h_2|)|\psi^\k|\big)|\mathcal{N}^\k|.
\end{multline}

Finally, as in \cite{AAG20}, we improve the weights on the $R$- and $T$-derivatives in the near-horizon region $\mathcal R\cap \mathcal A'$ by integrating in that region the quantities $\chi_\mathrm{near}e^{g}\partial_u\psi^\k\cdot \partial_u\partial_v\psi^\k$ and $\partial_u\big(\chi_\mathrm{near}(r-M)^\delta |\partial_v\psi^\k|^2\big)$, where $g(r)\doteq (\log(r-M)^{-1})^{-1}$, and $\chi_\mathrm{near}$ equals $1$ on $\mathcal A'$ and vanishes on the complement of $\mathcal A$. 

By summing over $\k$, we have thus achieved
\begin{multline}\label{eq:mor-aux-4}
    \text{LHS of \eqref{eq:Morawetz}} \les \text{RHS of \eqref{eq:T-energy-est}} + \iint_{\mathcal R} Dr\big(|X_1 \psi^k|+|X_2\psi^k|+(|h_1|+|h_2|)|\psi^k|\big)|\mathcal{N}^k|\\ +\iint_{\mathcal R\cap\mathcal A}\big(e^{f_2}|\partial_u\psi^k|+D^{\delta/2}|\partial_v\psi^k|\big)|\mathcal{N}^k|.
\end{multline}
Using now the sequence of basic estimates $|\partial_u\psi^k|\les r|\partial_v\phi^k|+D|\phi^k|$, $|\partial_v\psi^k|\les r|\partial_u\phi^k|+D|\phi^k|$, $|f_1|\les 1$, $|Rf_1|\les Dr^{-3}$,  $|h_1|\les D^{3/2}r^{-3}$, $|Rf_2|\les |\chi_\mathrm{far}'|+ \chi_\mathrm{far}r^{-1-\delta}$, and $e^{g}\les 1$, we straightforwardly bound the bulk error on the right-hand side of \eqref{eq:mor-aux-4} by $\mathbb E_{T,k}(\mathcal R)+\mathbb E_{Z,k}(\mathcal R)$. This completes the sketch of the proof. \end{proof}

It is convenient to define
\begin{equation*}
    \mathbb M_k(\mathcal R)\doteq \int_\mathrm{I}\big(|\partial_u\phi^k|^2+D|\slashed\nabla\phi^k|^2+D|\phi^k|^2\big) + \int_\mathrm{II}\big(r^2|\partial_v\phi^k|^2+r^2|\slashed\nabla\phi^k|^2+|\phi^k|^2\big) + \mathbb E_{T,k}(\mathcal R)+ \mathbb E_{Z,k}(\mathcal R)
\end{equation*}
i.e., the right-hand side of the Morawetz estimate \eqref{eq:Morawetz}.

\subsubsection{Removing the degeneration at trapping}

In this section, we show how to estimate the remaining part of the integrated energy in the region $\mathcal B_{r_{-1},3M}$ at the expense of picking up a higher-order energy on the right-hand side of the estimate. 

\begin{prop}\label{prop:removing-trapping} Let $\phi$ be a solution of the wave equation \eqref{eq:wave-1} on the region $\mathcal R$ depicted in \cref{fig:butterfly}. For any $k\in\Bbb N_0$, it holds that
   \begin{equation*}
        \iint_{\mathcal R\cap\mathcal B_{r_{-1},3M}} \big(|\partial_u\phi^k|^2+|\partial_v\phi^k|^2+|\slashed\nabla\phi^k|^2\big)\les \Bbb M_{k+1}(\mathcal R)+\mathbb E_{\mathrm{trap},k}(\mathcal R),
   \end{equation*}
   where \begin{equation*}
       \mathbb E_{\mathrm{trap},k}(\mathcal R)\doteq \iint_{\mathcal R\cap\mathcal B_{r_{-2},4M}}|\phi^k||\mathcal N^k|,\quad r_{-2}\doteq M+\frac{r_{-1}-M}{2}.
   \end{equation*}
\end{prop}
\begin{proof} Let $|\k|\le k$ and let $\chi:\Bbb R\to [0,1]$ be a cutoff function such that $\chi(r)=1$ on $\mathcal B_{r_{-1},3M}$ and $\chi(r)=0$ outside of $\mathcal B_{r_{-2},4M}$. The radial vector field $R=\frac 12(\partial_v-\partial_u)$ satisfies the identity
\[|\partial_u\phi^\k|^2+|\partial_v\phi^\k|^2 = 4|R\phi^\k|^2 + 2\partial_u\phi^\k \cdot\partial_v\phi^\k.\]
Multiplying this identity by $\chi$, integrating over $\mathcal R$, turning $\slashed\nabla$ into an angular commutation, and using the Morawetz estimate at order $k+1$, we estimate
\begin{equation*}
      \iint_{\mathcal R} \chi\big(|\partial_u\phi^\k|^2+|\partial_v\phi^\k|^2+|\slashed\nabla\phi^\k|^2\big)\les \iint_\mathcal{R}\chi\big(|\phi^{k+1}|^2+|R\phi^\k|^2+\partial_u\phi^\k\cdot\partial_v\phi^\k\big)\les \mathbb M_{k+1}(\mathcal R)+ \iint_\mathcal{R}\chi\,\partial_u\phi^\k\cdot\partial_v\phi^\k. 
\end{equation*}
We integrate by parts in $u$ and use the wave equation to obtain
\begin{equation*}
 \iint_\mathcal{R}\chi\,\partial_u\phi^\k\cdot\partial_v\phi^\k= \iint_{\mathcal R}\big(\chi'D \phi^\k\cdot\partial_v\phi^\k-2\chi r^{-1}DR\phi^\k\cdot\phi^\k -\chi D\phi^\k\cdot\slashed\Delta\phi^\k+\chi D\phi^\k\cdot\mathcal N^\k\big).
\end{equation*}
Since $\chi'\ne 0$ where our Morawetz estimate is coercive, the first term is $\les\mathbb M_k(\mathcal R)$, the second term is $\les\mathbb M_k(\mathcal R)$ since both $R\phi^\k$ and $\phi^\k$ are controlled without degeneration, the third term can be integrated by parts on the sphere and then estimated by $\les\mathbb M_{k+1}(\mathcal R)$ after turning $\slashed\nabla$ into an angular commutation, and the third term gives rise to $\mathbb E_{\mathrm{trap},k}(\mathcal R)$. This completes the proof.
\end{proof}

\subsection{The \texorpdfstring{$\mathcal H^+$}{H +}-localized hierarchy}\label{sec:horizon-hierarchy}
In this section, we prove the $(r-M)^{-p}$-hierarchy of weighted energy estimates in the near region. For the linear wave equation on Reissner--Nordstr\"om, these estimates were introduced by Aretakis in \cite{Aretakis-instability-1} for $p=0,1$, and $2$, and then for $p\in[0,2]$ by the first-named author, Aretakis, and Gajic in \cite{angelopoulos2020late}.

\begin{prop}[The $(r-M)^{-p}$ hierarchy]\label{prop:horizon-hierarchy}
    Let $\phi$ be a solution of the wave equation \eqref{eq:wave-1} on the region $\mathcal R$ depicted in \cref{fig:butterfly}. For any $k\in\Bbb N_0$ and $p\in [\delta,2-\delta]$, it holds that 
    \begin{multline}\label{eq:r-M}
     \int_{\mathrm{V}}\big((r-M)^{-p}(|\partial_u\psi^k|^2+|\partial_u\phi^k|^2)+c_{p_\star}(r-M)^{-p_\star+2}|\phi^k|^2\big) +\int_\mathrm{VI}(r-M)^{-p+2}\big(|\slashed\nabla\phi^k|^2+|\phi^k|^2\big) \\ + \iint_{\mathcal R\cap\mathcal A}\big((r-M)^{-p+1}(|\partial_u\psi^k|^2+|\partial_u\phi^k|^2)+(r-M)^{-p+3}|\slashed\nabla\phi^k|^2+(r-M)^{-p+3}|\phi^k|^2\big)   \\ \les \int_\mathrm{V}(r-M)^{-p}\big(|\partial_u\psi^k|^2+|\partial_u\phi^k|^2\big)+\mathbb M_k(\mathcal R)+ \underline{\mathbb E}{}_{p,k}(\mathcal R),
    \end{multline}
    where
    \begin{equation*}
         \underline{\mathbb E}{}_{p,k}(\mathcal R)\doteq \iint_{\mathcal R_{\le\Lambda}} (r-M)^{-p+2} |\partial_u\psi^k||\mathcal N^k|. 
    \end{equation*}
\end{prop}
\begin{proof} Let $|\k|\le k$. Let $\chi:\Bbb R\to [0,1]$ be a cutoff function such that $\chi(r)=1$ for $r\le r_{-2}$ and $\chi(r)=0$ for $r\ge r_{-1}$. We evaluate the expression
    \begin{equation*}
        \iint_{\mathcal R}\partial_v \big(\chi(r)(r-M)^{-p}|\partial_u\psi^\k|^2\big) 
    \end{equation*}
    first by integrating by parts and then using the wave equation \eqref{eq:wave-psi}. This leads to the identity
    \begin{multline}\label{eq:horizon-proof-1}
        \int_\mathrm{V}\chi( r-M)^{-p}|\partial_u\psi^\k|^2 -  \int_\mathrm{I}\chi( r-M)^{-p}|\partial_u\psi^\k|^2 
        = -\iint_{\mathcal R} p\chi (r-M)^{-p-1}D |\partial_u\psi^\k|^2 \\
+\iint_{\mathcal R} \chi' D ( r-M)^{-p}|\partial_u\psi^\k|^2 -\iint_{\mathcal R}2\chi (r-M)^{-p}DD'\phi^\k\cdot\partial_u\psi^\k\\
+\iint_{\mathcal R}2\chi (r-M)^{-p}D\slashed\Delta\psi^\k\cdot\partial_u\psi^\k
-\iint_{\mathcal R} 2\chi (r-M)^{-p} rD\mathcal N^\k\cdot \partial_u\psi^\k.
    \end{multline}
The first integral on the right-hand side gives rise to the main bulk term in \eqref{eq:r-M}. The second integral is $\les \mathbb M_k(\mathcal R)$ since $\chi'$ is supported where \eqref{eq:Morawetz} controls $|\partial_u\psi^\k|^2$. The third term is estimated by applying Young's inequality, Morawetz, and Hardy; for details see the proof of Lemma 6.18 in \cite{AKU24}. The fourth integral gives good bulk and flux terms. Indeed, integrating by parts on the sphere and then in $u$, we have 
\begin{multline*}
    \iint_{\mathcal R}2\chi ( r-M)^{-p} D\slashed\Delta\psi^\k \cdot\partial_u\psi^\k  = - \iint_{\mathcal R} \chi ( r-M)^{-p} D\partial_u|\slashed\nabla\psi^\k|_\circ^2 \\
    \les \mathbb M_k(\mathcal R)+ \iint_{\mathcal R} \chi \partial_u(r^{-2}(r-M)^{-p+2})|\slashed\nabla\psi^\k|^2_\circ - \int_\mathrm{VI}\chi (r-M)^{-p}D|\slashed\nabla\psi^\k|^2.
\end{multline*} As $r\to M$, we have that $
    \partial_u(r^{-2}(r-M)^{-p+2})\sim (p-2)(r-M)^{-p+3},$
which gives the good angular bulk term in \eqref{eq:r-M}. Finally, the last term in \eqref{eq:horizon-proof-1} is clearly bounded by $\underline{\mathbb E}{}_{p,k}(\mathcal R)$. 

To obtain the zeroth order term on $\mathrm{VI}$, we use the fundamental theorem of calculus and the Morawetz estimate to write
\begin{multline*}
    \int_\mathrm{VI}\chi (r-M)^{-p+2}|\phi^k|^2 =\iint_\mathcal{R}\partial_u\big(\chi(r-M)^{-p+2}|\phi^k|^2\big)\les\mathbb M_k(\mathcal R) \\+ \iint_\mathcal{R}\chi\big((r-M)^{-p+1}|\partial_u\phi^k|^2+(r-M)^{-p+3}|\phi^k|^2\big).
\end{multline*}
    The proof is now completed by using \cref{prop:T-energy-est,prop:Morawetz} to remove the cutoff $\chi$ and applying \cref{lem:hardy-near} to estimate the remaining terms in \eqref{eq:r-M}.
\end{proof}

\subsection{The \texorpdfstring{$\mathcal I^+$}{I+}-localized hierarchy}\label{sec:rp-hierarchy}

In this section, we prove Dafermos and Rodnianski's hierarchy of $r^p$-weighted energy estimates in the far region \cite{dafermos2010new}. 

\begin{prop}[The $r^p$ hierarchy]\label{prop:rp-hierarchy} Let $\phi$ be a solution of \eqref{eq:wave-1} on the region $\mathcal R$ depicted in \cref{fig:butterfly}. For any $k\in\Bbb N_0$ and $p\in [\delta,2-\delta]$, it holds that 
    \begin{multline}\label{eq:rp-estimate}
\int_\mathrm{IV}\big(r^p|\partial_v\psi^k|^2+c_{p_\star}r^{p_\star+2}|\partial_v\phi^k|^2+c_{p_\star}r^{p_\star}|\phi^k|^2\big) +\int_\mathrm{III} \big(r^{p+2}|\slashed\nabla\phi^k|^2+r^p|\phi^k|^2\big) \\  + \iint_{\mathcal R_{\ge\Lambda}}\big(r^{p-1}|\partial_v\psi^k|^2+r^{p+1}|\partial_v\phi^k|^2+r^{p+1}|\slashed\nabla\phi^k|^2+r^{p-1}|\phi^k|^2\big)  \les \int_\mathrm{II}r^p|\partial_v\psi^k|^2 + \mathbb M_k(\mathcal R) + \mathbb E_{p,k}(\mathcal R),
    \end{multline}
    where
    \begin{equation*}
       \mathbb E_{p,k}(\mathcal R)\doteq \iint_{\mathcal R_{\ge\Lambda}} r^{p+1}|\partial_v\psi^k||\mathcal N^k|. 
    \end{equation*}
\end{prop}

\begin{proof} Let $|\k|\le k$. Let $\chi:\Bbb R\to [0,1]$ be a cutoff function such that $\chi(r)=0$ for $r\le \Lambda$ and $\chi(r)=1$ for $r\ge 2\Lambda$.  We evaluate the expression
    \begin{equation*}
        \iint_{\mathcal R}\partial_u\big(\chi(r) r^p |\partial_v\psi^\k|^2\big) d\omega du dv
    \end{equation*} first by integrating by parts in $u$ and then using the wave equation. This leads to the identity
    \begin{multline}\label{eq:rp-proof-1}
        \int_\mathrm{IV}\chi r^p|\partial_v\psi^\k|^2 - \int_\mathrm{II}\chi r^p|\partial_v\psi^\k|^2 = - \iint_{\mathcal R}p\chi r^{p-1}D|\partial_v\psi^\k|^2- \iint_{\mathcal R}\chi' D r^p |\partial_v\psi^\k|^2 \\ - \iint_{\mathcal R}2\chi r^p DD'\phi^\k\cdot\partial_v\psi^\k + \iint_{\mathcal R}2\chi r^{p-2}D \mathring{\slashed{\Delta}}\psi^\k\cdot\partial_v\psi^\k
        -\iint_{\mathcal R}2\chi r^{p+1}D \mathcal N^\k\cdot\partial_v\psi^\k.
    \end{multline}
    The first term on the right-hand side gives rise to the main bulk term in \eqref{eq:rp-estimate}. The second integral is $\les\mathbb M_k(\mathcal R)$ since $\chi'$ has bounded support. The third term can be absorbed by using Young's inequality, Morawetz, and Hardy, since $|D'|\les r^{-2}$; see the proof of Lemma 6.23 in \cite{AKU24} for details. The fourth term on the right-hand side of \eqref{eq:rp-proof-1} contributes good bulk and flux terms. Indeed, integrating by parts on the sphere and then in $v$, we have
 \begin{multline*}
         \iint_{\mathcal R}2\chi r^{p-2}D \mathring{\slashed{\Delta}}\psi^\k\cdot\partial_v\psi^\k= -\iint_{\mathcal R} \chi r^{p-2}D \partial_v|\slashed\nabla\psi^\k|^2_\circ   \\
     \les \mathbb M_k(\mathcal R) + \iint_{\mathcal R}\chi \partial_v(r^{p-2}D) |\slashed\nabla\psi^\k|^2_\circ - \int_\mathrm{III}\chi r^p D|\slashed\nabla\psi^k|^2_\circ.
    \end{multline*}
    As $r\to \infty$, we have that $\partial_v(r^{p-2}D)\sim (p-2)r^{p-3}$, which gives the good angular bulk term in \eqref{eq:rp-estimate}. Finally, the last term in \eqref{eq:rp-proof-1} is clearly bounded by $\mathbb E_{p,k}(\mathcal R)$. 
    
    To obtain the zeroth order term on $\mathrm{III}$, we use the fundamental theorem of calculus and the Morawetz estimate to write
    \begin{equation*}
        \int_\mathrm{III}\chi r^p|\phi^k|^2 = \iint_{\mathcal R}\partial_v\left(\chi r^p|\phi^k|^2\right)\les \mathbb M_k(\mathcal R) +\iint_{\mathcal R_{\ge\Lambda}}\big(r^{p+1}|\partial_v\phi^k|^2+r^{p-1}|\phi^k|^2\big).
    \end{equation*}
    The proof is now completed by using \cref{prop:T-energy-est,prop:Morawetz} to remove the cutoff $\chi$ and applying \cref{lem:hardy-far} to estimate the remaining terms in \eqref{eq:rp-estimate}. \end{proof}

\section{The main estimates}\label{sec:error-est}

In this section, we estimate the nonlinear bulk error terms $\mathbb E_{T,k}(\mathcal R)$, $\mathbb E_{Z,k}(\mathcal R)$, $\mathbb E_{\mathrm{trap},k}(\mathcal R)$, $\underline{\mathbb E}{}_{p,k}(\mathcal R)$, and $\mathbb E_{p,k}(\mathcal R)$ that arose in the previous section. These errors are estimated in the context of a continuity (bootstrap) argument, which is set up in \cref{sec:bs-set}. In \cref{sec:structure}, we further analyze the structure of the error terms and set up some more notation to make the organization of the estimates more transparent. In \cref{sec:L1-L2}, we prove the $L^1L^\infty$ and $L^2L^\infty$ estimates that are fundamental to our approach. In \cref{sec:pw}, we prove pointwise estimates for $u$- and $v$-derivatives of $\phi^{n-6}$. In \cref{sec:trapping}, we estimate $\mathbb E_{\mathrm{trap},k}(\mathcal R)$. In \cref{sec:T-error}, we estimate $\mathbb E_{T,k}(\mathcal R)$ and $\mathbb E_{Z,k}(\mathcal R)$. In \cref{sec:horizon-error}, we estimate $\underline{\mathbb E}{}_{p,k}(\mathcal R)$. Finally, in \cref{sec:rp-error}, we estimate $\mathbb E_{p,k}(\mathcal R)$.

\subsection{The bootstrap assumptions}\label{sec:bs-set}

We define the ``bootstrap set'' of times on which the bootstrap assumptions hold. Recall the integer $n\ge12$ which denotes the maximum number of commutations that we assume are controlled on the initial data hypersurface $\Sigma_0$.

\begin{defn}\label{def:bootstrap}
    Let $\ve_0>0$ and $A\ge 1$. Let $\mathfrak B(\mathring\phi,\ve_0, A)$ denote the set of $\tau_f\in[1,\infty)$ such that the solution $\phi$ to \eqref{eq:wave-1} with characteristic initial data $\mathring\phi$ exists and is smooth on $\mathcal D^{\tau_f}$ and satisfies the following estimates, where $
        \ve\doteq A\ve_0$: 
\begin{enumerate}
    \item \underline{Pointwise estimates}: It holds that 
    \begin{align}
 \label{eq:boot-pw-1}   |\partial_u\psi^{n-6}|+r|\partial_u\phi^{n-6}|    &\le \ve^{1/2},\\
  \label{eq:boot-pw-2}    r^{3/2-\delta}|\partial_v\psi^{n-6}|+ r^2|\partial_v\phi^{n-6}|    &\le \ve^{1/2}
    \end{align}
    in $\mathcal D^{\tau_f}\cap \{r\ge\Lambda\}$ and 
        \begin{align}
  \label{eq:boot-pw-4}     ( r-M)^{-3/2+\delta}|\partial_u\phi^{n-6}|   &\le \ve^{1/2},\\
\label{eq:boot-pw-5}        |\partial_v\phi^{n-6}|  &\le \ve^{1/2}
    \end{align}
    in $\mathcal D^{\tau_f}\cap \{r\le\Lambda\}$. 

\item \underline{Boundedness, decay, and growth for the master energies with degeneration at trapping}: Let $1\le\tau_1\le\tau_2\le\tau_f$. For $p\in \{0\}\cup[\delta,2-\delta]$, it holds that
\begin{equation}\label{eq:boot-energy-1}
   \mathcal X^{\tau_f}_{p,n-2}(\tau_1,\tau_2)\le \ve^2\tau^{-2+\delta+p}_1,
\end{equation}
and for $p\in\{0\}\cup[\delta,1+\delta]$, it holds that
\begin{align}\label{eq:boot-energy-2}
   \mathcal X^{\tau_f}_{p,n-1}(\tau_1,\tau_2) &\le \ve^2 \tau^{\max\{-1-\delta+p,-1\}}_1, \\
     \label{eq:boot-energy-3}
   \mathcal X^{\tau_f}_{p,n}(\tau_1,\tau_2)&\le \ve^2 \tau^{\max\{0,-1+3\delta+p\}}_2.
\end{align}

\item \underline{Integrated boundedness and decay at trapping}: Let $1\le\tau_1\le\tau_2\le\tau_f$. It holds that 
\begin{equation}
    \label{eq:boot-bulk-lower-order}  \mathcal Y^{\tau_f}_{n-3}(\tau_1,\tau_2)\le \ve^{3/2}\tau_1^{-2+\delta}.
\end{equation}
\end{enumerate}
\end{defn}

We will often simply write $\mathfrak B$ for $\mathfrak B(\mathring\phi,\ve_0,A)$. By standard local well-posedness arguments for the characteristic initial value problem as mentioned already in \cref{sec:loc}, we have:

\begin{prop}\label{lem:nonempty}
     For any $A\ge 1$ sufficiently large and $\ve_0$ sufficiently small, the following holds:
     \begin{enumerate}[(i)]
         \item The bootstrap set $\mathfrak B(\mathring\phi,\ve_0,A)$ is nonempty and connected.

              \item For any $\theta\in (0,1)$ and $T>1$, there exists an $\eta(\theta,T)>0$ such that if $\tau_f\in \mathfrak B(\mathring\phi,\ve_0,A)$ with $\tau_f\le T$ has the property that each of the estimates \eqref{eq:boot-pw-1}--\eqref{eq:boot-bulk-lower-order} hold on $\mathcal D^{\tau_f}$ with right-hand sides multiplied by $\theta\in(0,1)$, then $\tau_f+\eta\in \mathfrak B(\mathring\phi,\ve_0,A)$. 
     \end{enumerate}
\end{prop}

The main analytic content of this paper is contained in the following proposition, which verifies the assumption of part (ii) of the previous proposition, and whose proof will be given in \cref{sec:completing} below and relies on all of the estimates proved in the remainder of this section.

\begin{prop}\label{prop:improving}
    For any $A\ge 1$ sufficiently large, $\ve_0$ sufficiently small, and $\tau_f\in\mathfrak B$, the estimates \eqref{eq:boot-pw-1}--\eqref{eq:boot-bulk-lower-order} hold with the constant 1 on the right-hand side replaced by $\frac 12$.
\end{prop}

We conclude this section with the following immediate consequence of the bootstrap assumption \eqref{eq:boot-energy-1}:

\begin{lem}\label{prop:pointwise-1}
 For any $A\ge 1$,  $\ve_0$ sufficiently small, and $\tau_f\in \mathfrak B$, it holds that
  \begin{align}\label{eq:psi-pw}
      |\phi^{n-4}|&\les \ve r^{-1}\tau^{-1/2+\delta/2},\\
      |\slashed\nabla\phi^{n-5}|&\les \ve r^{-2}\tau^{-1/2+\delta/2}\label{eq:nabla-psi-pw}
  \end{align}
  on $\mathcal D^{\tau_f}$.
\end{lem}
\begin{proof} Arguing as in the proof of the estimate (7.13) in \cite[Proposition 7.4]{AKU24}, using also the angular Sobolev inequality \eqref{eq:angular-Sobolev} and the bootstrap assumption \eqref{eq:boot-energy-1}, we find\footnote{Note the worse power in \eqref{eq:psi-pw} compared to \cite[(7.13)]{AKU24}. This is because we can only take $p$ up to $2-\delta$ here, which limits the decay of the energies.}
\begin{equation*}
    |\psi^{n-4}|^2\les \big(\underline{\mathcal E}{}_{1/2,n-2}^{\tau_f}(\tau)\big)^{1/2}\big(\underline{\mathcal E}{}_{3/2,n-2}^{\tau_f}(\tau)\big)^{1/2}+\big({\mathcal E}{}_{1/2,n-2}^{\tau_f}(\tau)\big)^{1/2}\big({\mathcal E}{}_{3/2,n-2}^{\tau_f}(\tau)\big)^{1/2}\les \ve^2 \tau^{-1+\delta}
\end{equation*} Then \eqref{eq:nabla-psi-pw} follows immediately from \eqref{eq:ang-1}.
\end{proof}

\subsection{Structure of the nonlinear error terms}\label{sec:structure}

With our bootstrap assumptions at hand, we can now group the nonlinear terms according to their most important structural features. 

\begin{prop}\label{prop:structure} Let $\mathcal N$ be a nonlinearity satisfying the hypotheses of \cref{def:null-cond}. For any $n\ge 12$, $k\le n$, $A\ge 1$,  $\ve_0$ sufficiently small, and $\tau_f\in \mathfrak B$, it holds that
\begin{equation*}
    D|\mathcal N^k|\les |\mathfrak N^k_*| + |\mathfrak N^k_\mathrm{i}|+|\mathfrak N^k_\mathrm{ii}|+|\mathfrak N^k_\mathrm{iii}|+|\mathfrak N^k_\mathrm{iv}|+|\mathfrak N_\mathrm{v}^k|, 
\end{equation*} on $\mathcal D^{\tau_f}$, where
\begin{gather*}
         |\mathfrak N^k_*| \doteq  \mathbf 1_{\{r_0\le r\le \Lambda\}}|\partial^{\le 1}\phi^{n-6}||\partial^{\le 1}\phi^k|,\qquad
            |\mathfrak N^k_\mathrm{i}|\doteq \mathbf 1_{\{r\le r_0,r\ge \Lambda\}}|\partial_v\phi^{n-6}||\partial_u\phi^k|, \\
     |\mathfrak N^k_\mathrm{ii}|\doteq  \mathbf 1_{\{r\le r_0,r\ge \Lambda\}}|\partial_u\phi^{n-6}||\partial_v\phi^k|,\qquad
     |\mathfrak N^k_\mathrm{iii}|\doteq  \mathbf 1_{\{r\le r_0,r\ge \Lambda\}}\ve\tau^{-1/2+\delta/2}Dr^{-2}|\slashed\nabla\phi^k|,\\
     |\mathfrak N^k_\mathrm{iv}|\doteq  \mathbf 1_{\{r\le r_0,r\ge \Lambda\}} |\partial_u\phi^{n-6}||\partial_v\phi^{n-6}||\phi^{k}|,\qquad |\mathfrak N_\mathrm{v}^k|\doteq  \mathbf 1_{\{r\le r_0,r\ge \Lambda\}}\ve^2 \tau^{-1+\delta}Dr^{-4}|\phi^k|, 
\end{gather*} where $|\partial^{\le 1}\phi^k|\doteq |\phi^k| + |\partial_u\phi^k|+|\partial_v\phi^k|+|\slashed\nabla\phi^k|$.
At low orders, we will use instead
\begin{equation}\label{eq:N-near}
    D|\mathcal N^{n-6}|\les |\partial_u\phi^{n-6}||\partial_v\phi^{n-6}| + D|\slashed\nabla\phi^{n-6}|^2+ \mathbf 1_{\{r_0\le r\le \Lambda\}}|\partial^{\le 1}\phi^{n-6}|^2,
\end{equation}
i.e., we ignore the factors of $f$ in \eqref{eq:nlin-1} and \eqref{eq:nlin-2}.
\end{prop}

\begin{rk}
    While these estimates are clearly somewhat wasteful (for instance, if we commute $n-6$ times then a term like $\partial_u\phi^{n-6}\partial_v\phi^{n-6}$ cannot actually appear), this crude counting scheme suffices to close the bootstrap argument. 
\end{rk}

\begin{proof} We first consider the regions $r\le r_0$ and $r\ge \Lambda$. For nonlinearities of type \eqref{eq:nlin-1}, we iterate the product and chain rules and use \eqref{eq:f-assn} to estimate
\begin{equation*}
    D|\mathcal N^k|\les \sum_{\substack{k_1+k_2+k_3\le k \\ l_1m_1\le k_1}}|\phi^{l_1}|^{m_1}|\partial_u\phi^{k_2}||\partial_v\phi^{k_3}|.
\end{equation*}
Since $n\ge 12$, there can be at most one factor with $\ge n-5$ commutations, which means that 
\begin{equation*}
    \sum_{\substack{k_1+k_2+k_3\le k \\ l_1m_1\le k_1}}|\phi^{l_1}|^{m_1}|\partial_u\phi^{k_2}||\partial_v\phi^{k_3}| \les |\mathfrak N^k_\mathrm{i}|+|\mathfrak N^k_\mathrm{ii}|+|\mathfrak N^k_\mathrm{iv}|,
\end{equation*} where we used the bootstrap assumption \eqref{eq:psi-pw} to estimate all factors of $\phi$ with $\le n-6$ commutations by $\les 1$. For nonlinearities of type \eqref{eq:nlin-2}, by the same argument, we have
\begin{equation*}
    D|\mathcal N^k|\les \sum_{\substack{k_1+k_2+k_3\le k \\ l_1m_1\le k_1}}|\phi^{l_1}|^{m_1}|\slashed\nabla\phi^{k_2}||\slashed\nabla\phi^{k_3}|\les |\mathfrak N^k_\mathrm{iii}| + D|\slashed\nabla\phi^{n-6}|^2|\phi^k|\les |\mathfrak N^k_\mathrm{iii}| + |\mathfrak N^k_\mathrm{v}|,
\end{equation*}
where we used the bootstrap assumption \eqref{eq:nabla-psi-pw} to estimate $|\slashed\nabla\phi^{n-6}|^2$. Note that when $k=n-6$, we simply have 
\begin{equation*}
    |\mathfrak N^{n-6}_\mathrm{iii}| + D|\slashed\nabla\phi^{n-6}|^2|\phi^{n-6}|\les  D|\slashed\nabla\phi^{n-6}|^2,
\end{equation*}
which is what is used in \eqref{eq:N-near}. 

When $r_0\le r\le \Lambda$, each term in $\mathcal N$ is of the form \eqref{eq:nlin-0} and so we clearly have $|\mathcal N^k|\les |\partial^{\le 1}\phi^{n-6}||\partial^{\le 1}\phi^k|$ since $n\ge 12$. \end{proof}

To make the error estimates below more systematic, we define the following bulk error integrals, where $\mathrm j\in\{\mathrm{i},\mathrm{ii},\mathrm{iii},\mathrm{iv},\mathrm{v}\}$:
\begin{gather*}
    \mathbb E(Z,\mathrm{j},k,\mathcal R)\doteq \iint_{\mathcal R}rD|\phi^k||\mathfrak N^k_\mathrm{j}|\,d\omega dudv,\quad  \mathbb E(\partial_u,\mathrm{j},k,\mathcal R)\doteq    \iint_{\mathcal R} r^2|\partial_u\phi^k||\mathfrak N^k_\mathrm{j}|\,d\omega du dv,\\
     \mathbb E(\partial_v,\mathrm{j},k,\mathcal R)\doteq   \iint_{\mathcal R} r^2|\partial_v\phi^k||\mathfrak N^k_\mathrm{j}|\,d\omega du dv,\quad   \mathbb E((r-M)^{-p}\partial_u,\mathrm{j},k,\mathcal R)\doteq  \iint_{\mathcal R_{\le\Lambda}}(\bar r-M)^{-p}|\partial_u\psi^k||\mathfrak N^k_\mathrm{j}|\,d\omega dudv,\\
     \mathbb E(r^p\partial_v,\mathrm{j},k,\mathcal R) \doteq     \iint_{\mathcal R_{\ge\Lambda}} r^{p+1}|\partial_v\psi^k|| \mathfrak N^k_\mathrm{j}|\,d\omega dudv.
\end{gather*}

\subsection{\texorpdfstring{$L^1L^\infty$}{L 1 L infty} and \texorpdfstring{$L^2L^\infty$}{L 2 L infty} estimates}\label{sec:L1-L2}

Before we state and prove our fundamental $L^1L^\infty$ and $L^2L^\infty$ estimates, we prove a useful integrated decay estimate for $|\phi^{n-4}|^2$ and $|\partial\phi^{n-4}|^2$ along timelike hypersurfaces of constant area-radius.  

\begin{lem} \label{lem:little-bit-of-averaging}
   For any $A\ge 1$,  $\ve_0$ sufficiently small, $\tau_f\in \mathfrak B$, $\tau_1\in [1,\tau_f]$, and $R_0>M$, it holds that 
     \begin{equation}\label{eq:curve-integral}
         \int_{\{r=R_0\}\cap\{\tau\ge \tau_1\}}\big(|\partial_u\phi^{n-4}|^2+|\partial_v\phi^{n-4}|^2+|\phi^{n-4}|^2\big) \,d\omega d\tau\les_{R_0}  \ve^{3/2}\tau_1^{-2+\delta}.
\end{equation}
\end{lem}
\begin{proof}
    We give the proof only for $\partial_u\phi^{n-4}$ and in the case when $r<\Lambda$. The other cases only require small modifications to this argument.
    
    We first prove that
    \begin{equation*}
        \int_{\{r=R_0\}\cap\{1\le\tau\le 2\}}|\partial_u\phi^{n-4}|^2\,d\omega ds\les \ve^{3/2}.
    \end{equation*}
    Given $(u,v)\in \mathcal R_{\le\Lambda}$ with $\tau(u,v)\in [1,2]$, we use the fundamental theorem of calculus and the wave equation \eqref{eq:wave-phi-uv} to write
    \begin{equation*}
        |\partial_u\phi^{n-4}|^2(u,v,\omega) \les  |\partial_u\phi^{n-4}|^2(u,0,\omega) +\int_0^v \partial_u\phi^{n-4}\big(|\partial_u\phi^{n-4}|+|\partial_v\phi^{n-4}|+|\slashed\Delta\phi^{n-4}|+|\mathcal N^{n-4}|\big)\,dv'.
    \end{equation*}
    After integrating in $u$ and $\omega$ (note that $du\sim d\tau$ along $\{ r=R_0\}$ since $R_0>M$), the first term on the right-hand side is $\les \ve^2_0$ by the assumptions on initial data, and each of the quadratic terms are $\les \ve^{3/2}$ by \eqref{eq:boot-bulk-lower-order} and turning $\slashed\Delta\phi^{n-4}$ into $\slashed\nabla\phi^{n-3}$. Finally, to handle the nonlinear term, we use \cref{prop:structure} and the pointwise bootstrap assumptions to reduce again to the quadratic case. 

    We now assume that $\tau_1\ge 2$ and prove \eqref{eq:curve-integral} in this case. Let $\chi:\Bbb R_\tau \to[0,1]$ be a cutoff function such that $\chi(\tau)= 1$ for $\tau\ge \tau_1$, and $\chi(\tau)=0$ for $\tau\le \tau_1-1$. Note that $\partial_v\tau\sim 1$ in $\mathcal R_{\le\Lambda}$. Using the fundamental theorem of calculus again, we have
    \begin{equation*}
        |\partial_u\phi^{n-4}|^2(u,v,\omega) \les \int_{\tilde v}^{\bar v}|\partial_u\phi^{n-4}|^2 dv'+  \int_{\tilde v}^v \chi \partial_u\phi^{n-4}\big(|\partial_u\phi^{n-4}|+|\partial_v\phi^{n-4}|+|\slashed\Delta\phi^{n-4}|+|\mathcal N^{n-4}|\big)\,dv',
    \end{equation*}
    where $\tau(u,\bar v)=\tau_1$ and $\tau(u,\tilde v)=\tau_1-1$. After integrating in $u$ and $\omega$, and using \eqref{eq:boot-bulk-lower-order} again, each term is $\les \ve^{3/2} (\tau_1-1)^{-2+\delta}\sim \ve^{3/2}\tau_1^{-2+\delta}$, which completes the proof.
\end{proof}

\begin{prop}[$L^1_uL^\infty_{v,\omega}$ and $L^2_uL^\infty_{v,\omega}$ estimates]\label{prop:L1-Linfty}
   For any $A\ge 1$, $\ve_0$ sufficiently small, $\tau_f\in \mathfrak B$, and $\tau_1 \in [1,\tau_f]$, it holds that 
     \begin{align}
       \label{eq:partial-u-integrated}  \int_{\tau_1}^{\infty}\sup_{C^{\tau_f}(\tau)}|r\partial_u\phi^{n-6}|\,d\tau &\les \ve^{3/4}\tau_1^{-1/2+3\delta/4},\\
\label{eq:partial-u-integrated-2}
    \int_{\tau_1}^{\infty} \sup_{C^{\tau_f}(\tau)}|r\partial_u\phi^{n-6}|^2\,d\tau&\les \ve^{3/2} \tau_1^{-2+3\delta/2}.
    \end{align}
\end{prop}

\begin{proof} It suffices to prove \eqref{eq:partial-u-integrated} and \eqref{eq:partial-u-integrated-2} when $\tau_1=L_{i_0}$, where $i_0\in\Bbb N_0$ and $L_i\doteq 2^i$. By Cauchy--Schwarz, we have
    \begin{equation}\label{eq:L1-proof-1}
     \int_{L_{i_0}}^{\infty}\sup_{C^{\tau_f}(\tau)}|r\partial_u\phi^k|\,d\tau  \les \sum_{i\ge i_0}  \left(L_i\int_{L_i}^{L_{i+1}}\sup_{(v,\omega)\in[v^\Lambda(u),\Gamma^v(\tau_f)]\times S^2}|r\partial_u\phi^k|^2(\Gamma^u(\tau),v,\omega)\, d\tau\right)^{1/2}, 
    \end{equation} where $\tau\mapsto (\Gamma^u(\tau),\Gamma^v(\tau))$ is a parametrization of $\{r=\Lambda\}$ and $v^\Lambda(u)$ is defined implicitly by the relation $r(u,v^\Lambda(u))=\Lambda$. We use the fundamental theorem of calculus and the wave equation in the form \eqref{eq:wave-nu} to estimate
    \begin{equation*}
        |r\partial_u\phi^{n-6}|^2(u,v,\omega) \les |\partial_u\phi^{n-6}|^2(u,v^\Lambda(u),\omega)+\int_{v^\Lambda(u)}^{\Gamma^v(\tau_f)}|r\partial_u\phi^{n-6}|\big(|\partial_v\phi^{n-6}|+ r|\slashed\Delta\phi^{n-6}|+r|\mathcal N^{n-6}|\big)(u,v',\omega)\,dv'.
    \end{equation*}
Using \eqref{eq:boot-pw-1}, \eqref{eq:nabla-psi-pw}, and \eqref{eq:N-near}, we find
\begin{equation*}
    r^2 |\partial_u\phi^{n-6}||\mathcal N^{n-6}|\les \ve^{1/2}\big(r|\partial_u\phi^{n-6}||\partial_v\phi^{n-6}|+|\partial_u\phi^{n-6}||\slashed\nabla\phi^{n-6}|\big),
\end{equation*}
so that after integrating over $u\in [\Gamma^u(L_i),\Gamma^u(L_{i+1})]$, using the angular Sobolev inequality, \cref{lem:little-bit-of-averaging}, Cauchy--Schwarz, and the bootstrap assumption \eqref{eq:boot-energy-1}, we have
\begin{multline*}
       \int_{L_i}^{L_{i+1}}\sup_{C^{\tau_f}(\tau)}|r\partial_u\phi^{n-6}|^2\, d\tau\les \ve^{3/2}L_i^{-2+\delta}+ \int_{L_i}^{L_{i+1}}\int_{C_u^{\tau_f}}\big(r|\partial_u\phi^{n-4}||\partial_v\phi^{n-4}|+ r|\partial_u\phi^{n-4}||\slashed\nabla\phi^{n-3}|\big) \\
       \les \ve^{3/2}L_i^{-2+\delta}+\left(\int_{L_i}^{L_{i+1}}\int_{C_u^{\tau_f}}r^{1-\delta}|\partial_u\phi^{n-4}|^2\right)^{1/2}\left(\int_{L_i}^{L_{i+1}}\int_{C_u^{\tau_f}}r^{1+\delta}\big(|\partial_v\phi^{n-4}|^2+|\slashed\nabla\phi^{n-3}|^2\big)\right)^{1/2}\\
       \les \ve^{3/2}L_i^{-2+\delta} + \big(\mathcal X^{\tau_f}_{n-4,0}(L_i,L_{i+1})\big)^{1/2}\big(\mathcal X^{\tau_f}_{n-3,\delta}(L_i,L_{i+1})\big)^{1/2}\les \ve^{3/2}L_i^{-2+3\delta/2}.
\end{multline*}
Summing over $i\ge i_0$ yields \eqref{eq:partial-u-integrated-2} with $\tau_1=L_{i_0}$ and inserting this estimate into \eqref{eq:L1-proof-1} yields \eqref{eq:partial-u-integrated}.
  \end{proof}

\begin{prop}[$L^1_vL^\infty_{u,\omega}$ and $L^2_vL^\infty_{u,\omega}$ estimates]\label{prop:integrated-v-estimates}
    For any $A\ge 1$,  $\ve_0$ sufficiently small, $\tau_f\in \mathfrak B$, and $\tau_1 \in [1,\tau_f]$, it holds that 
\begin{align}
      \label{eq:partial-v-integrated}   \int_{\tau_1}^{\infty}\sup_{\underline C{}^{\tau_f}(\tau)}|\partial_v\phi^{n-6}|\,d\tau &\les \ve^{3/4}\tau_1^{-1/2+3\delta/4},\\
     \label{eq:partial-v-integrated-2}
         \int_{\tau_1}^{\infty}\sup_{\underline C{}^{\tau_f}(\tau)}|\partial_v\phi^{n-6}|^2\,d\tau &\les \ve^{3/2}\tau^{-2+3\delta/2}_1.
\end{align}
\end{prop}
\begin{proof} As in the proof of the previous proposition, we prove the equivalent estimates for $\tau_1=L_{i_0}$. Again, it suffices to consider a single dyadic interval and prove
\begin{equation}\label{eq:v-integral-help}
     \int_{L_i}^{L_{i+1}}\sup_{\underline C^{\tau_f}(\tau)}|\partial_v\phi^{n-6}|^2\, d\tau\les \ve^{3/2}L_i^{-2+3\delta/2}.
\end{equation}  We use the fundamental theorem of calculus and the wave equation in the form \eqref{eq:wave-lambda} to estimate
\begin{equation*}
        |\partial_v\phi^{n-6}|^2(u,v,\omega) \les |\partial_v\phi^{n-6}|^2(u^\Lambda(v),v,\omega)+\int_{u^\Lambda(v)}^{\Gamma^u(\tau_f)}D|\partial_v\phi^{n-6}|\big(|\partial_u\phi^{n-6}|+ |\slashed\Delta\phi^{n-6}|+|\mathcal N^{n-6}|\big)(u,v',\omega)\,dv',
    \end{equation*}
    where $u^\Lambda(v)$ is defined implicitly by the relation $r(u^\Lambda(v),v)=\Lambda$.
Using \eqref{eq:boot-pw-4}, \eqref{eq:boot-pw-5}, \eqref{eq:nabla-psi-pw}, and \eqref{eq:N-near}, we find
\begin{equation*}
    D|\partial_v\phi^{n-6}||\mathcal N^{n-6}|\les \ve^{1/2}\big((r-M)^{3/2-\delta}|\partial_v\phi^{n-6}|^2+(r-M)^2|\partial_v\phi^{n-6}||\slashed\nabla\phi^{n-6}|\big)+\ve^{1/2}\mathbf 1_{\{r\ge r_0\}}|\partial^{\le 1}\phi^{n-6}|^2,
\end{equation*}
so that after integrating over $v\in[\Gamma^v(L_i),\Gamma^v(L_{i+1})]$, using the angular Sobolev inequality, \cref{lem:little-bit-of-averaging}, Cauchy--Schwarz, and the bootstrap assumptions \eqref{eq:boot-energy-1} and \eqref{eq:boot-bulk-lower-order}, we have
\begin{multline*}
  \int_{L_i}^{L_{i+1}}\sup_{\underline C{}_v^{\tau_f}}|\partial_v\phi^{n-6}|^2\, dv\les \ve^{3/2}L_i^{-2+\delta}+\int_{L_i}^{L_{i+1}}\int_{\underline C{}_v^{\tau_f}\cap\mathcal B_{r_0,\Lambda}}|\partial^{\le 1}\phi^{n-3}|^2 \\+\int_{L_i}^{L_{i+1}}\int_{\underline C{}_v^{\tau_f}\cap\mathcal A}\big( (r-M)^2|\partial_u\phi^{n-4}||\partial_v\phi^{n-4}|+(r-M)^{3/2-\delta}|\partial_v\phi^{n-4}|^2+(r-M)^2|\partial_v\phi^{n-4}||\slashed\nabla\phi^{n-3}| \big) \\
  \les \ve^{3/2}L_i^{-2+\delta} + \mathcal Y^{\tau_f}_{n-3}(L_i,L_{i+1}) + \mathcal X^{\tau_f}_{0,n-4}(L_i,L_{i+1}) + \big(\mathcal X^{\tau_f}_{n-4,0}(L_i,L_{i+1})\mathcal X^{\tau_f}_{n-3,\delta}(L_i,L_{i+1})\big)^{1/2} \les \ve^{3/2}L_i^{-2+3\delta/2}.
\end{multline*}
This proves \eqref{eq:v-integral-help} and completes the proof of the proposition. 
\end{proof}

\subsection{Pointwise estimates}\label{sec:pw}

Using the integrated estimates from the previous section, we now prove the fundamental scale of pointwise estimates for $\partial_u\phi^{n-6}$ that interpolates between boundedness and growth. The boundedness estimate recovers the bootstrap assumption for $\partial_u\phi^{n-6}$.

\begin{prop}[$\partial_u$ pointwise estimates]\label{prop:pointwise-u}    For any $A\ge 1$,  $\ve_0$ sufficiently small, $\tau_f\in \mathfrak B$, and $q\in [3/2-\delta,2]$, it holds that 
\begin{equation}\label{eq:u-pw-new}
    (r-M)^{-q}|\partial_u\phi^{n-6}| \les \ve \tau^{(\frac{1+\delta}{1+2\delta})(q-3/2+\delta)}
\end{equation}
in $\mathcal D_{\tau_f}\cap\{r\le\Lambda\}$ and
\begin{equation}
  \label{eq:u-pw-3}  |\partial_u\psi^{n-6}|+r|\partial_u\phi^{n-6}|\les \ve
\end{equation}
in $\mathcal D_{\tau_f}\cap\{r\ge\Lambda\}$.
\end{prop}
\begin{proof} We use the method of characteristics in the $v$ direction. Let $q=3/2-\delta$ or $q=2$, $|\k|\le n-6$, and set $\Psi^\k_q\doteq D^{-q/2}\partial_u\psi^\k$. Using the wave equation, we have 
\begin{equation}\label{eq:pw-proof-1}
     |\partial_v\Psi_q^\k +\tfrac q2 D'\Psi^\k_q|\les D^{1-q/2}D'|\phi^{n-6}|+D^{1-q/2}|\slashed\nabla\phi^{n-5}|+ D^{1-q/2}r|\mathcal N^{n-6}|
\end{equation}
and by \eqref{eq:N-near}, it holds that 
\begin{equation*}
    D^{1-q/2}r |\mathcal N^{n-6}| \les  |\Psi_q^{n-6}||\partial_v\phi^{n-6}|  +D^{1-q/2}|\phi^{n-6}||\partial_v\phi^{n-6}| + D^{1-q/2}r|\slashed\nabla\phi^{n-6}|^2 + \mathbf 1_{\{r_0\le r\le \Lambda\}}|\partial^{\le 1}\phi^{n-6}|^2,
\end{equation*}
where we have used the trivial estimate $r|\partial_u\phi^{n-6}|\les |\partial_u\psi^{n-6}|+D|\phi^{n-6}|$. Since $qD'\ge 0$, \eqref{eq:pw-proof-1} can be integrated forwards from initial data, where $|\Psi_q^\k|\le \|\mathring\phi\|_\star\le \ve_0$ by assumption. After summing over $|\k|\le n-6$, this yields
\begin{multline*}
     |\Psi_q^{n-6}|\les \ve_0 + \int_0^v \big(|\partial_v\phi^{n-6}||\Psi^{n-6}_q|+ D^{1-q/2}D'|\phi^{n-6}|+D^{1-q/2}|\slashed\nabla\phi^{n-5}|\\ +D^{1-q/2}|\phi^{n-6}||\partial_v\phi^{n-6}| + D^{1-q/2}r|\slashed\nabla\phi^{n-6}|^2 + \mathbf 1_{\{r_0\le r\le \Lambda\}}|\partial^{\le 1}\phi^{n-6}|^2 \big)\,dv',
\end{multline*}
where the integral is taken with $(u,\omega)$ fixed. Since $\partial_v\phi^{n-6}$ is bounded in $L^1_v$  by \eqref{eq:boot-pw-2} and \eqref{eq:partial-v-integrated}, we use Gr\"onwall's inequality to estimate
\begin{multline}\label{eq:integral-master-partial-u}
     |\Psi_q^{n-6}|\les \ve_0 + \int_0^v \big( D^{1-q/2}D'|\phi^{n-6}|+D^{1-q/2}|\slashed\nabla\phi^{n-5}|\\ +D^{1-q/2}|\phi^{n-6}||\partial_v\phi^{n-6}| + D^{1-q/2}|\slashed\nabla\phi^{n-6}|^2 + \mathbf 1_{\{r_0\le r\le \Lambda\}}|\partial^{\le 1}\phi^{n-6}|^2 \big)\,dv'.
\end{multline}

To estimate the first term in the integral, we observe that by \eqref{eq:psi-pw},
\begin{equation*}
    \int_0^v D'|\phi^{n-6}|\,dv'\les \ve \int_0^v(r-M)r^{-4}\tau^{-1/2+\delta}\,dv',
\end{equation*}
and the integral is easily observed to be $O(1)$ by considering separately the cases $\tau\ge(r-M)^{-1}$ and $r\le\Lambda$, $\tau\le (r-M)^{-1}$ and $r\le \Lambda$, and $r\ge \Lambda$. Since $1-q/2\ge 0$, this shows that
\begin{equation*}
       \int_0^v D^{1-q/2}D'|\phi^{n-6}|\,dv'\les \ve. 
\end{equation*}
To estimate the linear angular derivative term in \eqref{eq:integral-master-partial-u}, we again use \eqref{eq:psi-pw} to estimate the integral (for both choices of $q$) in $\mathcal R_{\ge\Lambda}$ by $\les \ve$. In $\mathcal R_{\le\Lambda}$ with $q=3/2-\delta$, we use Cauchy--Schwarz, Sobolev, and the bootstrap assumption \eqref{eq:boot-energy-1} to estimate
\begin{align*}
     \int_0^v\mathbf 1_{\{r\le\Lambda\}} (r-M)^{1/2+\delta}|\slashed\nabla\phi^{n-5}|\,dv'&\les \left( \int_0^v\mathbf 1_{\{r\le\Lambda\}} (r-M)^{1+\delta}\,dv'\right)^{1/2}\left(\int_0^v \mathbf 1_{\{r\le\Lambda\}}(r-M)^\delta |\slashed\nabla\phi^{n-3}|^2 \right)^{1/2}\\
    &\les  \big(\mathcal F_{2-\delta,n-3}(u,1,\tau(u,v))\big)^{1/2}\les \ve
\end{align*}
since $(r-M)^{1+\delta}$ is integrable in $v$ along outgoing null cones in $\{r\le \Lambda\}$. For $q=2$, we instead directly integrate \eqref{eq:nabla-psi-pw} which gives 
\begin{equation*}
    \int_0^v |\slashed\nabla\phi^{n-5}|\,dv'\les \ve \tau(u,v)^{1/2+\delta/2}.
\end{equation*}
To estimate the fourth term in the integral, we use \eqref{eq:psi-pw} and \eqref{eq:partial-v-integrated} to estimate
\begin{equation*}
    \int_0^v D^{1-q/2}|\phi^{n-6}||\partial_v\phi^{n-6}|\,dv' \les \ve \int_0^v |\partial_v\phi^{n-6}|\,dv'\les \ve^{7/4}.
\end{equation*}
To estimate the fifth term in the integral for $q=3/2-\delta$, we again perform angular Sobolev to obtain 
\begin{equation*}
        \int_0^v D^{1/2+\delta}r|\slashed\nabla\phi^{n-6}|^2\,dv'\les \int_0^v D^{1/2+\delta}r|\slashed\nabla\phi^{n-4}|^2\,d\omega dv' \les \mathcal X^{\tau_f}_{3/2-\delta,n-4}(1,\infty)\les \ve^2.
\end{equation*}
For $q=2$, \eqref{eq:boot-energy-1} this estimate still works in the far region, but near the horizon we simply square \eqref{eq:nabla-psi-pw} and integrate it to obtain 
\begin{equation*}
    \int_0^vr|\slashed\nabla\phi^{n-6}|^2\,dv' \les  \ve^2 \tau(u,v)^{\delta}.
\end{equation*}
 The final term in \eqref{eq:integral-master-partial-u} can be shown to be $\les\ve^{3/2}$ by the same argument as in \cref{lem:little-bit-of-averaging}. 

Altogether, we have now shown that $|\Psi^{n-6}_{3/2-\delta}|\les \ve$ and $|\Psi^{n-6}_{2}|\les\ve \tau^{1/2+\delta/2}$, so by a simple pointwise interpolation, we prove \eqref{eq:u-pw-new}.
\end{proof}

\begin{prop}[$\partial_v$ pointwise estimates]\label{prop:pointwise-v}
    For any $A\ge 1$,  $\ve_0$ sufficiently small, and $\tau_f\in \mathfrak B$, it holds that 
      \begin{equation}\label{eq:v-improved}
          r^{3/2-\delta}|\partial_v\psi^{n-6}|+ r^2|\partial_v\phi^{n-6}|  \les  \ve
      \end{equation}
      in $\mathcal D_{\tau_f}$.
\end{prop}
\begin{proof}
We use the method of characteristics in the $u$ direction. Set $\Psi_*^\k=r^{3/2-\delta}\partial_v\psi^\k$. Using the wave equation, we have
\begin{equation}\label{eq:v-pw-1}
    |\partial_u\Psi^\k_*+(\tfrac 32-\delta)Dr^{-1}\Psi^\k_*|\les r^{3/2-\delta}\big(DD'|\phi^{n-6}|+D|\slashed\nabla\phi^{n-5}|+Dr|\mathcal N^{n-6}|\big),
\end{equation}
and by \eqref{eq:N-near}, it holds that
\begin{equation*}
    r^{3/2-\delta} Dr|\mathcal N^{n-6}|\les |\partial_u\phi^{n-6}||\Psi_*^{n-6}|+Dr^{3/2-\delta}|\phi^{n-6}||\partial_u\phi^{n-6}| + Dr^{5/2-\delta} |\slashed\nabla\phi^{n-6}|^2 + \mathbf 1_{\{r_0\le r\le \Lambda\}}|\partial^{\le 1}\phi^{n-6}|^2,
\end{equation*}
where we have used the trivial estimate $r^{5/2-\delta}|\partial_v\phi^{n-6}|\les|\Psi_*^{n-6}|+Dr^{3/2-\delta}|\phi^{n-6}|$. Since $(3-\delta/2)D\ge 0$, \eqref{eq:v-pw-1} can be integrated forwards from the initial data, where $|\Psi^{n-6}_*|\le \|\mathring\phi\|_\star\le \ve_0$ by assumption. After summing over $|\k|\le n-6$, this yields
\begin{multline}\label{eq:v-pw-2}
  |\Psi_*^{n-6}|\les \ve_0+\int_0^u\big( |\partial_u\phi^{n-6}||\Psi_*^{n-6}|+ DD'r^{3/2-\delta}|\phi^{n-6}|+Dr^{3/2-\delta}|\phi^{n-6}||\partial_u\phi^{n-6}|  \\  + Dr^{3/2-\delta}|\slashed\nabla\phi^{n-5}|+ Dr^{5/2-\delta}|\slashed\nabla\phi^{n-6}|^2 +\mathbf 1_{\{r_0\le r\le \Lambda\}}|\partial^{\le 1}\phi^{n-6}|^2 \big)\,du',
\end{multline}
where the integral is taken with $(v,\omega)$ fixed. 

The first term in the integral is absorbed by Gr\"onwall, since by \eqref{eq:boot-pw-4} and \eqref{eq:partial-u-integrated}, $\partial_u\phi^{n-6}$ is bounded in $L^1_u$. Next, using \eqref{eq:psi-pw}, we observe that
\begin{equation*}
    \int_0^u DD'r^{3/2-\delta}|\phi^{n-6}|\,du' \les \ve\int_0^u \mathbf 1_{\{r\le\Lambda\}} (r-M)^3\, du' + \ve\int_0^u \mathbf 1_{\{r\ge\Lambda\}}r^{-3/2-\delta}\, du'\les\ve 
\end{equation*}
as $(r-M)^3$ is integrable at the horizon and $r^{-3/2-\delta}$ is integrable at infinity. Using the bootstrap assumptions \eqref{eq:boot-pw-1} and \eqref{eq:boot-pw-4}, the third term in the integral is shown to be $\les\ve$ in the same manner. 

To estimate the linear angular derivative term in \eqref{eq:v-pw-2}, we use \eqref{eq:nabla-psi-pw} to argue as for the previous two terms and bound the contribution in the near region by $\les \ve$. In the far region, we use the angular Sobolev inquality, Cauchy--Schwarz, and the bootstrap assumption \eqref{eq:boot-energy-1} to obtain 
\begin{align*}
    \int_0^u r^{3/2-\delta}|\slashed\nabla\phi^{n-5}|\,du' &\les \left( \int_0^u r^{-1-\delta}\,du'\right)^{1/2}\left( \int_0^u r^{4-\delta}|\slashed\nabla\phi^{n-3}|^2\,d\omega du'\right)^{1/2}\\&\les \big(\underline{\mathcal F}{}_{2-\delta,n-3}^{\tau_f}(v,1,\tau(u,v))\big)^{1/2}\les \ve.
\end{align*}
The quadratic angular derivative term in \eqref{eq:v-pw-2} is handled identically at the horizon and in the far region is estimated by $\les \underline{\mathcal F}{}_{1/2-\delta,n-3}^{\tau_f}(v,1,\tau(u,v))\les\ve^2$. The final term in \eqref{eq:v-pw-2} can be shown to be $\les\ve^{3/2}$ by the same argument as in \cref{lem:little-bit-of-averaging}. \end{proof}

\subsection{Nonlinear error estimates in the intermediate \texorpdfstring{$r$}{r} region}\label{sec:trapping}

Given a region $\mathcal R$ as depicted in \cref{fig:butterfly}, we define 
\begin{equation*}
    \mathbb T_k(\mathcal R)\doteq \iint_{\mathcal R\cap\mathcal B_{r_{-2},\Lambda}} |\partial^{\le 1}\phi^k||\mathcal N^k|,
\end{equation*} where $|\partial^{\le 1}\phi^k|\doteq |\phi^k|+|\partial\phi^k|$.
This quantity controls the trapping error term $\mathbb E_{\mathrm{trap},k}(\mathcal R)$ from \cref{prop:removing-trapping} as well as all of the nonlinear errors from \cref{sec:structure} in the region $r_{-2}\le r\le\Lambda$. 

\begin{prop}\label{prop:trapping-1}
For any $A\ge 1$,  $\ve_0$ sufficiently small, $\tau_f\in \mathfrak B$, and region $\mathcal R\subset\mathcal D^{\tau_f}$ as depicted in \cref{fig:butterfly}, it holds that 
    \begin{equation}\label{eq:TT-estimate}
      \mathbb T_k(\mathcal R) \les \ve^{5/2}\begin{cases}
          \tau_1^{-2+\delta}  & \text{ if }k = n-2  \\
          \tau_1^{-1}  &  \text{ if }k = n-1\\
           1 &  \text{ if }k = n\\
        \end{cases}.
    \end{equation}
\end{prop}

\begin{rk}
Given our bootstrap assumptions, $ \mathbb T_k(\mathcal R)$ actually decays better than this. However, any extra decay here is wasted since it cannot be used to prove more decay for the master energy $\mathcal X_{p,k}$ (which is limited by the length of the $p$-hierarchy). We have therefore chosen to record suboptimal, but sufficient (and simpler) decay estimates for $\mathbb T_k(\mathcal R)$ and all subsequent error estimates in \cref{sec:T-error,sec:rp-error,sec:horizon-error}.
\end{rk}

\begin{proof}[Proof of \cref{prop:trapping-1}] Immediately from the definition of $\mathcal N$ and the assumption $n\ge 12$, we have
\begin{equation*}
    \mathbb T_k(\mathcal R)\les \iint_{\mathcal R\cap\mathcal B_{r_{-2},\Lambda}}|\partial^{\le 1}\phi^{n-6}||\partial^{\le 1}\phi^k|^2.
\end{equation*}
   The right-hand side contains several different combinations of terms which must be treated slightly differently. The general strategy is to put $|\partial^{\le 1}\phi^k|^2$ in an energy, estimate $\partial^{\le 1}\phi^{n-6}$ pointwise by a lower order energy, and then dyadically decompose to take advantage of the integrated decay of the lower order energy via the bootstrap assumption \eqref{eq:boot-bulk-lower-order}. Here we only prove
    \begin{equation}\label{eq:trap-help-2}
        \iint_{\mathcal R\cap \mathcal B_{r_{-1},\Lambda}} |\partial_u\phi^{n-6}||\partial_u\phi^k|^2 \les \text{RHS of \eqref{eq:TT-estimate}}
    \end{equation}
    as this displays the main ideas.

Let $(u,v,\omega)\in \mathcal R\cap \mathcal B_{r_{-1},\Lambda}$. By the fundamental theorem of calculus,
\begin{equation}\label{eq:trap-help-1}
    |\partial_u\phi^{n-6}|^2(u,v,\omega) \les |\partial_u\phi^{n-6}|^2(u^{\Lambda}(v),v,\omega)+\int_{u^{\Lambda}(v)}^u |\partial_u\phi^{n-6}||\partial_u^2\phi^{n-6}|(u',v,\omega) \,du',
\end{equation}
where $u^{\Lambda}(v)$ is defined by $r(u^{\Lambda}(v),v)=\Lambda$. Since $T=\frac 12(\partial_u+\partial_v)$, we have
\begin{equation*}
    |\partial_u^2\phi^{n-6}|\les |\partial_uT\phi^{n-6}|+ |\partial_u\partial_v\phi^{n-6}|\les |\partial_u\phi^{n-5}|+ |R\phi^{n-6}|+|\slashed\Delta\phi^{n-6}|+|\mathcal N^{n-6}|.
\end{equation*}
Using \cref{prop:pointwise-1}, \eqref{eq:N-near}, and the trivial estimate $|\partial_v\phi^{n-6}|\les |\partial_u\phi^{n-6}|+|\phi^{n-5}|$ (which follows from the fact that $T$ is a commutation vector field), we estimate $|\mathcal N^{n-6}|\les |\phi^{n-5}|+|\partial_u\phi^{n-6}|$, and hence obtain
\begin{equation*}
     |\partial_u^2\phi^{n-6}|\les |\partial_u\phi^{n-5}|+|\slashed\nabla\phi^{n-5}|+|\phi^{n-5}|.
\end{equation*}
Inserting this into the estimate \eqref{eq:trap-help-1}, using the angular Sobolev inequality, and the definition of $\underline{\mathcal E}{}_{0,k}^{\tau_f}$, we find 
\begin{equation*}
    |\partial_u\phi^{n-6}|(u,v,\omega)\les \left(\int_{S^2} |\partial_u\phi^{n-4}|^2(u^{\Lambda}(v),v,\omega)\,d\omega\right)^{1/2} + \big(\underline{\mathcal E}{}_{0,n-3}^{\tau_f}((\Gamma^v)^{-1}(v))\big)^{1/2}.
\end{equation*}
Now let $L_{i_0}\le\tau_1$. Using the previous estimate, Cauchy--Schwarz, \cref{lem:little-bit-of-averaging}, and \eqref{eq:boot-bulk-lower-order}, we obtain 
\begin{multline*}
    \iint_{\mathcal R\cap\mathcal B_{r_{-1},\Lambda}}|\partial_u\phi^{n-6}||\partial_u\phi^k|^2\les \sum_{i\ge i_0}\iint_{\mathcal R\cap\mathcal B_{r_{-1},\Lambda}\cap\{{L_i\le \tau\le L_{i+1}\}}}|\partial_u\phi^{n-6}||\partial_u\phi^k|^2 \\\les \sum_{i\ge i_0}\int_{L_i}^{L_{i+1}} \left(\left(\int_{S^2} |\partial_u\phi^{n-4}|^2|_\Gamma\,d\omega\right)^{1/2} + \big(\underline{\mathcal E}{}_{0,n-3}^{\tau_f}(\tau)\big)^{1/2}\right) \underline{\mathcal E}^{\tau_f}_{0,k}(\tau)\, d\tau \\
    \les \sum_{i\ge i_0} L_i^{1/2} \sup_{\tau\in[L_i,L_{i+1}]}\underline{\mathcal E}{}_{0,k}^{\tau_f}(\tau) \left(\left(\int_{\Gamma\cap\{L_i\le \tau\le L_{i+1}\}}|\partial_u\phi^{n-4}|^2 \,d\omega d\tau\right)^{1/2}+\left(\int_{L_i}^{L_{i+1}}\underline{\mathcal E}{}_{0,n-3}^{\tau_f}(\tau)\,d\tau\right)^{1/2}\right)\\
    \les \sum_{i\ge i_0} L_i^{1/2}\mathcal X^{\tau_f}_{0,k}(L_i,L_{i+1})\left(\ve^{3/4}L_i^{-1+\delta/2}+\big(\mathcal Y^{\tau_f}_{n-3}(L_i,L_{i+1})\big)^{1/2}\right) \les \ve^{3/4}\mathcal X^{\tau_f}_{0,k}(L_{i_0},\infty),
\end{multline*}
as the dyadic sum converges. Now \eqref{eq:boot-energy-1}--\eqref{eq:boot-energy-3} imply \eqref{eq:trap-help-2}. \end{proof}

\subsection{Degenerate energy and Morawetz nonlinear error estimates}\label{sec:T-error} 

We now estimate the error terms in the $T$-energy estimate \eqref{eq:T-energy-est} and Morawetz estimate \eqref{eq:Morawetz}. In the following, we simply write $\mathcal X_{p,k}^{\tau_f}$ for $\mathcal X_{p,k}^{\tau_f}(\tau_1,\tau_2)$. 

\begin{prop}\label{prop:T-error}   For any $A\ge 1$,  $\ve_0$ sufficiently small, $\tau_f\in \mathfrak B$, and region $\mathcal R\subset\mathcal D^{\tau_f}$ as depicted in \cref{fig:butterfly}, it holds that 
\begin{equation} \label{eq:T-error}
    \mathbb E_{T,k}(\mathcal R)+\mathbb E_{Z,k}(\mathcal R) \les \ve^{5/2} \begin{cases}
          \tau_1^{-2+\delta}  & \text{ if }k = n-2  \\
          \tau_1^{-1}  &  \text{ if }k = n-1\\
          1 &  \text{ if }k = n\\
        \end{cases}. 
\end{equation}
\end{prop}
\begin{proof} Unpacking our notation, it holds that
\begin{equation}\label{eq:T-error-1}
      \mathbb E_{T,k}(\mathcal R)+\mathbb E_{Z,k}(\mathcal R) \les \mathbb T_k(\mathcal R)+ \sum_{\mathrm j\in\{\mathrm{i},\mathrm{ii},\mathrm{iii},\mathrm{iv},\mathrm{v}\}}\big(\mathbb E(\partial_u,\mathrm j,k,\mathcal R)+\mathbb E(\partial_v,\mathrm j,k,\mathcal R)+\mathbb E(Z,\mathrm j,k,\mathcal R)\big).
\end{equation}
The term $\mathbb T_k(\mathcal R)$ was estimated in \cref{prop:trapping-1}, so we now proceed to estimate the remaining terms. 

    \textsc{Estimate for $\mathbb E(\partial_u,\mathrm{i},k,\mathcal R)$:} Using \eqref{eq:partial-v-integrated} and \eqref{eq:v-improved}, we obtain
    \begin{multline}\label{eq:T-error-2}
        \mathbb E(\partial_u,\mathrm{i},k,\mathcal R) = \iint_{\mathcal R\cap\mathcal A}|\partial_v\phi^{n-6}||\partial_u\phi^k|^2+\iint_{\mathcal R_{\ge\Lambda}}r^2|\partial_v\phi^{n-6}||\partial_u\phi^k|^2 \\ \les\sup_{\tau\in[\tau_1,\tau_2]}\underline{\mathcal E}{}^{\tau_f}_{0,k}(\tau) \int_{\tau_1}^{\tau_2}\sup_{\underline C^{\tau_f}(\tau)}|\partial_v\phi^{n-6}|\,d\tau +\sup_{\mathcal R_{\ge\Lambda}}\big(r^2|\partial_v\phi^{n-6}|\big)\iint_{\mathcal R_{\ge\Lambda}}|\partial_u\phi^k|^2 \les \ve^{3/4} \mathcal X^{\tau_f}_{0,k}.
    \end{multline}

    \textsc{Estimate for $\mathbb E(\partial_u,\mathrm{ii},k,\mathcal R)$:}  Using \eqref{eq:u-pw-new} with $q=3/2-\delta$, Cauchy--Schwarz, and \eqref{eq:partial-u-integrated-2}, we obtain
    \begin{multline}
        \mathbb E(\partial_u,\mathrm{ii},k,\mathcal R)=  \iint_{\mathcal R\cap\mathcal{A}} |\partial_u\phi^{n-6}||\partial_u\phi^k||\partial_v\phi^k|+ \iint_{\mathcal R_{\ge\Lambda}} r^2 |\partial_u\phi^{n-6}||\partial_u\phi^k||\partial_v\phi^k| \\ \les      \ve^{3/4}\iint_{\mathcal R\cap\mathcal{A}} (r-M)^{3/2-\delta}|\partial_u\phi^k||\partial_v\phi^k| + \left(\iint_{\mathcal R_{\ge\Lambda}}r^{1-\delta}|\partial_u\phi^k|^2\right)^{1/2} \left(\iint_{\mathcal R_{\ge\Lambda}}r^{1+\delta}|r\partial_u\phi^{n-6}|^2|\partial_v\phi^k|^2\right)^{1/2} \\ \les \ve^{3/4}\mathcal X^{\tau_f}_{0,k} + \big(\mathcal X^{\tau_f}_{0,k}\big)^{1/2}\left(\sup_{\tau\in[\tau_1,\tau_2]}\mathcal E_{0,k}^{\tau_f}(\tau)\int_{\tau_1}^{\tau_2}\sup_{C^{\tau_f}(\tau)}|r\partial_u\phi^{n-6}|^2\,d\tau\right)^{1/2}  \les  \ve^{3/4}\mathcal X^{\tau_f}_{0,k}.
    \end{multline}

\textsc{Estimate for $\mathbb E(\partial_u,\mathrm{iii},k,\mathcal R)$:} Using Cauchy--Schwarz, we obtain
\begin{multline}
    \mathbb E(\partial_u,\mathrm{iii},k,\mathcal R) = \iint_{\mathcal R\cap \mathcal A}\ve\tau^{-1/2+\delta/2}(r-M)^2|\slashed\nabla\phi^k||\partial_u\phi^k|+  \iint_{\mathcal R_{\ge\Lambda}} \ve\tau^{-1/2+\delta/2}|\slashed\nabla\phi^k||\partial_u\phi^k|\\ \le \ve\tau_1^{-1/2+\delta/2}\iint_{\mathcal R\cap\mathcal A}(r-M)^2|\slashed\nabla\phi^k||\partial_u\phi^k|+ \ve\tau_1^{-1/2+\delta/2}\iint_{\mathcal R_{\ge\Lambda}}|\slashed\nabla\phi^k||\partial_u\phi^k|\\
    \les \ve\tau_1^{-1/2+\delta/2}\big(\mathcal X_{\delta,k}^{\tau_f}\big)^{1/2}\big(\mathcal X_{0,k}^{\tau_f}\big)^{1/2}+ \ve\tau_1^{-1/2+\delta/2}\mathcal X_{0,k}^{\tau_f}.
\end{multline}

\textsc{Estimate for $\mathbb E(\partial_u,\mathrm{iv},k,\mathcal R)$:} Using \eqref{eq:u-pw-new} with $q=3/2-\delta$, \eqref{eq:u-pw-3}, \eqref{eq:v-improved}, Cauchy--Schwarz, and \eqref{eq:partial-v-integrated-2}, we obtain
\begin{multline}
    \mathbb E(\partial_u,\mathrm{iv},k,\mathcal R)= \iint_{\mathcal R\cap\mathcal A}|\partial_u\phi^{n-6}||\partial_v\phi^{n-6}||\phi^k||\partial_u\phi^k|+\iint_{\mathcal R_{\ge\Lambda}}r^2|\partial_u\phi^{n-6}||\partial_v\phi^{n-6}||\phi^k||\partial_u\phi^k| \\ \les \ve\iint_{\mathcal R\cap\mathcal A}(r-M)^{3/2-\delta}|\partial_v\phi^{n-6}||\phi^k||\partial_u\phi^k| +\ve^2 \iint_{\mathcal R_{\ge\Lambda}}r^{-1}|\phi^k||\partial_u\phi^k| \\
    \les \ve\left(\iint_{\mathcal R\cap\mathcal A}(r-M)^{3-2\delta}|\phi^k|^2\right)^{1/2}\left(\iint_{\mathcal R\cap\mathcal A}|\partial_v\phi^{n-6}|^2|\partial_u\phi^k|^2\right)^{1/2} + \ve^2 \mathcal X^{\tau_f}_{0,k}\\
    \les \ve^{7/4} \tau_1^{-1+3\delta/4}\big(\mathcal X_{2\delta,k}^{\tau_f}\big)^{1/2}\big(\mathcal X_{0,k}^{\tau_f}\big)^{1/2}+\ve^2 \mathcal X^{\tau_f}_{0,k}.
\end{multline}

\textsc{Estimate for $\mathbb E(\partial_u,\mathrm{v},k,\mathcal R)$:} Using Cauchy--Schwarz, we obtain
\begin{multline}
   \mathbb E(\partial_u,\mathrm{v},k,\mathcal R)\le \ve\tau_1^{-1+\delta}\iint_{\mathcal R\cap\mathcal A}(r-M)^2|\phi^k||\partial_u\phi^k| +\ve\tau_1^{-1+\delta}\iint_{\mathcal R\cap\mathcal A}r^{-2}|\phi^k||\partial_u\phi^k|  \\\les 
   \ve\tau_1^{-1+\delta}\left(\iint_{\mathcal R\cap\mathcal A}(r-M)^{1+\delta}|\partial_u\phi^k|^2\right)^{1/2}\left(\iint_{\mathcal R\cap\mathcal A}(r-M)^{3-\delta}|\phi^k|^2\right)^{1/2}+ \ve\tau_1^{-1+\delta}\mathcal X^{\tau_f}_{0,k}\\\les  \ve\tau_1^{-1+\delta}\big(\mathcal X_{0,k}^{\tau_f}\big)^{1/2}\big(\mathcal X_{\delta,k}^{\tau_f}\big)^{1/2} +\ve\tau_1^{-1+\delta}\mathcal X^{\tau_f}_{0,k}.
\end{multline}

\textsc{Estimate for $\mathbb E(\partial_v,\mathrm{i},k,\mathcal R)$:} Using Cauchy--Schwarz, \eqref{eq:v-improved}, and \eqref{eq:partial-v-integrated-2}, we obtain
\begin{multline}
    \mathbb E(\partial_v,\mathrm{i},k,\mathcal R)= \iint_{\mathcal R\cap\mathcal A} |\partial_v\phi^{n-6}||\partial_v\phi^k||\partial_u\phi^k| + \iint_{\mathcal R_{\ge\Lambda}}r^2 |\partial_v\phi^{n-6}||\partial_v\phi^k||\partial_u\phi^k| \\ \les \left(\iint_{\mathcal R\cap\mathcal A} (r-M)^{1+\delta}|\partial_v\phi^k|^2 \right)^{1/2}\left(\iint_{\mathcal R\cap\mathcal A} (r-M)^{-1-\delta}|\partial_v\phi^{n-6}|^2|\partial_u\phi^k|^2 \right)^{1/2} +  \ve\iint_{\mathcal R_{\ge\Lambda}}|\partial_v\phi^k||\partial_u\phi^k| \\ \les \ve^{3/4}\tau_1^{-1+3\delta/4} \big(\mathcal X_{0,k}^{\tau_f}\big)^{1/2}\big(\mathcal X^{\tau_f}_{1+\delta}\big)^{1/2} + \ve\mathcal X_{0,k}^{\tau_f}.
\end{multline}

\textsc{Estimate for $\mathbb E(\partial_v,\mathrm{ii},k,\mathcal R)$:} Using \eqref{eq:u-pw-new} with $q=3/2-\delta$ and \eqref{eq:partial-u-integrated}, we obtain
\begin{multline}
    \mathbb E(\partial_v,\mathrm{ii},k,\mathcal R) = \iint_{\mathcal R\cap\mathcal A}|\partial_u\phi^{n-6}||\partial_u\phi^k|^2+\iint_{\mathcal R_{\ge\Lambda}} r^2|\partial_u\phi^{n-6}||\partial_v\phi^k|^2 \\  \les \ve\iint_{\mathcal R\cap\mathcal A}(r-M)^{3/2-\delta}|\partial_v\phi^k|^2 + \sup_{\tau\in[\tau_1,\tau_2]}\mathcal E_{0,k}^{\tau_f}(\tau)\int_{\tau_1}^{\tau_2}\sup_{C^{\tau_f}(\tau)}|r\partial_u\phi^{n-6}|\,d\tau \les \ve^{3/4}\mathcal X_{0,k}^{\tau_f}.
\end{multline}

\textsc{Estimate for $\mathbb E(\partial_v,\mathrm{iii},k,\mathcal R)$:} Using Cauchy--Schwarz as in the estimate for $\mathbb E(\partial_u,\mathrm{iii},k,\mathcal R)$, we obtain
\begin{multline}
    \mathbb E(\partial_v,\mathrm{iii},k,\mathcal R)\le  \ve\tau^{-1/2+\delta/2}_1\iint_{\mathcal R\cap\mathcal A}(r-M)^2|\slashed\nabla\phi^k||\partial_v\phi^k|+\ve\tau^{-1/2+\delta/2}_1\iint_{\mathcal R_{\ge\Lambda}}|\slashed\nabla\phi^k||\partial_v\phi^k|  \\ \les \ve\tau_1^{-1/2+\delta/2}\big(\mathcal X^{\tau_f}_{\delta,k}\big)^{1/2}\big(\mathcal X^{\tau_f}_{0,k}\big)^{1/2}+  \ve\tau_1^{-1/2+\delta/2}\mathcal X_{0,k}^{\tau_f} .
\end{multline}

\textsc{Estimate for $\mathbb E(\partial_v,\mathrm{iv},k,\mathcal R)$:} Using \eqref{eq:u-pw-new} with $q=3/2-\delta$, \eqref{eq:v-improved}, Cauchy--Schwarz, and \eqref{eq:partial-v-integrated-2}, we obtain
\begin{multline}
    \mathbb E(\partial_v,\mathrm{iv},k,\mathcal R)= \iint_{\mathcal R\cap\mathcal A}|\partial_u\phi^{n-6}||\partial_v\phi^{n-6}||\phi^k||\partial_v\phi^k|+\iint_{\mathcal R_{\ge\Lambda}}r^2|\partial_u\phi^{n-6}||\partial_v\phi^{n-6}||\phi^k||\partial_v\phi^k| \\ \les \ve\iint_{\mathcal R\cap\mathcal A}(r-M)^{3/2-\delta/2}|\partial_v\phi^{n-6}||\phi^k||\partial_v\phi^k|+\ve^2\iint_{\mathcal R_{\ge\Lambda}} r^{-1}|\phi^k||\partial_u\phi^k| \\ \les \ve\left(\iint_{\mathcal R\cap\mathcal A}(r-M)^{1+\delta}|\partial_v\phi^k|^2\right)^{1/2} \left(\iint_{\mathcal R\cap\mathcal A}(r-M)^{2-2\delta}|\partial_v\phi^{n-6}|^2|\phi^k|^2\right)^{1/2}+\ve^2\mathcal X^{\tau_f}_{0,k}\\ \les \ve^{7/4}\tau_1^{-1+3\delta/4}\big(\mathcal X_{2\delta,k}^{\tau_f}\big)^{1/2}\big(\mathcal X_{0,k}^{\tau_f}\big)^{1/2}+\ve^2\mathcal X^{\tau_f}_{0,k}.
\end{multline}

\textsc{Estimate for $\mathbb E(\partial_v,\mathrm{v},k,\mathcal R)$:} Using Cauchy--Schwarz as in the estimate for $\mathbb E(\partial_u,\mathrm{v},k,\mathcal R)$, we obtain 
\begin{multline}
    \mathbb E(\partial_v,\mathrm{v},k,\mathcal R)\le  \ve^2\tau^{-1+\delta}_1\iint_{\mathcal R\cap\mathcal A}(r-M)^2|\phi^k||\partial_v\phi^k|+\ve\tau^{-1+\delta}_1\iint_{\mathcal R_{\ge\Lambda}}r^{-2}|\phi^k||\partial_v\phi^k|  \\ \les \ve^2\tau_1^{-1+\delta}\big(\mathcal X_{0,k}^{\tau_f}\big)^{1/2}\big(\mathcal X_{\delta,k}^{\tau_f}\big)^{1/2} +\ve^2\tau_1^{-1+\delta}\mathcal X^{\tau_f}_{0,k}.
\end{multline}

\textsc{Estimate for $\mathbb E(Z,\mathrm{i},k,\mathcal R)$:} Using Cauchy--Schwarz, \eqref{eq:v-improved}, and \eqref{eq:partial-v-integrated-2}, we obtain
\begin{multline}
    \mathbb E(Z,\mathrm{i},k,\mathcal R)=  \iint_{\mathcal R\cap\mathcal A} (r-M)^2|\partial_v\phi^{n-6}||\phi^k||\partial_u\phi^k|+\iint_{\mathcal R_{\ge\Lambda}} r|\partial_v\phi^{n-6}||\phi^k||\partial_u\phi^k| \\ \les \left(\iint_{\mathcal R\cap\mathcal A}(r-M)^4|\phi^k|^2\right)^{1/2}\left(\iint_{\mathcal R\cap\mathcal A}|\partial_v\phi^{n-6}|^2|\partial_u\phi^k|^2\right)^{1/2} + \ve \iint_{\mathcal R_{\ge\Lambda}}r^{-1}|\phi^k||\partial_u\phi^k|\les \ve\mathcal X^{\tau_f}_{0,k}.
\end{multline}

\textsc{Estimate for $\mathbb E(Z,\mathrm{ii},k,\mathcal R)$:} Using \eqref{eq:u-pw-new} with $q=3/2-\delta$, \eqref{eq:partial-u-integrated}, and Cauchy--Schwarz, we obtain
\begin{multline}
    \mathbb E(Z,\mathrm{ii},k,\mathcal R)=  \iint_{\mathcal R\cap\mathcal A} (r-M)^2|\partial_u\phi^{n-6}||\phi^k||\partial_v\phi^k|+\iint_{\mathcal R_{\ge\Lambda}} r|\partial_u\phi^{n-6}||\phi^k||\partial_v\phi^k|  \\ \les \ve\iint_{\mathcal R\cap\mathcal A}(r-M)^{7/2-\delta}|\phi^k||\partial_v\phi^k|+ \sup_{\tau\in[\tau_1,\tau_2]}\mathcal E_{0,k}^{\tau_f}(\tau)\int_{\tau_1}^{\tau_2}\sup_{C^{\tau_f}(\tau)}|r\partial_u\phi^{n-6}|\,d\tau\les   \ve^{3/4}\mathcal X_{0,k}^{\tau_f}.
\end{multline}

\textsc{Estimate for $\mathbb E(Z,\mathrm{iii},k,\mathcal R)$:} Using Cauchy--Schwarz, we obtain
\begin{multline}
    \mathbb E(Z,\mathrm{iii},k,\mathcal R)\le \ve^{1/2}\tau_1^{-1/2+\delta/2}  \iint_{\mathcal R\cap\mathcal A}(r-M)^4|\phi^k||\slashed\nabla\phi^k|+ \ve^{1/2}\tau_1^{-1/2+\delta/2} \iint_{\mathcal R_{\ge\Lambda}}r^{-1}|\phi^k||\slashed\nabla\phi^k|  \\ \les \ve\tau_1^{-1/2+\delta/2}\mathcal X^{\tau_f}_{0,k}.
\end{multline}

\textsc{Estimate for $\mathbb E(Z,\mathrm{iv},k,\mathcal R)$:} Using \eqref{eq:u-pw-new} with $q=3/2-\delta$, \eqref{eq:v-improved}, and \eqref{eq:u-pw-3}, we obtain
\begin{multline}
    \mathbb E(Z,\mathrm{iv},k,\mathcal R)=  \iint_{\mathcal R\cap\mathcal A}(r-M)^2|\partial_u\phi^{n-6}||\partial_v\phi^{n-6}||\phi^k|^2+\iint_{\mathcal R_{\ge\Lambda}}r^{1-\delta}|\partial_u\phi^{n-6}||\partial_v\phi^{n-6}||\phi^k|^2  \\ \les \ve^{2}\iint_{\mathcal R\cap\mathcal A}(r-M)^{7/2-\delta}|\phi^k|^2 +\ve^2 \iint_{\mathcal R_{\ge\Lambda}}r^{-2}|\phi^k|^2\les \ve^{2}\mathcal X^{\tau_f}_{0,k}.
\end{multline}

\textsc{Estimate for $\mathbb E(Z,\mathrm{v},k,\mathcal R)$:} Using the definition of $\mathcal X_{0,k}^{\tau_f}$ directly, we obtain
\begin{equation}\label{eq:T-error-final}
    \mathbb E(Z,\mathrm{v},k,\mathcal R)\le  \ve\tau_1^{-1+\delta}\iint_{\mathcal R\cap\mathcal A}(r-M)^4|\phi^k|^2+\ve\tau_1^{-1+\delta}\iint_{\mathcal R_{\ge\Lambda}}r^{-3}|\phi^k|^2 \les \ve\tau_1^{-1+\delta}\mathcal X_{0,k}^{\tau_f}.
\end{equation}

\textsc{Proof of \eqref{eq:T-error}:} Inserting the estimates \eqref{eq:T-error-2}--\eqref{eq:T-error-final} into \eqref{eq:T-error-1} yields 
\begin{multline*}
     \mathbb E_{T,k}(\mathcal R)+\mathbb E_{Z,k}(\mathcal R) \les \mathbb T_k(\mathcal R)+ \ve^{3/4} \mathcal X^{\tau_f}_{0,k} +\ve\tau_1^{-1/2+\delta/2}\big(\mathcal X^{\tau_f}_{\delta,k}\big)^{1/2}\big(\mathcal X^{\tau_f}_{0,k}\big)^{1/2} + \ve^{7/4} \tau_1^{-1+3\delta/4}\big(\mathcal X_{2\delta,k}^{\tau_f}\big)^{1/2}\big(\mathcal X_{0,k}^{\tau_f}\big)^{1/2}\\ +  \ve^{3/4}\tau_1^{-1+3\delta/4} \big(\mathcal X_{0,k}^{\tau_f}\big)^{1/2}\big(\mathcal X^{\tau_f}_{1+\delta}\big)^{1/2}.
\end{multline*}
Applying now \cref{prop:trapping-1} and the bootstrap assumptions \eqref{eq:boot-energy-1}--\eqref{eq:boot-energy-3}, we obtain an estimate which beats \eqref{eq:T-error} and hence completes the proof of the proposition.
\end{proof}

\subsection{\texorpdfstring{$(r-M)^{-p}$}{(r-M)-p} nonlinear error estimates}\label{sec:horizon-error}

We now estimate the $(r-M)^{-p}$ error $\underline{\mathbb E}{}_{p,k}(\mathcal R)$. These estimates require some care and slightly different arguments will have to be used for different ranges $p$. It will therefore be convenient to check the estimates for specific values of $p$ and then interpolate using the following general inequality for a nonnegative function $w$ on a measure space $(X,\mu)$ and $p_1<p<p_2$, which is immediately obtained from H\"older's inequality:
\begin{equation}\label{eq:interpolation}
        \int_X w^p\,d\mu\le\left(\int_X w^{p_1}\,d\mu\right)^\frac{p_2-p}{p_2-p_1}\left(\int_X w^{p_2}\,d\mu\right)^\frac{p-p_1}{p_2-p_1}.
    \end{equation}

\begin{prop}\label{prop:horizon-error}
 For any $A\ge 1$,  $\ve_0$ sufficiently small, $\tau_f\in \mathfrak B$, and region $\mathcal R\subset\mathcal D^{\tau_f}$ as depicted in \cref{fig:butterfly}, it holds that 
\begin{equation} \label{eq:horizon-error}
    \underline{\mathbb E}_{p,k}(\mathcal R)\les \ve^{5/2} \begin{cases}
         \tau_1^{-2+\delta+p}  & \text{if $k=n-2$ and $p\in [\delta,2-\delta]$}  \\
         \tau_1^{\max\{-1-\delta+p,-1\}}  &  \text{if $k=n-1$ and $p\in[\delta,1+\delta]$}\\
          \tau^{\max\{0,-1+3\delta+p\}}_2 &  \text{if $k=n$ and $p\in[\delta,1+\delta]$}\\
          \end{cases}.
\end{equation}    
\end{prop}
\begin{proof} Unpacking our notation, we have
\begin{equation}\label{eq:h-error-1}
      \underline{\mathbb E}_{p,k}(\mathcal R)\les \mathbb T_k(\mathcal R)+\sum_{\mathrm j\in\{\mathrm{i},\mathrm{ii},\mathrm{iii},\mathrm{iv},\mathrm{v}\}}\mathbb E(( r-M)^{-p}\partial_u,\mathrm{j},k,\mathcal R).
\end{equation}
As before, $\mathbb T_k(\mathcal R)$ was estimated in \cref{prop:trapping-1}, so we now proceed to estimate the remaining terms.

    \textsc{Estimate for $\mathbb E(( r-M)^{-p}\partial_u,\mathrm{i},k,\mathcal R)$:} 
    Using \eqref{eq:partial-v-integrated}, we obtain
\begin{multline}\label{eq:h-error-2}
    \mathbb E(( r-M)^{-p}\partial_u,\mathrm{i},k,\mathcal R)= \iint_{\mathcal R_{\le\Lambda}}(r-M)^{-p}|\partial_v\phi^{n-6}||\partial_u\phi^k||\partial_u\psi^k| \\ \les \sup_{\tau\in[\tau_1,\tau_2]}\underline{\mathcal E}{}_{p,k}^{\tau_f}(\tau) \int_{\tau_1}^{\tau_2}\sup_{\underline C^{\tau_f}(\tau)}|\partial_v\phi^{n-6}|\,d\tau \les \ve^{3/4} \mathcal X_{p,k}^{\tau_f}.
\end{multline}

 \textsc{Estimate for $\mathbb E(( r-M)^{-\delta}\partial_u,\mathrm{ii},k,\mathcal R)$:} Using \eqref{eq:u-pw-new} with $p=3/2-\delta$ and Cauchy--Schwarz as in the estimate for $\mathbb E(\partial_u,\mathrm{ii},k,\mathcal R)$, we obtain
 \begin{multline}\label{eq:h-error-3}
     \mathbb E(( r-M)^{-\delta}\partial_u,\mathrm{ii},k,\mathcal R) = \iint_{\mathcal R\cap \mathcal A}(r-M)^{-\delta}|\partial_u\phi^{n-6}||\partial_v\phi^k||\partial_u\psi^k| \\ \les \ve\iint_{\mathcal R\cap\mathcal A}(r-M)^{3/2-2\delta}|\partial_v\phi^k||\partial_u\psi^k| \les \ve \mathcal X_{0,k}^{\tau_f}.
 \end{multline}

\textsc{Estimate for $\mathbb E(( r-M)^{-1+3\delta}\partial_u,\mathrm{ii},k,\mathcal R)$:} Using \eqref{eq:u-pw-new} with $q=3/2-\delta$ and Cauchy--Schwarz, we obtain
\begin{multline}\label{eq:h-error-4}
    \mathbb E(( r-M)^{-1+3\delta}\partial_u,\mathrm{ii},k,\mathcal R)= \iint_{\mathcal R\cap\mathcal A}(r-M)^{-1+3\delta}|\partial_u\phi^{n-6}||\partial_v\phi^k||\partial_u\psi^k| \\ \les \ve\iint_{\mathcal R\cap\mathcal A}(r-M)^{1/2+2\delta}|\partial_v\phi^k||\partial_u\psi^k|\les \ve\left(\iint_{\mathcal R\cap \mathcal A}(r-M)^{1+\delta}|\partial_v\phi^{k}|^2\right)^{1/2} \left(\iint_{\mathcal R\cap \mathcal A}(r-M)^{3\delta}|\partial_u\psi^{k}|^2\right)^{1/2} \\ \les \ve\big(\mathcal X^{\tau_f}_{0,k}\big)^{1/2}\big(\mathcal X^{\tau_f}_{1-3\delta,k}\big)^{1/2}.
\end{multline}

\textsc{Estimate for $\mathbb E(( r-M)^{-1-\delta}\partial_u,\mathrm{ii},k,\mathcal R)$ with $k=n-1,n$:} This estimate is slightly anomalous. Recall the dyadic time steps $L_i=2^i$, set $\mathcal A_i\doteq \mathcal A \cap\{L_i\le\tau\le L_{i+1}\}$, and let $i_0=\lfloor \log_2\tau_1\rfloor$. Using \eqref{eq:u-pw-new} with $q=3/2+\delta$ and Cauchy--Schwarz, we obtain
\begin{multline*}
 \iint_{\mathcal R\cap\mathcal A_i}(r-M)^{-1-\delta}|\partial_u\phi^{n-6}||\partial_v\phi^k||\partial_u\psi^k| \les \ve L_i^{2\delta(\frac{1+\delta}{1+2\delta})}\iint_{\mathcal R\cap\mathcal A_i}(r-M)^{1/2}|\partial_v\phi^k||\partial_u\psi^k| \\\les \ve L_i^{2\delta(\frac{1+\delta}{1+2\delta})}\left(\iint_{\mathcal R\cap \mathcal A_i}(r-M)^{1+\delta}|\partial_v\phi^{k}|^2\right)^{1/2} \left(\iint_{\mathcal R\cap \mathcal A_i}(r-M)^{-\delta}|\partial_u\psi^{k}|^2\right)^{1/2} \\
 \les \ve L_i^{2\delta(\frac{1+\delta}{1+2\delta})} \big(\mathcal X^{\tau_f}_{0,k}(L_i,L_{i+1})\big)^{1/2}\big(\mathcal X^{\tau_f}_{1+\delta,k}(L_i,L_{i+1})\big)^{1/2}.
\end{multline*}
For $k=n-1$, we use the bootstrap assumption \eqref{eq:boot-energy-2} and sum over $i\ge i_0$ to obtain
\begin{equation}\label{eq:h-error-5}
    \mathbb E(( r-M)^{-1-\delta}\partial_u,\mathrm{ii},n-1,\mathcal R)\les \ve^{3} \tau_1^{-1/2+2\delta(\frac{1+\delta}{1+2\delta})},
\end{equation}
and for $k=n$, we use the bootstrap assumption \eqref{eq:boot-energy-3} and sum over $i_0\le i\le \lceil \log_2\tau_2\rceil$ to obtain
\begin{equation}\label{eq:h-error-extra}
    \mathbb E(( r-M)^{-1-\delta}\partial_u,\mathrm{ii},n,\mathcal R)\les \ve^{3}\tau_2^{2\delta +2\delta(\frac{1+\delta}{1+2\delta})}\les \ve^{3}\tau^{4\delta}_2.
\end{equation}

  \textsc{Estimate for $\mathbb E(( r-M)^{-2+\delta}\partial_u,\mathrm{ii},n-2,\mathcal R)$:} This estimate is slightly anomalous and we again decompose dyadically. Using \eqref{eq:u-pw-new} with $q=2$, Cauchy--Schwarz, and the bootstrap assumption \eqref{eq:boot-energy-1}, we obtain
  \begin{multline*}
     \iint_{\mathcal R\cap \mathcal A_i}(r-M)^{-2+\delta}|\partial_u\phi^{n-6}||\partial_v\phi^{n-2}||\partial_u\psi^{n-2}|  \les \ve^{3/4} L_i^{1/2+\delta/2}\iint_{\mathcal R\cap \mathcal A_i}(r-M)^{\delta}|\partial_v\phi^{n-2}||\partial_u\psi^{n-2}| \\\les \ve L_i^{1/2+\delta/2}\left(\iint_{\mathcal R\cap \mathcal A_i}(r-M)^{1+\delta}|\partial_v\phi^{n-2}|^2\right)^{1/2}\left(\iint_{\mathcal R\cap \mathcal A_i}(r-M)^{-1+\delta}|\partial_u\psi^{n-2}|^2\right)^{1/2}\\\les \ve L_i^{1/2+\delta/2}\big(\mathcal X^{\tau_f}_{0,n-2}(L_i,L_{i+1})\big)^{1/2}\big(\mathcal X^{\tau_f}_{2-\delta,n-2}(L_i,L_{i+1})\big)^{1/2}\les \ve^{3} L_i^{-1/2+\delta}.
  \end{multline*}
  Summing over $i\ge i_0$ yields
  \begin{equation}
      \label{eq:h-error-6}   \mathbb E(( r-M)^{-2+\delta}\partial_u,\mathrm{ii},n-2,\mathcal R) \les\ve^{3} \tau_1^{-1/2+\delta}.
  \end{equation}

    \textsc{Estimate for $\mathbb E(( r-M)^{-p}\partial_u,\mathrm{iii},k,\mathcal R)$:} Using Cauchy--Schwarz, we obtain
    \begin{multline}\label{eq:h-error-7}
        \mathbb E(( r-M)^{-p}\partial_u,\mathrm{iii},k,\mathcal R) \le \ve\tau^{-1/2+\delta/2}_1 \iint_{\mathcal R\cap\mathcal A}(r-M)^{-p+2}|\slashed\nabla\phi^k||\partial_u\psi^k| \\ \les \ve\tau_1^{-1/2+\delta/2}\iint_{\mathcal R\cap\mathcal A}\big((r-M)^{-p+1}|\partial_u\psi^k|^2+(r-M)^{-p+3}|\slashed\nabla\phi^k|^2\big) \les  \ve\tau^{-1/2+\delta/2}_1 \mathcal X_{p,k}^{\tau_f}.
    \end{multline}

\textsc{Estimate for $\mathbb E(( r-M)^{-p}\partial_u,\mathrm{iv},k,\mathcal R)$:} This estimate is slightly anomalous. By revisiting the proof of \cref{prop:integrated-v-estimates} (simply multiply \eqref{eq:v-integral-help} by $L_i^{2\delta}$ before summing), we find the alternate estimate
\begin{equation*}
    \int_{\tau_1}^\infty \tau^{2\delta} \sup_{\underline C^{\tau_f}(\tau)}|\partial_v\phi^{n-6}|^2 \,d\tau \les \ve^{3/2}\tau_1^{-2+7\delta/2}.
\end{equation*}
Using \eqref{eq:u-pw-new} with $q=3/2$, Cauchy--Schwarz, and our new integrated estimate, we obtain
\begin{multline}\label{eq:h-error-8}
        \mathbb E(( r-M)^{-p}\partial_u,\mathrm{iv},k,\mathcal R) \les  \iint_{\mathcal R\cap\mathcal A}( r-M)^{-p}|\partial_u\phi^{n-6}||\partial_v\phi^{n-6}||\phi^k||\partial_u\psi^k|\\ \les \ve \iint_{\mathcal R\cap\mathcal A}(r-M)^{-p+3/2}\tau^{\delta(\frac{1+\delta}{1+2\delta})}|\partial_v\phi^{n-6}||\phi^k||\partial_u\psi^k| \les \ve \left(\iint_{\mathcal R\cap\mathcal A}(r-M)^{-p+3}|\phi^k|^2\right)^{1/2}\\\cdot\left(\iint_{\mathcal R\cap\mathcal A}(r-M)^{-p}\tau^{2\delta(\frac{1+\delta}{1+2\delta})}|\partial_v\phi^{n-6}|^2|\partial_u\psi^k|^2\right)^{1/2} \les\ve^{7/4}\tau_1^{-1+7\delta/4}\mathcal X^{\tau_f}_{p,k}.
    \end{multline}

\textsc{Estimate for $\mathbb E(( r-M)^{-p}\partial_u,\mathrm{v},k,\mathcal R)$:} Using Cauchy--Schwarz, we obtain
\begin{equation}\label{eq:h-error-fin}
        \mathbb E(( r-M)^{-p}\partial_u,\mathrm{v},k,\mathcal R) \les  \ve\tau_1^{-1+\delta}\iint_{\mathcal R_{\le\Lambda}} (r-M)^{-p+2}|\phi^k||\partial_u\psi^k| \les \ve\tau_1^{-1+\delta}\mathcal X_{p,k}^{\tau_f}.
    \end{equation}

\textsc{Proof of \eqref{eq:horizon-error}:} Using \cref{prop:trapping-1}, \eqref{eq:h-error-1}--\eqref{eq:h-error-fin}, and the bootstrap assumptions \eqref{eq:boot-energy-1}--\eqref{eq:boot-energy-3}, we obtain
\begin{gather*}
    \underline{\mathbb E}_{\delta,n-2}(\mathcal R)\les \ve^{5/2}\tau_1^{-2+2\delta},\quad  \underline{\mathbb E}_{2-\delta,n-2}(\mathcal R)\les \ve^{5/2},\\
    \underline{\mathbb E}_{\delta,n-1}(\mathcal R)\les \ve^{5/2}\tau_1^{-1},\quad \underline{\mathbb E}_{1+\delta,n-1}(\mathcal R)\les \ve^{5/2},\\
    \underline{\mathbb E}_{\delta,n}(\mathcal R)\les \ve^{5/2},\quad \underline{\mathbb E}_{1-3\delta,n}(\mathcal R)\les \ve^{5/2},\quad \underline{\mathbb E}_{1+\delta,n}(\mathcal R)\les \ve^{5/2}\tau_2^{4\delta}.
\end{gather*}
Using \eqref{eq:interpolation} to interpolate, we obtain from this exactly \eqref{eq:horizon-error}.
\end{proof}

\subsection{\texorpdfstring{$r^p$}{r p} nonlinear error estimates}\label{sec:rp-error}

Finally, we estimate the $r^p$ error $\mathbb E_{p,k}(\mathcal R)$. 

\begin{prop}\label{prop:rp-error}   For any $A\ge 1$,  $\ve_0$ sufficiently small, $\tau_f\in \mathfrak B$, and region $\mathcal R$ as depicted in \cref{fig:butterfly}, it holds that 
\begin{equation}\label{eq:rp-error}
  \mathbb E_{p,k}(\mathcal R)\les   \ve^{5/2} \begin{cases}
         \tau_1^{-2+\delta+p}  & \text{if $k=n-2$ and $p\in[\delta,2-\delta]$}  \\
         \tau_1^{\max\{-1-\delta+p,-1\}}  &  \text{if $k=n-1$ and $p\in[\delta,1+\delta]$}\\
          \tau^{\max\{0,-1+3\delta+p\}}_2 &  \text{if $k=n$ and $p\in[\delta,1+\delta]$}\\
        \end{cases}.
\end{equation}
\end{prop}
\begin{proof} Unpacking our notation, it holds that
\begin{equation}\label{eq:rp-error-proof-1}
      {\mathbb E}_{p,k}(\mathcal R)\les \sum_{\mathrm j\in\{\mathrm{i},\mathrm{ii},\mathrm{iii},\mathrm{iv},\mathrm{v}\}}\mathbb E(r^p\partial_v,\mathrm{j},k,\mathcal R).
\end{equation}
We now estimate each of the terms.

\textsc{Estimate for $\mathbb E(r^\delta\partial_v,\mathrm{i},k,\mathcal R)$:} Using \eqref{eq:v-improved} and Cauchy--Schwarz, we obtain
\begin{multline}\label{eq:rp-error-proof-2}
  \mathbb E(r^\delta\partial_v,\mathrm{i},k,\mathcal R)=      \iint_{\mathcal R_{\ge\Lambda}} r^{1+\delta}|\partial_v\phi^{n-6}||\partial_v\psi^{k}||\partial_u\phi^{k}|\les \ve\iint_{\mathcal R_{\ge\Lambda}}r^{-1+\delta}|\partial_v\psi^{k}||\partial_u\phi^{k}| \\ \les \ve\left(\iint_{\mathcal R_{\ge\Lambda}}r^{1-\delta}|\partial_u\phi^{k}|^2\right)^{1/2}\left(\iint_{\mathcal R_{\ge\Lambda}}r^{-3+3\delta}|\partial_v\psi^k|^2\right)^{1/2}\les \ve \mathcal X^{\tau_f}_{0,k}.
\end{multline}

\textsc{Estimate for $\mathbb E(r^{1+\delta}\partial_v,\mathrm{i},k,\mathcal R)$:}  Using \eqref{eq:v-improved} and Cauchy--Schwarz, we obtain
\begin{multline}\label{eq:rp-error-proof-3}
     \mathbb E(r^{1+\delta}\partial_v,\mathrm{i},k,\mathcal R)=      \iint_{\mathcal R_{\ge\Lambda}} r^{2+\delta}|\partial_v\phi^{n-6}||\partial_v\psi^{k}||\partial_u\phi^{k}|\les \ve\iint_{\mathcal R_{\ge\Lambda}} r^{\delta}|\partial_v\psi^{k}||\partial_u\phi^{k}| \\ \les \ve \left(\iint_{\mathcal R_{\ge\Lambda}}r^{1-\delta}|\partial_u\phi^{k}|^2\right)^{1/2}\left(\iint_{\mathcal R_{\ge\Lambda}}r^{-1+3\delta}|\partial_v\psi^k|^2\right)^{1/2}\les \ve \big( \mathcal X^{\tau_f}_{0,k}\big)^{1/2} \big(\mathcal X^{\tau_f}_{3\delta,k}\big)^{1/2}.
\end{multline}

\textsc{Estimate for $\mathbb E(r^{2-\delta}\partial_v,\mathrm{i},k,\mathcal R)$:}  Using \eqref{eq:v-improved} and Cauchy--Schwarz, we obtain
\begin{multline}\label{eq:rp-error-proof-4}
     \mathbb E(r^{2-\delta}\partial_v,\mathrm{i},k,\mathcal R)=      \iint_{\mathcal R_{\ge\Lambda}} r^{3-\delta}|\partial_v\phi^{n-6}||\partial_v\psi^{k}||\partial_u\phi^{k}|\les \ve\iint_{\mathcal R_{\ge\Lambda}} r^{1-\delta}|\partial_v\psi^{k}||\partial_u\phi^{k}| \\ \les \ve \left(\iint_{\mathcal R_{\ge\Lambda}}r^{1-\delta}|\partial_u\phi^{k}|^2\right)^{1/2}\left(\iint_{\mathcal R_{\ge\Lambda}}r^{1-\delta}|\partial_v\psi^k|^2\right)^{1/2}\les \ve \big( \mathcal X^{\tau_f}_{0,k}\big)^{1/2} \big(\mathcal X^{\tau_f}_{2-\delta,k}\big)^{1/2}.
\end{multline}

       \textsc{Estimate for $\mathbb E(r^p\partial_v,\mathrm{ii},k,\mathcal R)$:} Using \eqref{eq:u-pw-3}, the trivial estimate $|\partial_v\phi^k|\les r^{-1}|\partial_v\psi^k|+r^{-1}|\phi^k|$, and Young's inequality, we obtain
       \begin{equation}\label{eq:rp-error-proof-5}
           \mathbb E(r^p\partial_v,\mathrm{ii},k,\mathcal R)= \iint_{\mathcal R_{\ge\Lambda}} r^{p+1}|\partial_u\phi^{n-6}||\partial_v\psi^k||\partial_v\phi^k| \les \ve\iint_{\mathcal R_{\ge\Lambda}} r^{p-1}\big(|\partial_v\psi^k|^2+|\phi^k|^2\big)\les\ve\mathcal X^{\tau_f}_{p,k}.
       \end{equation}

 \textsc{Estimate for $\mathbb E(r^p\partial_v,\mathrm{iii},k,\mathcal R)$:} 
Using Cauchy--Schwarz, we obtain
 \begin{equation}\label{eq:rp-error-proof-6}
     \mathbb E(r^p\partial_v,\mathrm{iii},k,\mathcal R) = \ve\tau_1^{-1/2+\delta/2} \iint_{\mathcal R_{\ge\Lambda}} r^{p-1}|\slashed\nabla\phi^k||\partial_v\psi^k|  \les \ve\tau_1^{-1/2+\delta/2} \mathcal X^{\tau_f}_{p,k}.
 \end{equation}

  \textsc{Estimate for $\mathbb E(r^p\partial_v,\mathrm{iv},k,\mathcal R)$:} Using \eqref{eq:v-improved} and \eqref{eq:u-pw-3}, we obtain
  \begin{equation}\label{eq:rp-error-proof-7}
      \mathbb E(r^p\partial_v,\mathrm{iii},k,\mathcal R) = \iint_{\mathcal R_{\ge\Lambda}}r^{p+1}|\partial_u\phi^{n-6}||\partial_v\phi^{n-6}||\phi^k||\partial_v\psi^k|\les \ve^2\iint_{\mathcal R_{\ge\Lambda}}r^{p-2}|\phi^k||\partial_v\psi^k|\les \ve^2\mathcal X^{\tau_f}_{p,k}.
  \end{equation}

   \textsc{Estimate for $\mathbb E(r^p\partial_v,\mathrm{v},k,\mathcal R)$:} 
Using Cauchy--Schwarz, we obtain
 \begin{equation}\label{eq:rp-error-proof-8}
     \mathbb E(r^p\partial_v,\mathrm{iii},k,\mathcal R) = \ve^2\tau_1^{-1+\delta} \iint_{\mathcal R_{\ge\Lambda}} r^{p-3}|\phi^k||\partial_v\psi^k|  \les  \ve^2\tau_1^{-1+\delta}\mathcal X^{\tau_f}_{p,k}.
 \end{equation}

 \textsc{Proof of \eqref{eq:rp-error}:} Using \eqref{eq:rp-error-proof-1}--\eqref{eq:rp-error-proof-8}, the interpolation inequality \eqref{eq:interpolation}, and the bootstrap assumptions \eqref{eq:boot-energy-1}--\eqref{eq:boot-energy-3}, we obtain an estimate which beats \eqref{eq:rp-error} and hence completes the proof of the proposition.
\end{proof}

\section{Completing the proof of the main theorem}\label{sec:completing}

We are now ready to complete the proof of the main theorem. As detailed arguments for all of the remaining routine steps have appeared in related (in fact much more complicated) contexts elsewhere, such as \cite{Christo09,Luk-local-existence,A16,DHRT22,AKU24}, we will not give detailed proofs in this section. We begin by recording the content of \cref{sec:a-priori,sec:error-est} with the following

\begin{prop}[Master energy hierarchy]\label{prop:master-hierarchy}   For any $A\ge 1$,  $\ve_0$ sufficiently small, $\tau_f\in \mathfrak B$, and $1\le\tau_1\le\tau_2\le\tau_f$, it holds that 
\begin{equation}\label{eq:master-hierarchy}
     \mathcal X^{\tau_f}_{p,k}(\tau_1,\tau_2) \les \underline{\mathcal E}{}^{\tau_f}_{p,k}(\tau_2)+\mathcal E^{\tau_f}_{p,k}(\tau_1)+  \ve^{5/2} \begin{cases}
         \tau_1^{-2+\delta+p}  & \text{if $k=n-2$ and $p\in\{0\}\cup[\delta,2-\delta]$}  \\
         \tau_1^{\max\{-1-\delta+p,-1\}}  &  \text{if $k=n-1$ and $p\in\{0\}\cup[\delta,1+\delta]$}\\
          \tau^{\max\{0,-1+3\delta+p\}}_2 &  \text{if $k=n$ and $p\in\{0\}\cup[\delta,1+\delta]$}\\
        \end{cases}.
\end{equation}
\end{prop}
\begin{proof}
    Let $\mathcal R$ be a spacetime region as depicted in \cref{fig:butterfly} with $\mathrm{I}\cup\mathrm{II}\subset\{\tau=\tau_1\}$. By \cref{prop:T-error,prop:horizon-error,prop:rp-error}, $\mathbb E_{T,k}(\mathcal R)$, $\mathbb E_{Z,k}(\mathcal R)$, $\underline{\mathbb E}{}_{p,k}(\mathcal R)$, and $\mathbb E_{p,k}(\mathcal R)$ decay or grow in a manner consistent with the error term on the right-hand side of \eqref{eq:master-hierarchy}. Therefore, \eqref{eq:master-hierarchy} follows from the a priori energy estimates of \cref{prop:T-energy-est,prop:Morawetz,prop:horizon-hierarchy,prop:rp-hierarchy} and taking the supremum over $\mathcal R\subset\mathcal D^{\tau_f}$.
\end{proof}

We use this hierarchy and the pointwise estimates from \cref{sec:pw} to recover the bootstrap assumptions. 

\begin{proof}[Proof of \cref{prop:improving}] \textsc{Improving the pointwise estimates:} By \cref{prop:pointwise-u,prop:pointwise-v}, there exists a constant $C$ such that 
    \begin{align*}
   |\partial_u\psi^{n-6}|+r|\partial_u\phi^{n-6}|    &\le C\ve,\\
    r^{3/2-\delta}|\partial_v\psi^{n-6}|+ r^2|\partial_v\phi^{n-6}|    &\le C\ve
    \end{align*}
    in $\mathcal D^{\tau_f}\cap \{r\ge\Lambda\}$ and 
        \begin{align*}
     (r-M)^{-3/2+\delta}|\partial_u\phi^{n-6}|   &\le C\ve,\\
       |\partial_v\phi^{n-6}|  &\le C\ve
    \end{align*}
    in $\mathcal D^{\tau_f}\cap \{r\le\Lambda\}$. For $\ve_0$ sufficiently small depending on $A$ and $C$, $C\ve\le \tfrac 12\ve^{1/2}$, which improves the constant in  \eqref{eq:boot-pw-1}--\eqref{eq:boot-pw-5}.

   \textsc{Improving the energy estimates:} By a standard application of the pigeonhole principle over dyadic time intervals (see for instance \cite[Section 7.1]{AKU24}), the master hierarchy of \cref{prop:master-hierarchy} implies that there exists a constant $C$ such that
   \begin{equation*}
     \mathcal X^{\tau_f}_{p,k}(\tau_1,\tau_2) \le C (\ve_0^2+\ve^{5/2})  \begin{cases}
        \tau_1^{-2+\delta+p}  & \text{if $k=n-2$ and $p\in\{0\}\cup[\delta,2-\delta]$}  \\
          \tau_1^{\max\{-1-\delta+p,-1\}}  &  \text{if $k=n-1$ and $p\in\{0\}\cup[\delta,1+\delta]$}\\
           \tau^{\max\{0,-1+3\delta+p\}}_2 &  \text{if $k=n$ and $p\in\{0\}\cup[\delta,1+\delta]$}\\
        \end{cases},
\end{equation*}
where the dependence on $\ve_0$ comes from the initial data. For $A$ sufficiently large depending on $C$ and $\ve_0$ sufficiently small depending on $A$ and $C$, $C (\ve_0^2+\ve^{5/2})\le \frac 12\ve^2$, which improves the constant in \eqref{eq:boot-energy-1}--\eqref{eq:boot-energy-3}.

   \textsc{Improving integrated local energy decay at trapping:} By the bootstrap assumption \eqref{eq:boot-energy-1} and \cref{prop:removing-trapping,prop:trapping-1}, there exists a constant $C$ such that 
   \begin{equation*}
       \mathcal Y^{\tau_f}_{n-3}(\tau_1,\tau_2)\le C\ve^2\tau_1^{-2+\delta}. 
   \end{equation*}
   For $\ve_0$ sufficiently small depending on $A$ and $C$, $C\ve^2\le \frac 12\ve^{3/2}$, which improves the constant in \eqref{eq:boot-bulk-lower-order} and completes the proof of the proposition. 
\end{proof}

As we have now recovered the bootstrap assumptions, the proof of the main theorem now follows from a standard continuity argument.

\begin{proof}[Proof of \cref{thm:main}] Fix $A$ sufficiently large such that for every $\ve_0$ sufficiently small, \cref{lem:nonempty,prop:improving} hold. By \cref{prop:improving}, the extra hypothesis of part (ii) of \cref{lem:nonempty} holds for every $\tau_f\in\mathfrak B(\mathring\phi,\ve_0,A)$. Hence, $\mathfrak B(\mathring\phi,\ve_0,A)=[1,\infty)$ and the solution $\phi$ exists and is smooth on the domain of outer communication $\mathcal D$, and satisfies the desired estimates. Since our estimates prove that in any slab of bounded $\tau$, $\phi$ is bounded in $C^1$ up to the event horizon $\mathcal H^+$ (with respect to a regular coordinate system), a standard propagation of regularity argument implies that $\phi$ extends smoothly to $\mathcal H^+$. 
\end{proof}

\appendix

\section{Sketch of the proof of \texorpdfstring{\cref{thm:dynamical}}{Theorem 1.5}}\label{app:A}

The black hole exteriors $(\mathcal M,g)$ from \cite{AKU24} are covered by eschatologically normalized null coordinates $(u,v)$ which bring the dynamical metric $g$ into the standard double null form 
\[g= -\Omega^2dudv+r^2\mathring{\slashed g}.\]
 Associated to $(u,v)$ is also an extremal Reissner--Nordstr\"om solution $\bar g$ in Eddington--Finkelstein gauge, with area-radius function $\bar r(u,v)$, to which $g$ converges in a suitable sense.

Since $g$ is spherically symmetric, the rotation vector fields $\Gamma_1$, $\Gamma_2$, and $\Gamma_3$ are still Killing, but $T$ will no longer be. However, because of the additional assumption on $\mathcal N$, it will only be necessary to commute with the rotation vector fields. With this in mind, we define all of the energies as in \cref{def:energies}, but with $\bar r$ in place of $r$, as was done in \cite{AKU24}. (The $r^p$ energies can be defined with either $r$ or $\bar r$, but the $(r-M)^{-p}$ energies are sensitive to the difference.) We again make the bootstrap assumptions as in \cref{def:bootstrap}, but now omitting any commutations with $T$. 

The a priori energy estimates now follow the same strategy as in \cref{sec:a-priori}, with important differences. Recall the \emph{Hawking mass} $m \doteq \frac r2(1+4\Omega^{-2}\partial_u r\partial_vr)$  and $\kappa\doteq \partial_vr/(1-\frac{2m}{r})$. Note that $\kappa=1$ in exact extremal Reissner--Nordstr\"om. With these notations, we have $\Omega^2 = - 4\kappa \partial_ur$. If we directly try to carry out the proofs of \cref{sec:a-priori} with $r$ replaced by $\bar r$, the integration by parts arguments cause terms like $\partial_u^2r$ or $\partial_v\kappa$ to appear, which are not estimated in \cite{AKU24}. 

To avoid this, we modify the coordinate vector fields in the multiplier calculations according to 
\begin{equation*}
    \partial_u\mapsto\frac{\partial_u\bar r}{\partial_ur}\partial_u,\quad \partial_v\mapsto \frac 1\kappa\partial_v. 
\end{equation*}
After several tedious calculations, one can then prove versions of \cref{prop:T-energy-est,prop:Morawetz,prop:removing-trapping,prop:horizon-hierarchy,prop:rp-hierarchy}, with additional geometric error terms on the right-hand side, which are however compatible with the bootstrap assumptions. 

The main remaining difference lies in the trapping estimate \cref{prop:trapping-1}. In the present paper, we utilize the commutation with $T$ to handle non-null condition error terms, such as $|\partial_u\phi^{n-6}||\partial_u\phi^k|^2$, in the trapping region. In fact, this is the only use of $T$ commutation in the paper. When such terms satisfy the null condition, such as  $|\partial_u\phi^{n-6}||\partial_u\phi^k||\partial_v\phi^k|$, we can avoid this commutation with $T$ as follows: We first use Cauchy--Schwarz to estimate
\begin{align*}
    \iint_{\mathcal R\cap\mathcal B_{r_{-1},\Lambda}}|\partial_u\phi^{n-6}||\partial_u\phi^k||\partial_v\phi^k|\les \left(\iint_{\mathcal R\cap\mathcal B_{r_{-1},\Lambda}}|\partial_u\phi^k|^2\right)^{1/2} \left(\iint_{\mathcal R\cap\mathcal B_{r_{-1},\Lambda}}|\partial_u\phi^{n-6}|^2|\partial_v\phi^k|^2\right)^{1/2}.
\end{align*}
Importantly, we can derive an $L^2_vL^\infty_{u,\omega}$ estimate for $\partial_u\phi^{n-6}$ on dyadic timescales away from the horizon by following again the proof of \cref{prop:L1-Linfty}, so the second term can be estimated analogously to the estimate for $\mathbb E(\partial_u,\mathrm{ii},k,\mathcal R)$ in the proof of \cref{prop:T-error}. By applying the bootstrap assumptions on dyadic slabs and then summing, we recover \eqref{eq:TT-estimate}. Note the crucial use of the null condition here: $\partial_u\phi^{n-6}$ multiplies $\partial_v\phi^k$, so we only take the supremum of $\partial_u\phi^{n-6}$ in the $v$-direction, and hence can use the wave equation instead of a $T$ commutation after using the fundamental theorem of calculus.

\printbibliography[heading=bibintoc] 
\end{document}